\documentclass[12 pt, hidelinks, leqno]{amsart} 

\usepackage[top=3cm, bottom=3cm, left=2cm, right=2cm]{geometry}
\usepackage[english]{babel}
\usepackage{hyperref}
\usepackage{amsmath}
\usepackage{amsthm}
\usepackage{amssymb}
\usepackage{ucs}
\usepackage{mathtools}
\usepackage{enumerate}
\usepackage[T1]{fontenc}
\usepackage{color}
\usepackage{tikz}
\usepackage{diagbox}
\usepackage{tikz-cd}
\usepackage{tkz-graph}
\usepackage{comment}
\usepackage{relsize,exscale}
\usepackage{mathrsfs}

\usepackage{xcolor}

\AtBeginDocument{%
   \def\MR#1{}
}

\theoremstyle{plain}
\newtheorem{theorem}{Theorem}[section]
\newtheorem{theoremA}{Theorem}

\newtheorem{lemma}[theorem]{Lemma}
\newtheorem{proposition}[theorem]{Proposition}
\newtheorem{corollary}[theorem]{Corollary}
\newtheorem{property}{Property} 
\theoremstyle{definition}
\newtheorem{definition}[theorem]{Definition}

\newtheorem{notation}[theorem]{Notation}
\newtheorem{remark}[theorem]{Remark}

\newcommand{\C}{\mathbb{C}}

\newcommand{\Lcal}{\mathcal{L}}

\newcommand{\ot}{\otimes}
\newcommand{\id}{{\rm id}}

\newcommand{\Pol}{{\rm Pol}}

\newcommand{\Linf}{{\rm L}^\infty}

\newcommand{\Mor}{\mathrm{Mor}}

\newcommand{\Rep}{\mathrm{Rep}}

\newcommand{\GG}{\mathbb{G}}

\newcommand{\NC}{{\rm NC}}
\newcommand{\Ccal}{\mathcal{C}}

\usepackage{tikz}
\usetikzlibrary{calc,decorations.pathreplacing}
\usepackage{caption}

\begin{document}

\title{Representation Category of Free Wreath Product of Classical Groups}
\author{Yigang Qiu}
\address{Yigang Qiu
\newline
Universit\'e Paris Cit\'e and Sorbonne Universit\'e, CNRS, IMJ-PRG, F-75013 Paris, France.}
\email{yqiu@imj-prg.fr}

\date{} 

\maketitle
 \begin{abstract}
In this paper, we construct a rigid concrete $C^*$-tensor category whose
associated compact quantum group, reconstructed via Woronowicz--Tannaka--Krein
duality, is the free wreath product of classical groups.
 \end{abstract}

\section{Introduction} 
\subsection*{Background}
\noindent The free wreath product, introduced by Bichon in \cite{Bi04} to describe the quantum automorphism group of $N$ disjoint copies of a given graph, has been a source of important examples in quantum group theory. It has later been generalized in various directions \cite{pit14,FP16,TW18,FS18,Fr22,FT25}. The more advanced construction from \cite{FT25} produces, given any compact quantum groups $G,H$ and an action $\beta\,:\,H\curvearrowright(B,\psi)$ on a finite dimensional C*-algebra preserving a faithful state $\psi$, a compact quantum group $G\wr_{*,\beta}H$, called the generalized free wreath product and it generalizes the previous constructions. It also produces new interesting examples that do not fit in the previous constructions. One of them is the free wreath product of classical groups: $\Gamma$ is a discrete group and $\Lambda$ is a finite group with action $\Lambda\curvearrowright(C^*(\Lambda),\tau_\Lambda)$ by left translation, where $\tau_\Lambda$ is the canonical trace on $C^*(\Lambda)$. Let us call $\GG:=\widehat{\Gamma}\wr_{*,\beta}\Lambda$ the generalized free wreath product. Most of the results from \cite{FT25}, except the ones about approximation properties, are not applicable to $\GG$, since $\Lambda$ is assumed to be finite.

The compact quantum group
\[
\GG=\widehat{\Gamma}\wr_{\ast,\beta}\Lambda
\]
was already studied in a previous work of Fima and the author \cite{FimaQiu2025FreeWreath}. There we established an explicit combinatorial formula for the Haar state on $\GG$, which is not known in the general setting of generalized free wreath products considered in \cite{FT25}. This formula was then used to analyze the operator-algebraic structure of $\GG$, leading in particular to results on simplicity and uniqueness of trace for $C_r(\GG)$, and on factoriality, fullness, and the center of $\Linf(\GG)$.

For the convenience of the reader, and since it will serve as the starting point of the representation-theoretic considerations below, we briefly recall from \cite{FT25} the presentation of $\GG$ by generators and relations. The algebra $C(\GG)$ is the universal unital C*-algebra generated by elements $\nu_\gamma(g)$ for $\gamma\in\Lambda$ and $g\in \Gamma$ and by $C^*(\Lambda)$ with relations,
\begin{subequations}\label{eq:nu-relations}
\begin{align}
(\nu_\gamma(g))^*&=\nu_{\gamma^{-1}}(g^{-1}),
\label{eq:nu-relations-a}
\\
\nu_\gamma(1)&=\delta_{\gamma,1}\,1,
\label{eq:nu-relations-b}
\\
\nu_\gamma(gh)
&=
\sum_{\substack{r,s\in\Lambda\\ rs=\gamma}}
\nu_r(g)\nu_s(h),
\label{eq:nu-relations-c}
\\
s\,\nu_{rs}(g)&=\nu_{sr}(g)\,s.
\label{eq:nu-relations-d}
\end{align}
\end{subequations}
\noindent The comultiplication on $C(\GG)$ is the unique unital $*$-homomorphism $\Delta\,:\, C(\GG)\rightarrow C(\GG)\ot C(\GG)$ such that:
$$\Delta(\gamma)=\gamma\ot\gamma\quad\text{and}\quad\Delta(\nu_\gamma(g))=\sum_{r,s\in\Lambda,\,rs=\gamma} \nu_r(g)s\ot\nu_s(g)\quad\text{for all }\gamma\in\Lambda,\,g\in\Gamma.$$

\subsection*{Main results}

\begin{theoremA}\label{thm:main}
Let $\mathcal C_{\Gamma,\Lambda}$ be the concrete linear category whose objects
are finite tuples of elements of $\Gamma$, and whose morphism spaces are spanned
by the partition operators associated with admissible bi-coloured noncrossing
partitions in the sense of Definitions~\ref{bi-coloured} and
\ref{DEF: INTERTWINER}.

Then, equipped with the usual composition, tensor product, and involution of
operators, $\mathcal C_{\Gamma,\Lambda}$ is a rigid concrete $C^*$-tensor
category. Moreover, the compact quantum group reconstructed from
$\mathcal C_{\Gamma,\Lambda}$ by Woronowicz's Tannaka--Krein theorem is
canonically isomorphic to $\GG$.
\end{theoremA}

The first assertion is proved in Theorem~\ref{thm:C-Gamma-Lambda-category},
and the reconstruction statement is established in
Theorem~\ref{Reconstruction}.

\subsection*{Organization of the paper} The paper is organized as follows. In Section~2, we introduce the boundary structure of bi-coloured noncrossing partitions and recall the Woronowicz--Tannaka--Krein reconstruction theorem. In Section~3, we introduce a family of fundamental representations of $\GG$ together with the intertwiner operators induced by bi-coloured noncrossing partitions satisfying the required compatibility conditions. In Section~4, we construct the concrete rigid $C^*$-tensor category $\mathcal C_{\Gamma,\Lambda}$. A substantial part of this section is devoted to the definition of the vertical composition of coloured partitions. To this end, we introduce the notion of connected components associated with a pair of partitions $(p,q)\in \NC_{\Lambda}(k,l)\times \NC_{\Lambda}(l,m),$ and study in detail the interplay between their geometric structure and the colouring rules. More precisely, from Definition~4.4 to Corollary~4.28, we analyze the relation between the geometry of these connected components and the admissible colourings. Then, from Definition~4.29 to Proposition~4.32, we define the vertical composition of coloured partitions satisfying the boundary conditions and prove that this operation is well defined. From Corollary~4.33 to Proposition 4.45, we prove that the operators induced by coloured partitions are closed under composition. Finally, from Proposition~4.46 to the end of the section, we complete the construction of the concrete rigid $C^*$-tensor category $\mathcal C_{\Gamma,\Lambda}$ and prove that it is well defined. In the final section, we reconstruct the compact quantum group $\GG=\widehat{\Gamma}\wr_{\ast,\beta}\Lambda$
from the category $\mathcal C_{\Gamma,\Lambda}$ using Woronowicz--Tannaka--Krein duality together with the intertwiner relations induced by several basic partition operators.

 \section{Preliminary}
 
 \subsection{Boundary Condition for Bi-Coloured Noncrossing Partitions}
\begin{definition}\label{blocks}
 For $k,l\geq 0$, we define $\NC(k,l)$ and we view a partition $p\in\NC(k,l)$ as a diagram with $[k]=\{1,\dots,k\}$ upper points, $[l]=\{1,\dots l\}$ lower points and with lines joining the upper points and the lowers points which are in the same block of $p$. Note that since $p$ is non-crossing, no lines from different blocks do cross. When $p\in\NC(k,l)$ and $V\in p$, we write $V_+:=V\cap[k]$ and $V_-:=V\cap[l]$.

Let $B$ be a block of $p$.
\begin{itemize}
\item If $B\cap[k]\neq\varnothing$ and $B\cap[\ell]\neq\varnothing$ (a through-block), define
\[
\mathrm{span}(B)
:= \{\, i\in [k]\mid 1\le i\le \max(B_+)\,\}
\ \sqcup\
\{\, j\in [\ell]\mid 1\le j\le \max(B_-)\,\}.
\]
\item  If $B$ is single-layer (i.e.\ $B\subset[k]$ or $B\subset[\ell]$), define
\[
\mathrm{span}(B):=\{\, i\in\mathbb{N}\mid \min(B)\le i\le \max(B)\,\}.
\]

where $\min(\cdot),\max(\cdot)$ are taken in the natural order of the
corresponding row.
\end{itemize}
A block $B$ is \emph{global-outer} if there is no $C\in p\setminus\{B\}$ with
$\mathrm{span}(B)\subsetneq \mathrm{span}(C)$. The collection of all
global-outer blocks of $p$ is called the \emph{boundary} of $p$, denoted by
$\partial p$.\end{definition}

\begin{definition}
Let $p\in \NC(k,l)$, let $B\in p$ be a single-layer block, and let $C\in p\setminus\{B\}$ be another block. We say that $B$ is \emph{nested in} $C$ if one of the following holds:
\begin{itemize}
    \item $B\subset [k]$, $C_+\neq \varnothing$, and
    \[
    \min(C_+)<\min(B)\le \max(B)<\max(C_+);
    \]
    \item $B\subset [l]$, $C_-\neq \varnothing$, and
    \[
    \min(C_-)<\min(B)\le \max(B)<\max(C_-).
    \]
\end{itemize}
\end{definition}

\begin{definition}
On $\partial p$, define a strict total order $\prec$ as follows.  
For distinct $B,C\in\partial p$, we set $B\prec C$ if and only if one of the following holds:
\begin{itemize}
    \item $B$ and $C$ are upper single-layer blocks, and $\max(B_+)>\max(C_+)$;
    \item $B$ is an upper single-layer block, and $C$ is not;
    \item $C$ is a lower single-layer block, and $B$ is not;
    \item $B$ and $C$ are lower single-layer blocks, and $\max(B_-)<\max(C_-)$.
\end{itemize}
\end{definition}

It is immediate that $\prec$ is a strict total order on $\partial p$: upper single-layer blocks come first, then the through-block if it exists, and finally lower single-layer blocks; within the upper class the order is by decreasing $\max(B_+)$, and within the lower class by increasing $\max(B_-)$.

We write this ordered list as $\partial p=\{V_1\prec\cdots\prec V_r\}$.

\begin{definition}\label{through block order}
Let \(\pi\in \mathrm{NC}(m,n)\), and set
\[
\mathrm{Th}(\pi):=\{\,B\in \pi:\ B\cap[m]\neq\emptyset,\ B\cap[n]\neq\emptyset\,\}.
\]

For \(V,W\in \mathrm{Th}(\pi)\), write \(W\lhd V\) if
\[
\max(W\cap[n])<\min(V\cap[n]).
\]
Equivalently, by non-crossing,
\[
W\lhd V
\quad\Longleftrightarrow\quad
\max(W\cap[m])<\min(V\cap[m]).
\]
Thus \(\lhd\) is the natural left-to-right order on \(\mathrm{Th}(\pi)\).
\end{definition}

\begin{definition}

For $p\in \NC(k,l)$, we regard $p$ as the set of its blocks.
A $\Lambda$-coloring ($\Lambda$-labelling) of $p$ is a map
\[
\mathrm{col}:p\to\Lambda,\qquad B\longmapsto \mathrm{col}(B).
\]
Equivalently, a colored partition is a pair
\[
(p,\mathrm{col})\in \NC(k,l)\times \Lambda^{|p|},
\]
where $\Lambda^{|p|}:=\{\mathrm{col}(V)_{V\in p}:p\to\Lambda\}$ denotes the set of all maps from $p$ to $\Lambda$.

For a coloring $\vec{t}:=(t_V)_{V\in p}=\mathrm{col}(V)_{V\in p}$,
we define the ordered product along the boundary by
\[
\prod_{V\in\partial p}^{\prec} t_V\ :=\ t_{V_1}\cdots t_{V_r}.
\]

We say that a $\Lambda$-colored partition $(p,\mathrm{col})$ satisfies the
\emph{boundary condition} if
\[
\prod_{V\in\partial p}^{\prec} \mathrm{col}(V)=1.
\]

\end{definition}

\begin{definition}\label{boundary product}
\textit{Single-layer outer block.}
If $B\subset [k]$ (upper row), then $B$ is an upper outer block iff there is no single-layer $C\subset [k]$, $C\neq B$, such that
$$
\mathrm{span}(B)\ \subsetneq\ \mathrm{span}(C)\quad
\text{(resp.\ if $B\subset[\ell]$, a lower outer block with $C\subset[\ell]$).}
$$

\textit{Relative outer blocks.}
Let $\mathcal{S}$ be a collection of single-layer blocks of $\pi$ lying on the same row.  
We call a block $B\in\mathcal{S}$ \emph{outer relative to $\mathcal{S}$} if no other
$C\in\mathcal{S}$,
$
\mathrm{span}(B)\subsetneq \mathrm{span}(C).$

(Upper row) If $\mathcal{S}\subset\{B:\,B\subset[k]\}$, define the set of relative outer
blocks of $\mathcal{S}$ by
\[
\partial^{\uparrow}(\mathcal{S})
:=\{\,B\in\mathcal{S}:\ \nexists\,C\in\mathcal{S},\,C\neq B,\ 
\mathrm{span}(B)\subsetneq \mathrm{span}(C)\,\},
\]
and analogously define $\partial^{\downarrow}(\mathcal{S})$ for 
$\mathcal{S}\subset\{B:\,B\subset[\ell]\}$.

\textit{Ordering.}
Let $\mathcal{S}$ be a finite set of single-layer blocks, and let 
$\vec{t}=(t_B)_{B\in\mathcal{S}}\in\Lambda^{\mathcal{S}}$ 
be a family of elements of $\Lambda$ indexed by~$\mathcal{S}$.

If $\mathcal{S}$ consists of upper single-layer blocks, let $
\partial^{\uparrow}(\mathcal{S})=\{B_1\prec\cdots\prec B_m\}$be its relative upper boundary, ordered by the restriction of $\prec$ to
upper single-layer blocks. We define the \emph{upper boundary product} of
$\mathcal{S}$ by
$$
\mathcal{T}^{\uparrow}(\mathcal{S})
\;:=(\;
\prod_{B\in \partial^{\uparrow}(\mathcal{S})}^{\prec} t_B)^{-1}
\;=\;
\bigl(t_{B_1}\cdots t_{B_m}\bigr)^{-1}.
$$

If $\mathcal{S}$ consists of lower single-layer blocks, let $
\partial^{\downarrow}(\mathcal{S})=\{B_1\prec\cdots\prec B_m\}$be its relative lower boundary, ordered by the restriction of $\prec$ to
lower single-layer blocks. We define the \emph{lower boundary product} of
$\mathcal{S}$ by
$$
\mathcal{T}^{\downarrow}(\mathcal{S})
\;:=\;
\prod_{B\in \partial^{\downarrow}(\mathcal{S})}^{\prec} t_B
\;=\;
t_{B_1}\cdots t_{B_m}.
$$
\end{definition}

\begin{definition}
For every non-empty subset $A\subset [n]$, we define
\[
\mathrm{Int}(A):=[\min A,\max A]
=\{\min A,\min A+1,\dots,\max A\}\subset [n].
\]
We call $\mathrm{Int}(A)$ the consecutive interval generated by $A$.

A subset $I\subset [n]$ is called a \emph{consecutive interval} if
\[
I=[s,t]:=\{s,s+1,\dots,t\}
\]
for some $1\le s\le t\le n$.
Equivalently, $I$ is a consecutive interval if and only if $\mathrm{Int}(I)=I$.

For any vector $\vec a=(a_1,\dots,a_n)$ and any non-empty subset $A\subset [n]$,
we set
\[
\prod_A \vec a:=\prod_{\mathrm{Int}(A)}\vec a
:=\prod_{j=\min A}^{\max A} a_j.
\]
In particular, if $I=[s,t]\subset [n]$ is a consecutive interval, then $\prod_I \vec a=a_s a_{s+1}\cdots a_t.$

If $p\in \NC(n)$ and $V\in p$, we say that $V$ is a \emph{consecutive block}
of $p$ if $V$ is a consecutive interval. In that case, the notation $\prod_V \vec a$
has the same meaning.
\end{definition}

\begin{remark}
The notation
\[
\prod_A \vec a
\]
introduced above is \emph{not} the product of the restricted subvector
$\vec a|_A$ in general. Indeed, if
\[
A=\{i_1<\cdots<i_r\}\subset [n],
\]
we define the restricted subvector by
\[
\vec a|_A:=(a_{i_1},\dots,a_{i_r}),
\]
and its product by
\[
\prod(\vec a|_A):=a_{i_1}\cdots a_{i_r}.
\]
By contrast,
\[
\prod_A \vec a=\prod_{j=\min A}^{\max A} a_j
\]
is the product over the whole consecutive interval generated by $A$.

Thus, in general,
\[
\prod_A \vec a\neq \prod(\vec a|_A),
\]
unless $A$ itself is a consecutive interval. For example, if $A=\{2,5\}$, then
\[
\prod_A \vec a=a_2a_3a_4a_5,
\qquad
\prod(\vec a|_A)=a_2a_5.
\]
\end{remark}

\begin{definition}\label{bi-coloured}

\begin{enumerate}
\item
Let $k, l \in \mathbb{N}$. We define
\[
\NC_{\Lambda}(k, l) := \left\{ (p, \vec{t}) \in (\NC(k, l)\times\Lambda^{|p|} )\;\middle|\; 
\vec{t} \in \Lambda^{|p|} \text{ provides a labeling of the blocks of } p,\; 
\prod \vec{t}\vert_{\partial_p} = 1 \right\}.
\]
where the product $\prod_{\partial_p}\vec{t}\vert_{\partial_p}=1$ is defined as the ordered multiplication of the labels of the boundary blocks of $p$ in counterclockwise direction.

\item Let $\vec{g}\in\Gamma^k$ and $\vec{h}\in\Gamma^l$. We define
\[
\NC_{\Gamma}(\vec g,\vec h)
:=
\Bigl\{
p\in \NC(k,l)\;\Bigm|\;
\forall V\in p,\ 
\prod \vec g|_{V_+}=\prod \vec h|_{V_-}
\Bigr\},
\]
where the empty product is understood to be \(1\).

\item Let $\vec g\in \Gamma^k$ and $\vec h\in \Gamma^l$. We define
\[
\NC_{\Lambda}(\vec g,\vec h)
:=
\Bigl\{
(p,\vec t)\in \NC_{\Lambda}(k,l)
\;\Bigm|\;
\forall V\in p,\ 
\prod \vec g|_{V_+}=\prod \vec h|_{V_-}
\Bigr\},
\]
where the empty product is understood to be \(1\).

\end{enumerate}
\end{definition}

\subsection{Woronowicz--Tannaka--Krein reconstruction}
\begin{theorem}[Woronowicz--Tannaka--Krein reconstruction \cite{NT13}]
\label{thm:TK-reconstruction}
Let $\mathcal C$ be a rigid concrete monoidal $C^*$-category, and let
$F:\mathcal C\to \mathrm{Hilb}_f$ be the forgetful functor. Then there exist a compact quantum group
\[
\mathbb H=(C(\mathbb H),\Delta_{\mathbb H})
\]
and a unitary monoidal equivalence
\[
\Phi:\mathcal C\to \Rep(\mathbb H)
\]
such that $F$ is naturally unitarily monoidally isomorphic to the composition of $\Phi$
with the canonical fiber functor $\Rep(\mathbb H)\to \mathrm{Hilb}_f$.
Moreover, the Hopf $*$-algebra $\Pol(\mathbb H)$ is uniquely determined up to Hopf $*$-algebra isomorphism.
\end{theorem}

\begin{remark}
\label{rem:TK-infinite-generation}
If a rigid monoidal $C^*$-category is generated by a finite family of objects, one may replace this family by their direct sum and reduce to the compact matrix quantum group case. In the general case, one works with the full rigid monoidal $C^*$-subcategories generated by finite subsets and passes to the inductive limit of the reconstructed Hopf $*$-algebras.
\end{remark}

\begin{proposition}[Concrete universal form, cf.\ {\cite[Theorem~1.1 and Section~4]{Mal18}}]
\label{prop:WTK-concrete-universal}
Let $\mathcal C$ be a rigid concrete monoidal $C^*$-category, and assume that $\mathcal C$
is generated, as a rigid monoidal $C^*$-category, by a family of objects $(v_i)_{i\in I}$
together with their conjugates. Let
\[
\mathbb H=(C(\mathbb H),\Delta_{\mathbb H})
\]
be the compact quantum group reconstructed from $\mathcal C$ by Theorem~\ref{thm:TK-reconstruction},
and let
\[
\Phi:\mathcal C\to \Rep(\mathbb H)
\]
be the corresponding unitary monoidal equivalence. Set
\[
w_i:=\Phi(v_i)\in \Rep(\mathbb H),\qquad i\in I.
\]
Then $\Rep(\mathbb H)$ is generated, as a rigid monoidal $C^*$-category, by the family
$(w_i)_{i\in I}$ together with their conjugates, and $\Pol(\mathbb H)$ is generated as a
Hopf $*$-algebra by the matrix coefficients of the representations $w_i$, $i\in I$.

Moreover, let $\mathbb G=(C(\mathbb G),\Delta_{\mathbb G})$ be a full compact quantum group
generated by the coefficients of a family of finite-dimensional unitary representations
$(u_i)_{i\in I}$. Assume that for all $k,l\ge 0$ and all
$(i_1,\dots,i_k)\in I^k$, $(j_1,\dots,j_l)\in I^l$, one has
\[
\Mor_{\mathcal C}\bigl(v_{i_1}\otimes\cdots\otimes v_{i_k},
v_{j_1}\otimes\cdots\otimes v_{j_l}\bigr)
\subseteq
\Mor_{\mathbb G}\bigl(u_{i_1}\otimes\cdots\otimes u_{i_k},
u_{j_1}\otimes\cdots\otimes u_{j_l}\bigr).
\]
Then there exists a surjective Hopf $*$-algebra morphism
\[
\pi:\Pol(\mathbb H)\to \Pol(\mathbb G)
\]
such that
\[
(\id\otimes \pi)(w_i)=u_i,\qquad i\in I.
\]
Consequently, by the universal property of the full $C^*$-completion, $\pi$ extends to a
surjective $*$-homomorphism
\[
\varphi:C(\mathbb H)\to C(\mathbb G)
\]
satisfying
\[
(\id\otimes \varphi)(w_i)=u_i,\qquad i\in I.
\]
\end{proposition}
\begin{proof}
By Theorem~\ref{thm:TK-reconstruction}, there exist a compact quantum group $\mathbb H$
and a unitary monoidal equivalence
\[
\Phi:\mathcal C\to \Rep(\mathbb H).
\]
Setting $w_i:=\Phi(v_i)$, it follows that $\Rep(\mathbb H)$ is generated, as a rigid monoidal
$C^*$-category, by the family $(w_i)_{i\in I}$ together with their conjugates. Hence
$\Pol(\mathbb H)$ is generated as a Hopf $*$-algebra by the matrix coefficients of the
representations $w_i$, $i\in I$.

For the second assertion, first assume that the generating family is finite. Replacing it by
a direct sum, one is reduced to the single-object reconstruction of \cite[Theorem~1.1]{Mal18}.
In that construction, the reconstructed Hopf $*$-algebra is obtained as the quotient by the
ideal generated by the prescribed intertwiner relations. Since, by assumption,
\[
\Mor_{\mathcal C}\bigl(v_{i_1}\otimes\cdots\otimes v_{i_k},
v_{j_1}\otimes\cdots\otimes v_{j_l}\bigr)
\subseteq
\Mor_{\mathbb G}\bigl(u_{i_1}\otimes\cdots\otimes u_{i_k},
u_{j_1}\otimes\cdots\otimes u_{j_l}\bigr),
\]
the family $(u_i)_{i\in I}$ satisfies all the intertwiner relations encoded by $\mathcal C$.
Hence there exists a surjective Hopf $*$-algebra morphism
\[
\pi:\Pol(\mathbb H)\to \Pol(\mathbb G)
\]
such that
\[
(\id\otimes \pi)(w_i)=u_i,\qquad i\in I.
\]

For an arbitrary generating family, let $E\subset I$ be finite, and denote by $\mathcal C_E$
the full rigid monoidal $C^*$-subcategory generated by $(v_i)_{i\in E}$ and their conjugates.
Let $\mathbb H_E$ be the compact quantum group reconstructed from $\mathcal C_E$, and let
$\mathbb G_E\subset \mathbb G$ be the compact quantum subgroup generated by the coefficients of
$(u_i)_{i\in E}$. By the finitely generated case, there exists a surjective Hopf $*$-algebra morphism
\[
\pi_E:\Pol(\mathbb H_E)\to \Pol(\mathbb G_E)
\]
such that
\[
(\id\otimes \pi_E)(w_i^{(E)})=u_i,\qquad i\in E.
\]
If $E\subset E'$, the inclusions $\mathcal C_E\subset \mathcal C_{E'}$ induce compatible embeddings
\[
\Pol(\mathbb H_E)\hookrightarrow \Pol(\mathbb H_{E'}),\qquad
\Pol(\mathbb G_E)\hookrightarrow \Pol(\mathbb G_{E'}),
\]
and the morphisms $\pi_E$ are compatible with these embeddings. Therefore they define a morphism on
the inductive limits. By the construction recalled in \cite[Section~4 and Theorem~4.1]{Mal18},
one has $\Pol(\mathbb H)=\varinjlim_E \Pol(\mathbb H_E),$ and similarly, since $\mathbb G$ is generated by the coefficients of the family $(u_i)_{i\in I}$, $\Pol(\mathbb G)=\varinjlim_E \Pol(\mathbb G_E).$ Passing to the inductive limit yields a surjective Hopf $*$-algebra morphism $\pi:\Pol(\mathbb H)\to \Pol(\mathbb G)$
such that $(\id\otimes \pi)(w_i)=u_i, i\in I.$
Finally, by the universal property of the full $C^*$-completion, $\pi$ extends to a surjective
$*$-homomorphism $\varphi:C(\mathbb H)\to C(\mathbb G).$
\end{proof}

\section{Fundamental Representations and Intertwiners from Noncrossing Partitions}

\begin{definition}\label{block relations}
Let $\vec g \in \Gamma^k$, $\vec h \in \Gamma^l$, and
$(p,\vec t) \in \NC_\Lambda(\vec g,\vec h)$.
For $\vec r \in \Lambda^k$ and $\vec s \in \Lambda^l$, define
\[
\delta_{(p,\vec t)} : \Lambda^k \times \Lambda^l \to \{0,1\}
\]
by declaring that $\delta_{(p,\vec t)}(\vec r,\vec s)=1$ if and only if, for
every $V \in p$,
\begin{itemize}
    \item if $V_- \neq \emptyset$ and $V_+ \neq \emptyset$, then
    $\displaystyle \prod_{v=1}^{\max V_-} s_v
    =
    \left(\prod_{u=1}^{\max V_+} r_u\right)t_V$;
    \item if $V_- \neq \emptyset$ and $V_+ = \emptyset$, then
    $\displaystyle \prod_{v=\min V_-}^{\max V_-} s_v = t_V$;
    \item if $V_- = \emptyset$ and $V_+ \neq \emptyset$, then
    $\displaystyle \left(\prod_{u=\min V_+}^{\max V_+} r_u\right)t_V = 1$.
\end{itemize}
All products are taken from left to right.
\end{definition}

\begin{definition}\label{DEF: INTERTWINER} 
To every $(p,\vec{t})\in \NC_{\Lambda}(\vec{g},\vec{h})$, we associate the linear map
$$T_{(p,\vec t)}:(l^2(\Lambda))^{\ot k}\longrightarrow (l^2(\Lambda))^{\ot l},\,\,
T_{(p,\vec t)}(e_{\vec{r}} )=\sum_{\vec {s}\in {\Lambda}^l} \delta_{(p,\vec t)}(\vec{r},\vec {s})e_{\vec{s}},$$
where $(e_s)_{s\in\Lambda}$ is the canonical orthonormal basis of $l^2(\Lambda)$ and $e_{\vec{s}}:=e_{s_1}\ot\dots\ot e_{s_l}$ for $\vec{s}=(s_1,\dots,s_l)\in\Lambda^l$.
\end{definition}

For $g\in\Gamma$ define the element $u(g)\in\Lcal(l^2(\Lambda))\ot C(\GG)$ by $u(g):=\sum_{r,s\in\Lambda}e_{rs}\ot \nu_{rs^{-1}}(g)s$

\begin{proposition}
For all $g\in\Gamma$, $u(g)$ is a unitary representation of $\GG$ and $C(\GG)$ is generated, as a C*-algebra, by the coefficients of $u(g)$, for $g\in\Gamma$.
\end{proposition}

\begin{proof}
Let $g\in\Gamma$. Recall that
\[
u(g)=\sum_{r,s\in\Lambda} e_{rs}\ot \nu_{rs^{-1}}(g)\,s.
\]

First, using the formula for $\Delta$ on the generators together with
\eqref{eq:nu-relations-d}, we get
\begin{eqnarray*}
(\id\ot\Delta)(u(g))
&=&
\sum_{r,s,x,y\in\Lambda,\;xy=rs^{-1}}
e_{rs}\ot \nu_x(g)\,ys\ot \nu_y(g)\,s\\
&=&
\sum_{r,s,t\in\Lambda}
e_{rs}\ot \nu_{rt^{-1}}(g)\,t\ot \nu_{ts^{-1}}(g)\,s\\
&=&
u(g)_{12}u(g)_{13}.
\end{eqnarray*}
Thus $u(g)$ is a representation.

Next, to prove unitarity, we use \eqref{eq:nu-relations-a} to compute the adjoint,
then \eqref{eq:nu-relations-c} with $h=g^{-1}$, and finally
\eqref{eq:nu-relations-b}. This gives
\begin{eqnarray*}
u(g)u(g)^*
&=&
\sum_{r,s\in\Lambda}e_{rs}\ot
\left(\sum_{t\in\Lambda}\nu_{rt^{-1}}(g)\,\nu_{ts^{-1}}(g^{-1})\right)\\
&=&
\sum_{r,s\in\Lambda}e_{rs}\ot \nu_{rs^{-1}}(1)
\qquad\text{by \eqref{eq:nu-relations-c}}\\
&=&
\sum_{r\in\Lambda}e_{rr}\ot 1
\qquad\text{by \eqref{eq:nu-relations-b}}\\
&=&
1.
\end{eqnarray*}
Similarly, using again \eqref{eq:nu-relations-a}, \eqref{eq:nu-relations-c}
and \eqref{eq:nu-relations-b}, one proves that $u(g)^*u(g)=1$.

Finally, it follows from the definition of $C(\GG)$ that the $*$-algebra generated by
the coefficients of $u(g)$, for $g\in\Gamma$, is dense in $C(\GG)$.
\end{proof}

\begin{notation}
Let $n\ge 1$. For
\[
\vec g=(g_1,\dots,g_n)\in \Gamma^n,
\qquad
\vec\gamma=(\gamma_1,\dots,\gamma_n)\in \Lambda^n,
\]
we set
\[
u(\vec g):=u(g_1)\otimes\cdots\otimes u(g_n)
\in \mathcal L\bigl(\ell^2(\Lambda)^{\otimes n}\bigr)\otimes C(\mathbb G),
\]
and
\[
\nu_{\vec\gamma}(\vec g)
:=
\nu_{\gamma_1}(g_1)\cdots \nu_{\gamma_n}(g_n)
\in C(\mathbb G).
\]
\end{notation}

\begin{lemma}\label{consecutive block}
Every non-crossing partition \(p\in \NC(n)\) has a consecutive block.
\end{lemma}

\begin{proof}
Choose a block \(B\in p\) such that \(\max B-\min B\) is minimal among all blocks of \(p\).
We claim that \(B\) is consecutive. Indeed, if not, then there exists
\(x\in [n]\) such that
\[
\min B < x < \max B
\qquad\text{and}\qquad
x\notin B.
\]
Let \(D\) be the block of \(p\) containing \(x\). Since \(p\) is non-crossing, \(D\) cannot contain any point outside the interval
\([\min B,\max B]\), for otherwise \(D\) would cross \(B\). Hence
\[
\min B < \min D \le x \le \max D < \max B.
\]
Therefore
\[
\max D-\min D < \max B-\min B,
\]
contradicting the minimality of \(B\). Thus \(B\) is a consecutive block.
\end{proof}

\begin{lemma}\label{Consecutive criterion}
Let $(p,\vec t)\in \NC_\Lambda(0,n)$, and let
$V=\{a+1,\dots,a+m\}$ be a consecutive block of $p$, with block label $t_V$.
Write $\vec s_-=(s_1,\dots,s_a)\in \Lambda^a, \vec s_+=(s_{a+m+1},\dots,s_n)\in \Lambda^{\,n-a-m}.$
For each $\vec x=(x_{a+1},\dots,x_{a+m})\in \Lambda^m,$ set $\vec s(\vec x):=(s_1,\dots,s_a,x_{a+1},\dots,x_{a+m},s_{a+m+1},\dots,s_n)\in \Lambda^n.$
Then, for fixed $(\vec s_-,\vec s_+)$, the following are equivalent:
\begin{enumerate}
\item there exists $\vec x\in \Lambda^m$ such that $\delta_{(p,\vec t)}(1,\vec s(\vec x))=1$;
\item for every $\vec x\in \Lambda^m$ with $\prod_V \vec x=t_V$, one has $\delta_{(p,\vec t)}(1,\vec s(\vec x))=1$.
\end{enumerate}
\end{lemma}

\begin{proof}
Since $p$ has only lower points, we have
$\delta_{(p,\vec t)}(1,\vec s)=1\Longleftrightarrow
\prod_{i=\min B}^{\max B}s_i=t_B \text{for every }B\in p,$
where all products are taken from left to right.

Fix $\vec s_-$, $\vec s_+$ and $\vec x\in\Lambda^m$, and write
$\vec s(\vec x)=(\vec s_-,\vec x,\vec s_+).$ We claim that, for every block $B\in p$, the relation $\prod_{i=\min B}^{\max B}s(\vec x)_i=t_B$
depends on $\vec x$ only through $\prod_V\vec x$.

Indeed, if $B=V$, then this relation is exactly$\prod_V\vec x=t_V.$
Now assume $B\neq V$. Since $V=\{a+1,\dots,a+m\}$
is a consecutive block of the non-crossing partition $p$, exactly one of the following holds: $\max B<a+1,\min B>a+m,\min B<a+1\le a+m<\max B.$
In the first case, $B\subset\{1,\dots,a\}$, so
$\prod_{i=\min B}^{\max B}s(\vec x)_i$
depends only on $\vec s_-$ and is therefore independent of $\vec x$.
In the second case, $B\subset\{a+m+1,\dots,n\}$, so
$\prod_{i=\min B}^{\max B}s(\vec x)_i$
depends only on $\vec s_+$ and is again independent of $\vec x$.
In the third case, $B$ strictly surrounds $V$, and therefore

$\prod_{i=\min B}^{\max B}s(\vec x)_i
=
\left(\prod_{i=\min B}^{a}s_i\right)
\left(\prod_{j=a+1}^{a+m}x_j\right)
\left(\prod_{i=a+m+1}^{\max B}s_i\right).$
Hence in this case it depends on $\vec x$ only through
$\prod_{j=a+1}^{a+m}x_j=\prod_V\vec x.$

Thus, for every $B\in p$, the block relation $\prod_{i=\min B}^{\max B}s(\vec x)_i=t_B$
depends on $\vec x$ only through $\prod_V\vec x$. Consequently, $\delta_{(p,\vec t)}(1,\vec s(\vec x))$
itself depends on $\vec x$ only through $\prod_V\vec x$.

Now assume that there exists $\vec x\in\Lambda^m$ such that

$\delta_{(p,\vec t)}(1,\vec s(\vec x))=1.$
Then, in particular, the block relation for $V$ must hold, hence $\prod_V\vec x=t_V.$
Therefore, for every $\vec y\in\Lambda^m$ satisfying $\prod_V\vec y=t_V=\prod_V\vec x,$
we obtain $\delta_{(p,\vec t)}(1,\vec s(\vec y))
=\delta_{(p,\vec t)}(1,\vec s(\vec x))=1.$
This proves $(1)\Rightarrow(2)$. The converse is immediate.
\end{proof}

\begin{proposition}

$(p,\vec{t}) \in \NC_{\Lambda}((0,n),{\vec{g}}) $. Then $T_{(p,\vec{t})}\in Mor(\epsilon, u(g_1)\ot u(g_2)\ot\dots\ot u(g_n))$ 

\end{proposition}

\begin{proof}

It suffices to prove that
\[
\bigl((T_{(p,\vec t)}\otimes \id)\epsilon\bigr)(1)
=
\Bigl(\bigl(u(g_1)\otimes\cdots\otimes u(g_n)\bigr)(T_{(p,\vec t)}\otimes \id)\Bigr)(1)
\]
in \((\ell^2(\Lambda))^{\otimes n}\otimes \Pol(\mathbb G)\).

Let \(\{e_{\vec r}\}_{\vec r\in\Lambda^n}\) be the canonical tensor-product basis of
\((\ell^2(\Lambda))^{\otimes n}\), and let
\(\omega_{\vec r}\in ((\ell^2(\Lambda))^{\otimes n})^*\) be the corresponding coordinate functional.
It is therefore enough to prove that, for every \(\vec r\in\Lambda^n\),
\[
(\omega_{\vec r}\otimes \id)\Bigl(\bigl((T_{(p,\vec t)}\otimes \id)\epsilon\bigr)(1)\Bigr)
=
(\omega_{\vec r}\otimes \id)\Bigl(\bigl(u(g_1)\otimes\cdots\otimes u(g_n)\bigr)(T_{(p,\vec t)}\otimes \id)\Bigr)(1).
\]

Fix $\vec{\gamma}\in {\Lambda}^{n}$, by Lemma \ref{consecutive block}, we can take $V=\{a+1,\dots,a+m\}$ a  consecutive block of $p\in NC(\vec{g})$,
 and write $[n]=V_{-}\sqcup V\sqcup V_+$ ,where $V_-:=\{1,\dots,a\}$ and $V_+:=\{a+m+1,\dots,n\}$.
  Using the notation of Lemma~\ref{Consecutive criterion}, for each $\vec{x}\in\Lambda^m$, write
$$
\vec{s}(\vec{x})=(\vec{s}_-,\vec{x},\vec{s}_+)\in \Lambda^n.
$$
For simplicity, here we write $$
\prod \vec{s}_-:=\prod_{i=1}^{a}s_i,\qquad
\prod \vec{x}:=\prod_{i=a+1}^{a+m}x_i,\qquad
\prod \vec{s}_+:=\prod_{i=a+m+1}^{n}s_i.
$$

Fix $\vec\gamma=(\gamma_1,\dots,\gamma_n)\in \Lambda^n$. For every
$\vec\theta=(\theta_1,\dots,\theta_n)\in \Lambda^n$, define
\[
\vec\gamma(\vec\theta):=(d_1,\dots,d_n)\in \Lambda^n
\]
by
\[
d_\iota
=
\Bigl(\prod_{i=1}^{\iota-1}\theta_i\Bigr)\gamma_\iota
\Bigl(\prod_{i=1}^{\iota}\theta_i\Bigr)^{-1},
\qquad 1\le \iota\le n,
\]
with the convention that the empty product is equal to \(e\).

Now, for \(\vec\theta=\vec s(\vec x)=(\vec s_-,\vec x,\vec s_+)\), we write
\[
\vec\gamma(\vec s(\vec x))=(d_1,\dots,d_n),
\]
so that
\[
d_\iota
=
\Bigl(\prod_{i=1}^{\iota-1}(\vec s(\vec x))_i\Bigr)\gamma_\iota
\Bigl(\prod_{i=1}^{\iota}(\vec s(\vec x))_i\Bigr)^{-1},
\qquad 1\le \iota\le n.
\]

 Then
$$
\begin{aligned}
\vec{\gamma}(\vec{s}(\vec{x}))\vert_{V_-}
&:=
(\gamma_1s_1^{-1},\,s_1\gamma_2s_2^{-1}s_1^{-1},\,\dots,\,
(\prod_{i=1}^{a-1}s_i)\gamma_a(\prod \vec{s}_-)^{-1}),\\
\vec{\gamma}(\vec{s}(\vec{x}))\vert_{V_+}
&:=
((\prod \vec{s}_-\prod \vec{x})\gamma_{a+m+1}s_{a+m+1}^{-1}(\prod \vec{s}_-\prod \vec{x})^{-1},\,\dots,\\
&\qquad
(\prod \vec{s}_-\prod \vec{x})(\prod_{i=a+m+1}^{n-1}s_i)\gamma_n
(\prod \vec{s}_-\prod \vec{x}\prod \vec{s}_+)^{-1}).
\end{aligned}
$$
Hence, once $\vec{s}_-$ and $\vec{s}_+$ are fixed, both
$\vec{\gamma}(\vec{s}(\vec{x}))\vert_{V_-}$ and
$\vec{\gamma}(\vec{s}(\vec{x}))\vert_{V_+}$
remain constant for all $\vec{x}\in \Lambda^m$ satisfying
$\prod_V \vec{x}=t_V=\mathrm{col}(V)$.
 
$\vec{\gamma}(\vec{s}(\vec{x}))\vert_{V}
:=((\prod\vec{s}_-)\gamma_{a+1}x_{a+1}^{-1}(\prod \vec{s}_-)^{-1},\,\dots,
(\prod \vec{s}_-)(\prod_{i=a+1}^{a+m-1}x_{i})\gamma_{a+m}(\prod \vec{s}_-\prod\vec{x})^{-1})$

Then, $\prod_V\vec{\gamma}(\vec{s}(\vec{x}))=d_{a+1}\dots d_{a+m}=(\prod \vec{s}_-)(\prod_V\vec{\gamma})(\prod\vec{x})^{-1}(\prod\vec{s}_-)^{-1}$.

If $\vec{s}$  satisfying  $\delta_{(p,\vec{t})}(1,\vec{s}(\vec{x}))=1$, since $\delta_{(p,\vec t)}(1,\vec s(\vec x))=1$, the block relation holds for every outer block of $p$.
Let $V_1,\dots,V_r$ be the outer blocks of $p$, ordered so that
$$
\min V_1<\min V_2<\cdots<\min V_r.
$$
Then
$$
[n]=\bigsqcup_{j=1}^r [\min V_j,\max V_j].
$$ Because of boundary condition, we obtain
$$
\prod_{[n]} \vec{s}(\vec x)=e.
$$
Equivalently,
$$
(\prod \vec s_-)(\prod \vec x)(\prod \vec s_+)=(\prod \vec s_-)t_V(\prod \vec{s}_+)=e,
$$

 $$\prod_V\vec{\gamma}(\vec{s}(\vec{x}))=(\prod \vec{s}_-)(\prod_V\vec{\gamma})t_V^{-1}(\prod\vec{s}_-)^{-1}=(\prod \vec{s}_-)(\prod_V\vec{\gamma})(\prod \vec{s}_+)$$

Denote
$$
I_{(p,\vec t)}
:=
\Bigl\{
(\vec s_-,\vec s_+)\in \Lambda^{n-m}
:\exists\,\vec x\in \Lambda^m \text{ such that }
\delta_{(p,\vec t)}\bigl(1,\vec s(\vec x)\bigr)=1
\Bigr\},
$$
where $\vec s(\vec x):=(\vec s_-,\vec x,\vec s_+)\in \Lambda^n$.
By the previous lemma, equivalently,
$$
I_{(p,\vec t)}
=
\Bigl\{
(\vec s_-,\vec s_+)\in \Lambda^{n-m}
:
\delta_{(p,\vec t)}\bigl(1,\vec s(\vec x)\bigr)=1
\text{ for every }\vec x\in \Lambda^m \text{ with }\prod_V \vec x=t_V
\Bigr\}.
$$

For $(\vec s_-,\vec s_+)\in I_{(p,\vec t)}$, the quantity
$\prod_{V}\vec\gamma(\vec s(\vec x))\vert_V$ is independent of the choice of $\vec x$ with $\prod _V\vec x=t_V$.
Therefore, for each $z\in\Lambda$, we may define
$$
I_{(p,\vec t)}^{z}
:=
\Bigl\{
(\vec s_-,\vec s_+)\in I_{(p,\vec t)}
:\,
(\prod \vec s_-)\,(\prod_V\vec\gamma)\,(\prod \vec s_+)=z
\Bigr\}.
$$
Then
$$
I_{(p,\vec t)}=\bigsqcup_{z\in\Lambda} I_{(p,\vec t)}^{z}.
$$

\medskip
\noindent\textbf{Case 1.}
For the fixed $\vec\gamma\in\Lambda^n$ with
$
\prod_V \vec\gamma|_V\neq t_V,
$

Fix $(\vec s_-,\vec s_+)$. For every $\vec x\in \Lambda^m$ with
$
\prod \vec x=t_V,
$
recall that:
$$
\prod_{V}\vec\gamma(\vec s(\vec x))
=
(\prod \vec s_-)\,(\prod_{V}\vec\gamma)\,(\prod \vec x)^{-1}\,(\prod \vec s_-)^{-1}
=
(\prod \vec s_-)\,(\prod_{V}\vec\gamma)\,t_V^{-1}\,(\prod \vec s_-)^{-1}.
$$

Therefore $\prod_{V}\vec\gamma(\vec s(\vec x))$ does not depend on $\vec x$.

Since
$
\prod_{V}\vec\gamma\neq t_V,
$
it follows that
$
\prod_{V}\vec\gamma(\vec s(\vec x))\neq 1
$
for all $\vec x\in \Lambda^m$ satisfying $\prod_V \vec x=t_V$.

Moreover, for each fixed $(\vec s_-,\vec s_+)\in I_{(p,\vec t)}^{z}$, set
$
z_{(\vec s_-,\vec s_+)}
:=
(\prod \vec s_-)(\prod_V\vec\gamma|_V)t_V^{-1}(\prod \vec s_-)^{-1}\neq 1.
$
Then the map
\begin{eqnarray*}
\Phi_{(\vec s_-,\vec s_+)}:
\{\vec x\in \Lambda^m:\prod_{[m]} \vec x=t_V\}
&\longrightarrow&
\{\vec d\in \Lambda^m:\prod_{[m]} \vec d=z_{(\vec s_-,\vec s_+)}\},\\
\vec x=(x_{a+1},\dots,x_{a+m})
&\longmapsto&
\vec\gamma(\vec s(\vec x))|_V
\end{eqnarray*}
is a bijection.

Indeed, if
$
\vec\gamma(\vec s(\vec x))|_V=(d_{a+1},\dots,d_{a+m}),
$
then the first coordinate gives
\begin{eqnarray*}
d_{a+1}
&=&
(\prod \vec s_-)\gamma_{a+1}x_{a+1}^{-1}(\prod \vec s_-)^{-1},
\end{eqnarray*}
hence
\begin{eqnarray*}
x_{a+1}
&=&
\gamma_{a+1}
\Bigl(
(\prod \vec s_-)^{-1}
d_{a+1}
(\prod \vec s_-)
\Bigr)^{-1}.
\end{eqnarray*}
More generally, once $x_{a+1},\dots,x_{a+j-1}$ are known, the $(a+j)$-th coordinate
satisfies
\begin{eqnarray*}
d_{a+j}
&=&
(\prod \vec s_-)
(x_{a+1}\cdots x_{a+j-1})
\gamma_{a+j}
x_{a+j}^{-1} \\
&&\cdot
(x_{a+1}\cdots x_{a+j-1})^{-1}
(\prod \vec s_-)^{-1},
\end{eqnarray*}
so $x_{a+j}$ is uniquely determined. Hence the map is injective.
Conversely, given
$
(d_{a+1},\dots,d_{a+m})\in \Lambda^m
$
with
$
\prod_{[m]} \vec d=z_{(\vec s_-,\vec s_+)},
$
the same formulas determine successively a unique tuple
$
(x_{a+1},\dots,x_{a+m})\in \Lambda^m.
$
Moreover, the condition
$
\prod_{[m]} \vec d=z_{(\vec s_-,\vec s_+)}
$
implies
$
\prod_{[m]} \vec x=t_V.
$
Therefore the map is surjective, and hence bijective.

 Using the facts established above, we obtain:
 \begin{eqnarray*}
\sum_{\delta_{(p,\vec t)}(1,\vec \theta)=1}\nu_{\vec\gamma(\vec \theta)}(\vec g)
&=&
\sum_{z\in\Lambda\setminus\{e\}}
\sum_{(\vec s_-,\vec s_+)\in I_{(p,\vec t)}^{z}}
\sum_{\prod_{[m]} \vec x=t_V}
\nu_{\vec\gamma(\vec s(\vec x))}(\vec g).
\end{eqnarray*}
Now fix $z\in \Lambda\setminus\{e\}$ and
$(\vec s_-,\vec s_+)\in I_{(p,\vec t)}^{z}$.
Since we have already proved that, once $\vec s_-$ and $\vec s_+$ are fixed,
the restrictions
$
\vec\gamma(\vec s(\vec x))|_{V_-}
$
and
$
\vec\gamma(\vec s(\vec x))|_{V_+}
$
remain constant for all $\vec x\in \Lambda^m$ satisfying
$
\prod_{[m]} \vec x=t_V,
$
we may write
\begin{eqnarray*}
\sum_{\prod_{[m]} \vec x=t_V}
\nu_{\vec\gamma(\vec s(\vec x))}(\vec g)
&=&
\sum_{\prod_{[m]} \vec x=t_V}
\nu_{\vec\gamma(\vec s(\vec x))|_{V_-}}(\vec g_-)\,
\nu_{\vec\gamma(\vec s(\vec x))|_{V}}(\vec g|_{V})\,
\nu_{\vec\gamma(\vec s(\vec x))|_{V_+}}(\vec g_+)\\
&=&
\nu_{\vec\gamma(\vec s(\vec x))|_{V_-}}(\vec g_-)\,
\Bigl(
\sum_{\prod_{[m]} \vec x=t_V}
\nu_{\vec\gamma(\vec s(\vec x))|_{V}}(\vec g|_{V})
\Bigr)\,
\nu_{\vec\gamma(\vec s(\vec x))|_{V_+}}(\vec g_+).
\end{eqnarray*}
Here the first and the last factors are independent of the choice of
$\vec x$ under the condition $\prod_{[m]} \vec x=t_V$.

Moreover, by the bijection
\begin{eqnarray*}
\Bigl\{
\vec x\in \Lambda^m:\ \prod_{[m]}\vec x=t_V
\Bigr\}
&\longrightarrow&
\Bigl\{
\vec d\in \Lambda^m:\ \prod_{[m]} \vec d=z
\Bigr\},\\
\vec x
&\longmapsto&
\vec\gamma(\vec s(\vec x))|_V,
\end{eqnarray*}
we may rewrite the inner sum as
\begin{eqnarray*}
\sum_{\prod_{[m]} \vec x=t_V}
\nu_{\vec\gamma(\vec s(\vec x))|_{V}}(\vec g|_{V})
&=&
\sum_{\prod_{[m]} \vec d=z}\nu_{\vec d}(\vec g|_{V}).
\end{eqnarray*}
Now, by the algebraic relation \ref{eq:nu-relations-c}, we have
\begin{eqnarray*}
\sum_{\prod_{[m]} \vec d=z}\nu_{\vec d}(\vec g_{V})
&=&
\nu_z\Bigl(\prod_{V} \vec g\Bigr).
\end{eqnarray*}
Since $V$ is a block of $p$ and $(p,\vec t)\in \NC_\Lambda((0,n),\vec g)$, the block relation (See Definition \ref{bi-coloured})  gives
\begin{eqnarray*}
\prod_{V} \vec g=1.
\end{eqnarray*}
By the algebraic relation \ref{eq:nu-relations-b}, we obtain:
\begin{eqnarray*}
\sum_{\prod_{[m]} \vec x=t_V}
\nu_{\vec\gamma(\vec s(\vec x))|_{V}}(\vec g|_{V})
&=&
\nu_z(1).
\end{eqnarray*}

Therefore
\begin{eqnarray*}
\sum_{\delta_{(p,\vec t)}(1,\vec \theta)=1}\nu_{\vec\gamma(\vec \theta)}(\vec g)
&=&
\sum_{z\in\Lambda\setminus\{e\}}
\sum_{(\vec s_-,\vec s_+)\in I_{(p,\vec t)}^{z}}
\nu_{\vec\gamma(\vec s(\vec x))|_{V_-}}(\vec g_-)\,
\nu_z(1)\,
\nu_{\vec\gamma(\vec s(\vec x))|_{V_+}}(\vec g_+)\\
&=&0,
\end{eqnarray*}
since $\nu_z(1)=0$ for every $z\neq e$.

\medskip
\noindent\textbf{Case 2.} 
For the fixed $\vec\gamma\in\Lambda^n$ with
$
\prod_V \vec\gamma=t_V,
$
$$\sum_{\delta_{(p,\vec{t})}(1,\vec \theta)=1}\nu_{\vec\gamma(\vec \theta)}(\vec{g})=\sum_{z \in \Lambda}\sum_{(\vec{s}_-,\vec{s}_+)\in I_{(p,\vec{t})}^{z} }\sum_{\prod\vec{x}=t_V }\nu_{\vec{\gamma}(\vec{s}(\vec{x}))\vert_{V_{-}}}(\vec{g}_{-})
\nu_{\vec{\gamma}(\vec{s}(\vec{x}))\vert_{V}}(\vec{g}_{\vert_{V}})\nu_{\vec{\gamma}(\vec{s}(\vec{x}))\vert_{V_{+}}}(\vec{g}_{+})$$

For $(\vec s_-,\vec s_+)\in I_{(p,\vec t)}^{1}$, we have
$
\prod_V \vec\gamma(\vec s(\vec x))=1
$
for every $\vec x\in\Lambda^m$ with $\prod \vec x=t_V$.
Moreover, as shown above, once $(\vec s_-,\vec s_+)$ is fixed, the restrictions
$
\vec\gamma(\vec s(\vec x))|_{V_-}
$
and
$
\vec\gamma(\vec s(\vec x))|_{V_+}
$
do not depend on $\vec x$ under the condition $\prod \vec x=t_V$.
Hence, using the bijection
\[
\Bigl\{
\vec x\in \Lambda^m:\ \prod \vec x=t_V
\Bigr\}
\longrightarrow
\Bigl\{
\vec d\in \Lambda^m:\ \prod \vec d=1
\Bigr\},
\qquad
\vec x\longmapsto \vec\gamma(\vec s(\vec x))|_V,
\]
together with the defining algebraic relation
$
\sum_{\prod \vec d=z}\nu_{\vec d}(\vec h)=\nu_z(\prod \vec h),
$
applied to $z=1$ and $\vec h=\vec g|_V$, we obtain
\begin{eqnarray*}
\sum_{\delta_{(p,\vec{t})}(1,\vec \theta)=1}\nu_{\vec\gamma(\vec \theta)}(\vec{g})
&=&
\sum_{(\vec s_-,\vec s_+)\in I_{(p,\vec t)}^{1}}
\nu_{\vec\gamma(\vec s(\vec x))|_{V_-}}(\vec g_-)\,
\nu_{1}\!\Bigl(\prod \vec g|_V\Bigr)\,
\nu_{\vec\gamma(\vec s(\vec x))|_{V_+}}(\vec g_+)\\
&=&
\sum_{(\vec s_-,\vec s_+)\in I_{(p,\vec t)}^{1}}
\nu_{\vec\gamma(\vec s(\vec x))|_{V_-}}(\vec g_-)\,
\nu_{\vec\gamma(\vec s(\vec x))|_{V_+}}(\vec g_+),
\end{eqnarray*}
where in the second equality we used that $V$ is a block of $p$, hence
$
\prod \vec g|_V=1,
$
so that
$
\nu_1(\prod \vec g|_V)=\nu_1(1)=1.
$

Notice that, since $\delta_{(p,\vec{t})}(1,\vec{s}(\vec{x}))=1$,  $(\vec{\gamma}(\vec{s}(\vec{x}))\vert_{V_{-}} , \vec{\gamma}(\vec{s}(\vec{x}))\vert_{V_{+}})$ is equal to:
\begin{eqnarray*}
&(&d_1,\dots,d_{a-1},(\prod_{i=1}^{a-1}s_i)(\gamma_{a}t_V)(\prod_{i=1}^{a-1}s_i\cdot s_{a}t_{V})^{-1},(\prod_{i=1}^{a-1}s_i\cdot s_{a}t_{V})\gamma_{a+m+1}(\prod_{i=1}^{a-1}s_i\cdot s_{a}t_{V}\cdot s_{a+m+1})^{-1},\\&&(\prod_{i=1}^{a-1}s_i\cdot s_{a}t_{V}\cdot s_{a+m+1})\gamma_{a+m+2}(\prod_{i=1}^{a-1}s_i\cdot s_{a}t_{V}s_{a+m+1}s_{a+m+2})^{-1}
,\dots)
\end{eqnarray*}

It means:$(\vec{\gamma}(\vec{x}(\vec{s}))\vert_{V_{-}} , \vec{\gamma}(\vec{x}(\vec{s}))\vert_{V_{+}})$ can be seen as $\vec{\gamma}^{\prime}(\vec{s}^\prime)\in \Lambda^{n-m}$,  where:$$\vec{\gamma}^{\prime}=(\gamma_1,\dots,\gamma_{a-1},\gamma_{a}t_V,\gamma_{a+m+1},\dots,\gamma_{n})$$  and $$\vec{s}^\prime=(s_1,\dots,s_{a-1},s_{a}t_V,s_{a+m+1},\dots,s_{n})$$

Denote by $W$ the block containing the $a$-th point. We define $(p^{\prime},\vec{t}^{\prime})\in NC_{\Lambda}((0,n-m),{\vec{g}\prime})$ as following:
 
 If $maxW>a+m=maxV$, we simply erase the consecutive block $V$ and keep the labels of the remaining blocks same as in ${(p,\vec{t})}$.
If $maxW=a<a+1=minV$, we remove $V$ and relabel the block $W$ by
$t_{W}' = t_{W} t_V$, while keeping the labels of all other blocks the same as in $(p, \vec{t})$.
Here, $t_{W}$ denotes the label of $W$ in $(p, \vec{t})$. Finally, we define $\vec g'=(g_1,\dots,g_a,g_{a+m+1},\dots,g_n).$

For every block $B\in p$ with $B\neq V$, let $B'$ be its image in $p'$ under the order-preserving bijection
$\{1,\dots,n\}\setminus V\to\{1,\dots,n-m\}$.
Then, by construction of $(p',\vec t')$, the label condition for $B'$ is equivalent to that for $B$, namely
\[
\prod_{i=\min B'}^{\max B'}\gamma_i'=t_B'
\iff
\prod_{i=\min B}^{\max B}\gamma_i=t_B,
\qquad
\prod_{i=\min B'}^{\max B'}s_i'=t_B'
\iff
\prod_{i=\min B}^{\max B}s_i=t_B.
\]

$$ \prod_{i=\min B'}^{\max B'}g_i'=1
\iff
\prod_{i=\min B}^{\max B} g_i=1 $$

It remains to check the boundary condition for \((p',\vec t')\).
Let \(O_1,\dots,O_r\) be the outer blocks of \(p\), ordered from left to right.
Since \((p,\vec t)\in \NC_\Lambda((0,n),\vec g)\), we have $t_{O_1}\cdots t_{O_r}=e.$
First assume that \(\max W>a+m=\max V\). Then \(W\) surrounds \(V\), hence
\(V\) is not an outer block. Removing \(V\) does not change the ordered list
of outer blocks, except for applying the natural order-preserving
identification $\{1,\dots,n\}\setminus V \simeq \{1,\dots,n-m\}.$ Moreover, the labels of all remaining blocks are unchanged. Hence the
ordered product of the outer-block labels is unchanged, and the boundary
condition still holds for \((p',\vec t')\).

Now assume that \(\max W=a<a+1=\min V\). Recall that in this case the image
\(W'\) of \(W\) in \(p'\) is labelled by $t_{W'}'=t_Wt_V,$
while all other remaining blocks keep their original labels.

We distinguish two subcases. If \(V\) is not an outer block of \(p\), then
there exists a block \(U\) such that $\min U<\min V=a+1$ and $
\max U>\max V=a+m.$
Since \(W\) contains the point \(a\) and \(\max W=a\), noncrossingness implies
that \(W\) is contained in the interval \([\min U,\max U]\). Thus \(W\) is not
an outer block either. Consequently, neither \(V\) nor \(W\) contributes to
the ordered product of outer-block labels. After removing \(V\), the images
of the outer blocks of \(p\) are precisely the outer blocks of \(p'\), with
the same labels. Hence the boundary condition is unchanged.

If \(V\) is an outer block of \(p\), then \(W\) is also outer. Indeed, if
\(W\) were contained in a larger block \(U\), then \(U\) would have a point
to the left of \(W\) and a point strictly to the right of \(a\). Since the
points \(a+1,\dots,a+m\) all belong to \(V\), this would force
\(\max U>a+m\), so \(U\) would surround \(V\), contradicting the assumption
that \(V\) is outer. Hence \(W\) is outer. Moreover, \(W\) and \(V\) are  outer blocks. Thus, in the ordered
product of outer-block labels of \(p\), the two adjacent factors \(t_W\) and
\(t_V\) occur as $\cdots\, t_W t_V\, \cdots .$ In \(p'\), the block \(V\) is removed and \(W\) is replaced by \(W'\) with
label \(t_{W'}'=t_Wt_V\). Therefore the ordered product of outer-block labels
in \(p'\) is identical to that in \(p\). Hence it is still equal to \(e\).

In all cases, \((p',\vec t')\) satisfies the boundary condition.

By the definition  and argument above, if $\delta_{(p,\vec{t})}(1,\vec {\gamma})=0$ (resp.$\delta_{(p,\vec{t})}(1,\vec {\gamma})=1$), we know that $\delta_{(p^{\prime},\vec{t}^{\prime})}(1,\vec {\gamma}^{\prime})=0$ (resp.$\delta_{(p^{\prime},\vec{t}^{\prime})}(1,\vec {\gamma}^{\prime})=1$), otherwise, we would have $\delta_{(p,\vec{t})}(1,\vec {\gamma})=1$ (resp.     $\delta_{(p,\vec{t})}(1,\vec {\gamma})=0$ ), which leads to a contradiction. 

It remains to justify that the map
\[
\Phi:I_{(p,\vec t)}^{e}\longrightarrow
\{\vec s'\in \Lambda^{n-m}:\delta_{(p',\vec t')}(1,\vec s')=1\},
\qquad
(\vec s_-,\vec s_+)\longmapsto \vec s'
\]
is bijective, where
\[
\vec s'=(s_1,\dots,s_{a-1},s_at_V,s_{a+m+1},\dots,s_n).
\]
It is well defined, because for $(\vec s_-,\vec s_+)\in I_{(p,\vec t)}^{e}$ there exists
$\vec x\in\Lambda^m$ with $\prod \vec x=t_V$ and $\delta_{(p,\vec t)}(1,\vec s(\vec x))=1$; then, by the block-by-block equivalence proved above, we get
\[
\delta_{(p',\vec t')}(1,\vec s')=1.
\]
The map is injective, since $\vec s'$ determines $(\vec s_-,\vec s_+)$ uniquely:
the first $a-1$ coordinates are unchanged, the $a$-th coordinate is recovered from
$s_a=s_a't_V^{-1}$, and the remaining coordinates are exactly those of $\vec s_+$.
Conversely, if $\vec s'\in\Lambda^{n-m}$ satisfies $\delta_{(p',\vec t')}(1,\vec s')=1$, define
\[
\vec s_-:=(s_1',\dots,s_{a-1}',s_a't_V^{-1}),
\qquad
\vec s_+:=(s_{a+1}',\dots,s_{n-m}').
\]
Then the same equivalence of block conditions argument show that
$\delta_{(p,\vec t)}(1,\vec s(\vec x))=1$ for any $\vec x$ with
$\prod_V \vec x=t_V$, so $(\vec s_-,\vec s_+)\in I_{(p,\vec t)}^{e}$ and
$\Phi(\vec s_-,\vec s_+)=\vec s'$. Hence $\Phi$ is surjective.

Then, we have:
$$\sum_{\delta_{(p,\vec{t})}(1,\vec \theta)=1}\nu_{\vec\gamma(\vec \theta)}(\vec{g})=\sum_{(\vec{s}_-,\vec{s}_+)\in I_{(p,\vec{t})}^e}\nu_{\vec{\gamma}(\vec{x})\vert_{V_{-}}}(\vec{g}_{-})\nu_{\vec{\gamma}(\vec{x})\vert_{V_{+}}}(\vec{g}_{+})=\sum_{\delta_{(p^{\prime},\vec{t}^{\prime})}(1,{\vec {s}}^{\prime})=1}\nu_{\vec{\gamma}^{\prime}(\vec{s}^{\prime})}(\vec{g}^{\prime})$$
Note that in the above formula,  $\vert p^{\prime}\vert=\vert p\vert-1$,

We proceed by induction on $\vert p \vert$.

When $\vert p \vert = 1$,

$ \sum_{\delta_{(p,\vec{t})}(1,\vec{s})=1}\nu_{\vec{\gamma}(\vec{s})}(\vec{g})=
\begin{cases}
1 & \text{if } ,\delta_{(p,\vec{t})}(1,\vec{\gamma})=1\\
0 & \text{if } ,\delta_{(p,\vec{t})}(1,\vec{\gamma})=0
\end{cases}
$ follows directly from:

$\prod(\vec{\gamma}(\vec{s}))=\begin{cases}
1 & \text{if } ,\delta_{(p,\vec{t})}(1,\vec{\gamma})=1\\
\neq 1 & \text{if } ,\delta_{(p,\vec{t})}(1,\vec{\gamma})=0
\end{cases} $, for any $\vec{s}\in \Lambda^{n}$  s.t.  $\delta_{(p,\vec{t})}(1,\vec{s})=1$

Suppose the statement holds for all $\vert (p, \vec{t}) \vert = n-1$. We now consider the case where $\vert p \vert = n$.

If $\delta_{(p,\vec{t})}(1,\vec{\gamma})=1$,  then every consecutive block $V$satisfies $\prod_V\vec{\gamma}=t_V$, fix such a block $V$. By the previous argument and the induction hypothesis at length $n-1$, we obtain $\sum_{\delta_{(p,\vec{t})}(1,\vec \theta)=1}\nu_{\vec\gamma(\vec \theta)}(\vec{g})=\sum_{\delta_{(p^{\prime},\vec{t}^{\prime})}(1,{\vec {s}}^{\prime})=1}\nu_{\vec{\gamma}^{\prime}(\vec{s}^{\prime})}(\vec{g}^{\prime})=1$, where $\delta_{(p^{\prime},\vec{t}^{\prime})}(1,\vec {\gamma}^{\prime})=1 $ and $\vert(p^{\prime},\vec{t}^{\prime})\vert=n-1$. 
 
 If $\delta_{(p,\vec{t})}(1,\vec{\gamma})=0$, Take a consecutive block $V$.  If  $\prod_V\vec{\gamma}\neq t_V$, by the previous argument, we have $\sum_{\delta_{(p,\vec{t})}(1,\vec \theta)=1}\nu_{\vec\gamma(\vec \theta)}(\vec{g})=0$ directly. On the other hand, if $\prod_V\vec{\gamma}=t_V$, by the previous argument and the induction hypothesis at size $n-1$ , we obtain  $\sum_{\delta_{(p,\vec{t})}(1,\vec \theta)=1}\nu_{\vec\gamma(\vec \theta)}(\vec{g})=\sum_{\delta_{(p^{\prime},\vec{t}^{\prime})}(1,{\vec {x}}^{\prime})=1}\nu_{\vec{\gamma}^{\prime}(\vec{x}^{\prime})}(\vec{g}^{\prime})=0$,   where         $\delta_{(p^{\prime},\vec{t}^{\prime})}(1,\vec {\gamma}^{\prime})=0 $ and $\vert(p^{\prime},\vec{t}^{\prime})\vert=n-1$ .

then:
When $\delta_{(p,\vec{t})}(1,\vec {\gamma})=0$:
$$(\omega_{\vec \gamma}\otimes \id))[u(g_1)\ot\dots\ot u(g_n)(T_{(p,\vec{t})}\ot id)]=\sum_{\delta_{(p,\vec{t})}(1,\vec {s})=1}\nu_{\vec{\gamma}(\vec{s})}(\vec{g})=0=(\omega_{\vec \gamma}\otimes \id))[(T_{(p,\vec{t})}\ot id)\epsilon]$$
When $\delta_{(p,\vec{t})}(1,\vec {\gamma})=1$: 
$$ (\omega_{\vec \gamma}\otimes \id))[u(g_1)\ot\dots\ot u(g_n)(T_{(p,\vec{t})}\ot id)]=\sum_{\delta_{(p,\vec{t})}(1,\vec {s})=1}\nu_{\vec{\gamma}(\vec{s})}(\vec{g})=1=(\omega_{\vec \gamma}\otimes \id))[(T_{(p,\vec{t})}\ot id)\epsilon]$$From the above, we deduce that   $T_{(p,\vec{t})}\in Mor(\epsilon,u(g_1)\ot\dots\ot u(g_n))$
\end{proof}

\section{Construction of the Concrete Rigid $C^*$-Tensor Category $\mathcal C_{\Gamma,\Lambda}$}

\subsection{Construction of Vertical Composition}
\begin{definition}[Intuition: vertical concatenation and precomposition]
Given $(p,\vec{t})\in NC_{\Lambda}(k,l)$ and $(q,\vec{f})\in NC_{\Lambda}(l,m)$,
we \emph{visualize} their vertical concatenation by placing $q$ below $p$ and
gluing the lower points of $p$ to the corresponding upper points of $q$.
This produces a three-layer picture with upper, middle and lower rows
labelled by $[k]$, $[l]$ and $[m]$ respectively, in which the blocks of $p$
and of $q$ are drawn as non-crossing strings. We refer to this three-layer
picture as the \emph{precomposition} of $q$ and $p$.
\end{definition}

The previous definition is purely pictorial and meant to provide intuition.
We now give a precise combinatorial notion of connected components in this
precomposition.

 \begin{definition}
Fix $p\in {\rm NC}(k,l)$  and $q\in{\rm NC}(l,m)$ and define an equivalence relation $\mathcal{R}_{p,q}$ on $\{1,\dots,k\}\sqcup\{1,\dots,l\}\sqcup\{1,\dots,m\}$ as follows: 
 $(a,b)\in \mathcal{R}_{p,q}$ if and only if there exists a finite sequence of blocks $B_1,\dots,B_n\in p\sqcup q$ such that:
\begin{itemize}
    \item $a \in B_1$ and $b \in B_n$,
    \item $B_i \cap B_{i+1} \ne \emptyset$ for all $1 \leq i < n$,
\end{itemize}
If $(a,b)\in\mathcal{R}_{p,q}$, a finite sequence $B_1,\dots,B_n$ as above is called a \textit{joining sequence} from $a$ to $b$. A \textbf{$(p,q)$-connected component} is an equivalence class $\Ccal\subseteq[k]\sqcup[l]\sqcup[m]$ of the relation $\mathcal{R}_{p,q}$. We use the notation:
$$\Ccal_p:=\{B\in p\mid B\cap\mathcal{C}\neq\emptyset \}\text{ and }\Ccal_q:=\{B\in q\mid B\cap\mathcal{C}\neq\emptyset \}.$$

We denote by
\[
K(q,p)
:=
\{\,\mathcal C\subseteq [k]\sqcup [l]\sqcup [m]
\mid
\mathcal C \text{ is a }(p,q)\text{-connected component}\,\}
\]
the set of all \((p,q)\)-connected components.

For each \((p,q)\)-connected component \(\mathcal C\in K(q,p)\) with
\(\mathcal C\cap[l]\neq\emptyset\), 

we define $H_{\mathcal C}
:=
[\min(\mathcal C\cap[l]),\,\max(\mathcal C\cap[l])].$

\begin{remark}
By construction, any $(p,q)$-connected component $\mathcal{C}$ is the union
of all blocks of $p$ and $q$ that meet $\mathcal{C}$, that is
$$
  \mathcal{C}
  \;=\;
  \bigcup_{B\in \mathcal{C}_p\sqcup \mathcal{C}_q} B.
$$

\end{remark}

Moreover, these connected components are classified into the following types:

\begin{itemize}
    \item\textbf{upper-half} (resp. \textbf{lower-half}): $\mathcal{C}\cap [k]\neq \emptyset$ , $\mathcal{C}\cap [l]\neq \emptyset$ and $\mathcal{C}\cap [m]=\emptyset$  (resp. $\mathcal{C}\cap [m]\neq \emptyset$ , $\mathcal{C}\cap [l]\neq \emptyset$ and $\mathcal{C}\cap [k]=\emptyset$)
    
    \item\textbf{through}: $\mathcal{C}\cap [m]\neq \emptyset$ , $\mathcal{C}\cap [l]\neq \emptyset$ and $\mathcal{C}\cap [k]\neq\emptyset$
    \item\textbf{cycle}: $\mathcal{C}\cap [k]= \emptyset$ , $\mathcal{C}\cap [l]\neq \emptyset$ and $\mathcal{C}\cap [m]=\emptyset$

    \item\textbf{upper-trivial} (resp. \textbf{lower-trivial}): $\mathcal{C}\cap [k]\neq \emptyset$ , $\mathcal{C}\cap [l]= \emptyset$ and $\mathcal{C}\cap [m]=\emptyset$  (resp.$\mathcal{C}\cap [m]\neq \emptyset$ , $\mathcal{C}\cap [l]=\emptyset$ and $\mathcal{C}\cap [k]=\emptyset$) 
    
\end{itemize}
A $(p,q)$-connected component $\mathcal C$ is called \emph{non-trivial} if it is not upper-trivial or lower-trivial; equivalently, if
$$
\mathcal C\cap [l]\neq\emptyset.
$$

\end{definition}
\begin{remark}
Let $\mathcal C$ be a non-trivial $(p,q)$-connected component. Then every block
$D\in \mathcal C_p\sqcup \mathcal C_q$ meets $[l]$.
Indeed, if $D\cap [l]=\emptyset$, then $D$ cannot intersect any block of the other partition, and it is also disjoint from all other blocks of the same partition. Thus $D$ is isolated and forms a trivial connected component by itself, contradicting the non-triviality of $\mathcal C$. Therefore $\min(D\cap [l])$ and $\max(D\cap [l])$ are well defined for all $D\in \mathcal C_p\sqcup \mathcal C_q$.
\end{remark}
 \begin{definition}\label{def: Entrance}
 Let $\Ccal$ be a non-trivial $(p,q)$ connected component,
$\mathcal{C} \cap [l]=\{c_1<\dots<c_n\}$ is called \textit{underlying set of $\mathcal{C}$ on $[l]$}. A consecutive interval $E=\{c_k + 1,\dots, c_{k+1} - 1\}$ is called an \textit{entrance of $\mathcal{C}$} if it is non-empty and for every $x\in\{c_k + 1,\dots, c_{k+1} - 1\}$ there exists a block $B_x\in p\sqcup q$ such that:
\begin{itemize}
    \item $x \in B_x$, and
    \item $B_x$ isn't \emph{nested in}  $\mathcal{C}$.
\end{itemize}
A  block $B\in p$ (resp.\ $B\in q$) is said to be \emph{nested in} $\Ccal$ if $B\subset [l]$ and 
there exists $D\in\Ccal_p$ (resp.\ $D\in\Ccal_q$)  $
\min \{D\cap[l]\} < \min B \le \max B < \max\{D\cap[l]\}.$
We denote the set of all such blocks by $\mathcal{N}_{\Ccal}(q,p)$.

\end{definition}

\begin{lemma}\label{Entrance}
Let $\Ccal$ be a  non-trivial $(p,q)$ connected component 

$$S_p(\mathcal{C}):=\{\,x\in[\min{\mathcal{C} \cap [l]},\max{\mathcal{C} \cap [l]}]\mid \forall D\in\Ccal_p,\ (x<\min \{D\cap[l]\} \vee x>\max\{D\cap[l]\}\,\}$$
and
$$S_q(\mathcal{C}):=\{\,x\in[\min{\mathcal{C} \cap [l]},\max{\mathcal{C} \cap [l]}]\mid \forall D\in\Ccal_q,\ (x<\min \{D\cap[l]\} \vee x>\max\{D\cap[l]\}\,\}$$.
 Then a consecutive interval $E$ is an entrance of $\Ccal$ if and only if
$E\subseteq S_p \sqcup S_q$.

\begin{proof}
We prove the necessity. Fix $x\in E$.

If the $p$ blcok $B_x$ containing $x$ is of single layer, by Definition \ref{def: Entrance},  there is no block $D\in\Ccal_p$ satisfying
$\min \{D\cap[l]\} < \min B_x \le x \le \max B_x < \max\{D\cap[l]\}$,
i.e.\ $B_x$ is not nested in $\Ccal$.

Assume by contradiction that $x\notin S_p$.
$\exists\,D'\in\Ccal_p$ such that $\min \{D'\cap[l]\} \le x \le \max \{D'\cap[l]\}$.
Since $D'\neq B_x$ and blocks in a partition are disjoint, $x\in B_x$ implies $x\notin D'$.
Hence in fact $\min \{D'\cap[l]\} < x < \max \{D'\cap[l]\}$.

We claim that $B_x$ and $D'$ must cross, contradicting the fact that $p$ is noncrossing.
First, $D'$ cannot be nested in $B_x$. Indeed, if $\min B_x < \min \{D'\cap[l]\}  \le \max \{D'\cap[l]\} < \max B_x$,
then with $a=\min B_x$, $c=x\in B_x$ and $b=\min \{D'\cap[l]\} $, $d=\max \{D'\cap[l]\}\in D'$ , we get
$a<b<c<d$, hence $B_x$ and $D'$ cross. Second, $B_x$ is not nested in $D'$, by Definition \ref{def: Entrance}. Therefore neither block is nested in the other,
while $\mathrm{span}(B_x)\cap \mathrm{span}(D')\neq\varnothing$ intersect since $x\in[\min B_x,\max B_x]\cap(b,d)$.
Consequently either
$\min B_x<b<\max B_x<d$
or
$b<\min B_x<d<\max B_x$,
and in either case $B_x$ and $D'$ cross, a contradiction. Thus $x\in S_p$. This proves $E\subseteq S_p\cup S_q$. 

If the $p$ blcok $B_x$ ($B_x\notin \Ccal_p$) containing $x$ is of  two layer, by non-crossing property of $p$, there doesn't exist any $D\in\Ccal_p$ satisfying $\min \{D\cap[l]\} < \min\{ B_x\cap[l]\} \le x \le \max \{B_x\cap[l]\} < \max\{D\cap[l]\}$. The proof then proceeds exactly as in the case where \(B_x\) is a one-layer block, with \(\min B_x\) and \(\max B_x\) replaced throughout by \(\min(B_x\cap [l])\) and \(\max(B_x\cap [l])\), respectively.

The sufficiency is immediate from the
definition of the entrance.

To conclude, we show that $
 S_p\cap S_q=\emptyset.$ First observe that
$S_p(\Ccal)\cap(\Ccal\cap[l])=\varnothing$ and $S_q(\Ccal)\cap(\Ccal\cap[l])=\varnothing$.
Indeed, if $x\in \Ccal\cap[l]$ and $B_x\in\Ccal_p$ is the $p$-block containing $x$, then
$\min(B_x\cap[l])\le x\le \max(B_x\cap[l])$, so $x\notin S_p(\Ccal)$; the proof for $S_q(\Ccal)$ is identical.

We now prove $S_p(\Ccal)\cap S_q(\Ccal)=\varnothing$.
Suppose $x\in S_p(\Ccal)\cap S_q(\Ccal)$. Then $x\notin \Ccal\cap[l]$, hence
$c_k<x<c_{k+1}$ for some consecutive points $c_k,c_{k+1}\in \Ccal\cap[l]$.
Since $x\in S_p(\Ccal)\cap S_q(\Ccal)$, every block in $\Ccal_p\sqcup\Ccal_q$ lies entirely on one side of $x$, i.e.
for every $D\in \Ccal_p\sqcup\Ccal_q$, either $\max(D\cap[l])<x$ or $x<\min(D\cap[l])$.
Hence no joining sequence can pass from the left of $x$ to the right of $x$.
In particular, there is no joining sequence from $c_k$ to $c_{k+1}$, contradicting that $c_k,c_{k+1}$ belong to the same $(p,q)$-connected component $\Ccal$.
Therefore $S_p(\Ccal)\cap S_q(\Ccal)=\varnothing$.

Together with the already proved inclusion $E\subseteq S_p(\Ccal)\cup S_q(\Ccal)$, we obtain
$E\subseteq S_p(\Ccal)\sqcup S_q(\Ccal)$.
\end{proof}
\end{lemma}

\begin{definition}
Let $E$ be an entrance of the $(p,q)$-connected component $\Ccal$.
If $E\subseteq S_p$ (resp.\ $E\subseteq S_q$), then $E$ is called an \emph{upper} (resp.\ \emph{lower}) entrance of $\Ccal$.
We denote the collection of all upper (resp.\ lower) entrances of $\Ccal$ by
$\mathrm{Ent}^{\uparrow}(\Ccal)$
(resp.\ $\mathrm{Ent}^{\downarrow}(\Ccal)$).
\end{definition}

 \begin{corollary}\label{cor: entrance-one-side-nested}
Let $\mathcal C$ be a non-trivial $(p,q)$-connected component, and let $E$
be an entrance of $\mathcal C$. For $x\in E$, denote by $B_p(x)$
(resp.\ $B_q(x)$) the unique block of $p$ (resp.\ $q$) containing $x$. Then for every $x\in E$, at least one of the two blocks $B_p(x),B_q(x)$ is
nested in $\mathcal C$.

More precisely,
\[
x\in S_p(\Ccal)\setminus S_q(\Ccal) \;\Longrightarrow\; B_q(x)\in \mathcal N_{\mathcal C}(q,p),
\qquad
x\in S_q(\Ccal)\setminus S_p(\Ccal) \;\Longrightarrow\; B_p(x)\in \mathcal N_{\mathcal C}(q,p).
\]
\end{corollary}

\begin{proof}
Fix $x\in E$. Since $E$ is an entrance, Lemma~\ref{Entrance} yields
$E\subseteq S_p(\Ccal)\sqcup S_q(\Ccal)$, 
this implies that exactly one of the two possibilities
$x\in S_p, x\in S_q$
holds.

If $x\in S_p\setminus S_q$, then the $p$-block through $x$ is the non-nested
one, hence the other block through $x$, namely $B_q(x)$, must be nested in
$\mathcal C$. Thus $
B_q(x)\in \mathcal N_{\mathcal C}.$

Similarly, if $x\in S_q\setminus S_p$, then the $q$-block through $x$ is the
non-nested one, hence the other block through $x$, namely $B_p(x)$, must be
nested in $\mathcal C$. Therefore $B_p(x)\in \mathcal N_{\mathcal C}.$
 In either case, at least one of the two blocks containing $x$ is nested in
$\mathcal C$.
\end{proof}

\begin{definition}
Let $\Ccal$ be a non-trivial $(p,q)$-connected component and define
$$X_{u}(\Ccal)=\{p^{\prime} \in \NC(k,l)\mid B\in p^{\prime}, \forall B\in\mathcal{C}_{p}   \} \text { and }X_{d}(\Ccal)=\{q^{\prime} \in \NC(l,m)\mid B\in q^{\prime},\forall B\in\mathcal{C}_{q}   \}.$$

When the connected component $\Ccal$ is clear from the context, we simply write
$X_u$ and $X_d$ instead of $X_u(\Ccal)$ and $X_d(\Ccal)$.

\end{definition}

\begin{lemma}
Let $\Ccal$ be a $(p,q)$-connected component, then it can be viewed as $(p^{\prime}, q^{\prime})$-connected component, for any $p^{\prime}\in X_u=\{b \in NC(k,l)\mid B\in b, \forall B\in\mathcal{C}_{p}   \}$ and $q^{\prime}\in X_d=\{b \in NC(l,m)\mid B\in b,\forall B\in\mathcal{C}_{q}   \}$.

\begin{proof}

Let $\mathcal C'$ be the $(p',q')$-connected component containing $\mathcal C$.
It suffices to show that $\mathcal C'\subseteq \mathcal C$.
Let $y\in\mathcal C'$ and choose $x\in\mathcal C\subseteq\mathcal C'$.
Then $(x,y)\in \mathcal R_{p',q'}$.
If $\{x,y\}\subseteq B$ for some block $B\in p'$ (resp.\ $B\in q'$), then, since there is a unique
$p'$-block (resp.\ $q'$-block) containing $x$ and the unique $p$-block (resp.\ $q$-block) containing $x$
lies in $\mathcal C_p$ (resp.\ $\mathcal C_q$), the assumption $p'\in X_u$ (resp.\ $q'\in X_d$) implies
that $B\in \mathcal C_p\subset p'$ (resp.\ $B\in \mathcal C_q\subset q'$). In particular, $y\in\mathcal C$.
Finally, by induction on the length of a $(p',q')$-chain connecting $x$ to $y$, we conclude that
$y\in\mathcal C$. Hence $\mathcal C'=\mathcal C$.
\end{proof}
\end{lemma}

\begin{corollary}
Let $E$ be an entrance of $\Ccal$ as a $(p,q)$-connected component. Then $E$ is also  an entrance of $\Ccal$ as a $(p^{\prime},q^{\prime})$-connected component, where $p^{\prime} \in X_u$ and $q^{\prime} \in X_d$.

\begin{proof}

By the previous lemma, $\Ccal$ is also a $(p',q')$-connected component, and moreover
its associated families of component-blocks remain $\Ccal_p$ and $\Ccal_q$.
Hence the sets $S_p$ and $S_q$ occurring in the entrance criterion are unchanged.
The conclusion therefore follows from Lemma \ref{Entrance}.
\end{proof}
\end{corollary}

\begin{proposition}\label{lem:two-sided-middle-decomposition}
Let \(\mathcal C\) be a non-trivial \((p,q)\)-connected component with
\[
\mathcal C\cap[l]\neq\varnothing,
\qquad
H_{\mathcal C}:=
[\min(\mathcal C\cap[l]),\max(\mathcal C\cap[l])].
\]

Define
\[
\mathrm{Th}_p(\mathcal C)
:=
\{\,B\in \mathcal C_p : B\cap[k]\neq\varnothing,\ B\cap[l]\neq\varnothing\,\},
\qquad
\mathrm{Out}_p^{\downarrow}(\mathcal C)
:=
\{\,D\in \mathcal C_p : D \text{ is a lower outer block}\,\},
\]
and
\[
\mathrm{Th}_q(\mathcal C)
:=
\{\,B\in \mathcal C_q : B\cap[l]\neq\varnothing,\ B\cap[m]\neq\varnothing\,\},
\qquad
\mathrm{Out}_q^{\uparrow}(\mathcal C)
:=
\{\,D\in \mathcal C_q : D \text{ is an upper outer block}\,\}.
\]

For \(B\in \mathrm{Th}_p(\mathcal C)\cup \mathrm{Th}_q(\mathcal C)\), set $I(B):=[\min(B\cap[l]),\max(B\cap[l])],$
and for \(D\in \mathrm{Out}_p^{\downarrow}(\mathcal C)\cup \mathrm{Out}_q^{\uparrow}(\mathcal C)\), set $I(D):=\operatorname{span}(D).$

Define
\[
\mathrm{FreeOut}_p^{\downarrow}(\mathcal C)
:=
\{\,D\in \mathrm{Out}_p^{\downarrow}(\mathcal C):
I(D)\not\subset I(B)\ \text{for all }B\in \mathrm{Th}_p(\mathcal C)\,\},
\]
\[
\mathrm{FreeOut}_q^{\uparrow}(\mathcal C)
:=
\{\,D\in \mathrm{Out}_q^{\uparrow}(\mathcal C):
I(D)\not\subset I(B)\ \text{for all }B\in \mathrm{Th}_q(\mathcal C)\,\}.
\]

Then
\[
H_{\mathcal C}
=
\left(\bigsqcup_{B\in \mathrm{Th}_p(\mathcal C)} I(B)\right)
\sqcup
\left(\bigsqcup_{D\in \mathrm{FreeOut}_p^{\downarrow}(\mathcal C)} I(D)\right)
\sqcup
\left(\bigsqcup_{E\in \mathrm{Ent}^{\uparrow}(\mathcal C)} E\right),
\]
and
\[
H_{\mathcal C}
=
\left(\bigsqcup_{B\in \mathrm{Th}_q(\mathcal C)} I(B)\right)
\sqcup
\left(\bigsqcup_{D\in \mathrm{FreeOut}_q^{\uparrow}(\mathcal C)} I(D)\right)
\sqcup
\left(\bigsqcup_{E\in \mathrm{Ent}^{\downarrow}(\mathcal C)} E\right).
\]
\end{proposition}

\begin{proof}
We prove the $p$-decomposition. The proof for $q$ is identical after interchanging
$p$ and $q$, lower and upper, and $\mathrm{Ent}^{\uparrow}(\mathcal C)$ and $\mathrm{Ent}^{\downarrow}(\mathcal C)$.

We first check that the three families are pairwise disjoint and that each displayed union is disjoint.

If $B_1,B_2\in  \mathrm{Th}_p^{\uparrow}(\mathcal C)$ are distinct, then the intervals
$I(B_1)$ and $I(B_2)$ are disjoint. Indeed, if they were not, then by noncrossingness one of them
would contain the other, which is impossible for two distinct through-blocks of the same noncrossing partition.

If $D_1,D_2\in \mathrm{FreeOut}_p^{\downarrow}(\mathcal C)$ are distinct, then $I(D_1)$ and $I(D_2)$ are disjoint,
since distinct lower outer blocks are disjoint and cannot be nested.

Distinct entrances are disjoint by definition.

Now let $B\in \mathrm{Th}_p^{\uparrow}(\mathcal C)$ and $D\in \mathrm{FreeOut}_p^{\downarrow}(\mathcal C)$.
Then $I(B)\cap I(D)=\varnothing$, because otherwise noncrossingness would force
either $I(D)\subset I(B)$ or $I(B)\subset I(D)$. The first possibility is excluded by the definition
of free outer blocks, and the second is impossible since $D$ is single-layer while $B$ is through.

Next, if $E\in \mathrm{Ent}^{\uparrow}(\mathcal C)$, then $E\subset S_p$, if $x\in I(B)$ for some $B\in \mathrm{Th}_p^{\uparrow}(\mathcal C)$, then by definition
\[
x\in [\min(B\cap[l]),\max(B\cap[l])],
\]
hence $x\notin S_p$.
Similarly, if $x\in I(D)$ for some $D\in \mathrm{FreeOut}_p^{\downarrow}(\mathcal C)$, then
\[
x\in I(D)=\operatorname{span}(D),
\]
so $x\notin S_p$.
Since every $E\in \mathrm{Ent}^{\uparrow}(\mathcal C)$ is contained in $S_p$, it follows that
\[
E\cap I(B)=\varnothing,\qquad E\cap I(D)=\varnothing.
\]

Therefore entrances are disjoint from both through-block intervals and free outer spans.

It remains to prove that the three families cover $I_{\mathcal C}$.
Fix $x\in I_{\mathcal C}$.

If there exists $B\in \mathrm{Th}_p^{\uparrow}(\mathcal C)$ such that $x\in I(B)$, then $x$ belongs to the first family.

Assume now that
\[
x\notin I(B)\qquad\text{for every }B\in \mathrm{Th}_p^{\uparrow}(\mathcal C).
\]
If there exists a lower single-layer block $D\in \mathcal C_p$ such that $x\in I(D)$, choose an outer lower
single-layer block $D'\in \mathrm{Out}_p^{\downarrow}(\mathcal C)$ with
\[
I(D)\subset I(D').
\]
Then $x\in I(D')$. Since $x$ does not belong to any through-block interval, $D'$ cannot be $p$-covered.
Hence $D'\in \mathrm{FreeOut}_p^{\downarrow}(\mathcal C)$, and $x$ belongs to the second family.

Finally, assume that $x$ belongs neither to a through-block interval nor to the span of any free outer lower block.
We claim that $x\in S_p$.
Otherwise there exists a block $B\in \mathcal C_p$ with $B\cap[l]\neq\varnothing$ such that
\[
x\in [\min(B\cap[l]),\max(B\cap[l])].
\]
If $B$ is a through-block, then $x$ belongs to a through-block interval, contradiction.
If $B$ is a lower single-layer block, choose an outer lower single-layer block
$D\in \mathrm{Out}_p^{\downarrow}(\mathcal C)$ such that
\[
[\min(B\cap[l]),\max(B\cap[l])] \subset \operatorname{span}(D).
\]
If $D$ were $p$-covered, say $\operatorname{span}(D)\subset I(B')$ for some
$B'\in \mathrm{Th}_p^{\uparrow}(\mathcal C)$, then
\[
x\in [\min(B\cap[l]),\max(B\cap[l])] \subset \operatorname{span}(D)\subset I(B'),
\]
so $x$ would belong to a through-block interval, contradiction.
Hence $D$ is not $p$-covered, i.e. $D\in \mathrm{FreeOut}_p^{\downarrow}(\mathcal C)$.
Therefore
\[
x\in [\min(B\cap[l]),\max(B\cap[l])] \subset \operatorname{span}(D),
\]
so $x$ belongs to the span of a free outer lower block, again a contradiction.
Thus $x\in S_p$ and $x$ belongs to the third family.

This proves the $p$-decomposition. The $q$-decomposition is proved in the same way.
\end{proof}

\begin{definition}

We say that $x\notin\mathcal{C}$ is \emph{inner} the entrance $E$ of the connected component $\mathcal{C}$ if there exist $p^{\prime} \in X_{u}$ and $q^{\prime}\in X_{d} $ and a finite sequence $B_i\in p'\sqcup q'$ satisfying:

(a) $x\in B_1$;
(b) $B_i\cap B_{i+1}\neq\emptyset$ for $1\le i<n$;
(c) $B_i\cap E=\emptyset$ for $1\le i\le n-1$;
(d) $B_n\cap E\neq\emptyset$;
(e) $B_i\in\mathcal N_{\mathcal C}(p',q')$ for $1\le i\le n$.

\emph{Remark.} The case $n=1$ is allowed; then condition (c) is vacuous.

We denote the set of all such inner points by $I(E)$.
Note that the definition of $I(E)$ is \emph{probing} $E$ via auxiliary pairs $(p',q')\in X_u\times X_d$. We then return to the original pair $(p,q)$ and define the \emph{inner blocks} of $E$ by
$
I_E^{b}(q,p):=\{\,B\in \mathcal N_{\mathcal C}(q,p)\mid B\cap I(E)\neq\emptyset\,\},
$
namely the blocks of $(p,q)$ that are nested in $\mathcal C$ and contain at least one inner point.

\emph{Remark.} It follows from the Corollary \ref{cor: entrance-one-side-nested} that $E\subset I(E)$.

\end{definition}

\begin{definition}
Let $\Ccal$ be a $(p,q)$-connected component. For $p'\in X_u$ and $q'\in X_d$, set
\begingroup
\setlength{\abovedisplayskip}{.25\baselineskip}%
\setlength{\belowdisplayskip}{.25\baselineskip}%
\setlength{\abovedisplayshortskip}{.15\baselineskip}%
\setlength{\belowdisplayshortskip}{.15\baselineskip}%
\setlength{\jot}{0pt}%
\small
\[
\begin{aligned}
M_E(q',p')&:=\{\,x\in I(E)\mid \exists\,n\ge1,\ \exists\,B_1,\dots,B_n\in\mathcal N_{\Ccal}(q',p') \\
&\hspace{2.2em}\text{s.t. } x\in B_1,\ B_n\cap E\neq\emptyset,\ B_i\cap B_{i+1}\neq\emptyset\ (1\le i<n)\,\},\\
M_E^{b}(q',p')&:=\{\,B\in \mathcal N_{\Ccal}(q',p')\mid B\cap M_E(q',p')\neq\emptyset\,\},\\
\mathcal M_E^{b}(q',p')
&:=\bigsqcup_{X\in\{p',q'\}}\{\,D\in X:\exists\,B\in X\cap M_E^{b}(q',p')\ \text{s.t. } 
\mathrm{span}(D)\subset\mathrm{span}(B)\,\}\ \sqcup\ M_E^{b}(q',p'),\\
\widetilde{\mathcal M}_E^{b}(q,p)
&:=\{\, (D,\mathrm{col}(D)) : D\in \mathcal M_E^{b}(q,p)\,\}.
\end{aligned}
\]
\endgroup
\end{definition}

\begin{remark}
By definition, the entrance $E$ is contained in the complement of $\mathcal C\cap [l]$ inside the middle layer.
Moreover,
\[
M_E(q,p)\cap (\mathcal C\cap [l])=\varnothing.
\]
Indeed, if $x\in M_E(q,p)\cap (\mathcal C\cap [l])$, then by the definition of $M_E(q,p)$ there exist
$n\ge 1$ and blocks $B_1,\dots,B_n\in \mathcal N_{\mathcal C}(q,p)$ such that
\[
x\in B_1,\qquad B_n\cap E\neq\varnothing,\qquad B_i\cap B_{i+1}\neq\varnothing\ \ (1\le i<n).
\]
Hence $x$ is connected to a point of $E$ by a joining sequence. Since $x\in \mathcal C$, this would force
$E$ to meet the same $(p,q)$-connected component $\mathcal C$, contradicting the fact that $E$ is defined in the complement of $\mathcal C$.
\end{remark}
\begin{corollary}

 E is an entrance of $C$(viewed as a $(p',q')$-connected component). Then the inner points set of $E$   is independent with the choice of $p^{\prime}$ and $q^{\prime}$.

\end{corollary}

\begin{figure}[htbp]
\centering
\begin{tikzpicture}[
    x=0.42cm,
    y=0.72cm,
    pt/.style={circle,fill=black,inner sep=1.05pt},
    toplab/.style={font=\scriptsize},
    midlab/.style={font=\tiny},
    botlab/.style={font=\scriptsize},
    rowlab/.style={font=\small},
    gcomp/.style={
        draw=teal!70!black,
        line width=0.9pt,
        line cap=round,
        line join=round,
        rounded corners=1.8pt
    },
    bcomp/.style={
        draw=blue!75!black,
        line width=0.9pt,
        line cap=round,
        line join=round,
        rounded corners=1.8pt
    },
    rcomp/.style={
        draw=red!75!black,
        line width=0.9pt,
        line cap=round,
        line join=round,
        rounded corners=1.8pt
    },
    pcomp/.style={
        draw=violet!75!black,
        line width=0.9pt,
        line cap=round,
        line join=round,
        rounded corners=1.8pt
    }
]

\newcommand{\usingle}[2]{%
  \draw[#1]
    (#2-0.16,0.03)
      .. controls (#2-0.16,0.21) and (#2+0.16,0.21) ..
    (#2+0.16,0.03);
}
\newcommand{\lsingle}[2]{%
  \draw[#1]
    (#2-0.16,-0.03)
      .. controls (#2-0.16,-0.21) and (#2+0.16,-0.21) ..
    (#2+0.16,-0.03);
}

\foreach \i in {1,...,24}{
    \coordinate (M\i) at (\i,0);
}

\coordinate (T1) at (0.30,2.45);
\coordinate (T2) at (0.95,2.45);
\coordinate (T3) at (1.60,2.45);
\coordinate (T4) at (18.45,2.45);
\coordinate (T5) at (19.55,2.45);

\coordinate (B1) at (0.30,-2.45);
\coordinate (B2) at (0.95,-2.45);
\coordinate (B3) at (1.60,-2.45);
\coordinate (B4) at (2.25,-2.45);

\coordinate (BotMid) at ($(B1)!0.5!(B4)$);

\node[rowlab] at (-1.2,  2.45) {$[k]$};
\node[rowlab] at (-1.2,  0.00) {$[l]$};
\node[rowlab] at (-1.2, -2.45) {$[m]$};

\draw[gcomp] (T1) -- ($(T1)+(0,-0.32)$);
\draw[gcomp] (T2) -- ($(T2)+(0,-0.32)$);
\draw[gcomp] (T3) -- ($(T3)+(0,-0.32)$);
\draw[gcomp] ($(T1)+(0,-0.32)$) -- ($(T3)+(0,-0.32)$);
\draw[gcomp] ($(T2)+(0,-0.32)$) -- ($(T2)+(0,-0.88)$)
             -- ($(M1)+(0,1.55)$) -- (M1);

\draw[gcomp] (M3)  -- ($(M3)+(0,1.45)$);
\draw[gcomp] (M7)  -- ($(M7)+(0,1.45)$);
\draw[gcomp] (M12) -- ($(M12)+(0,1.45)$);
\draw[gcomp] ($(M3)+(0,1.45)$) -- ($(M12)+(0,1.45)$);

\draw[gcomp] (M8)  -- ($(M8)+(0,0.95)$);
\draw[gcomp] (M10) -- ($(M10)+(0,0.95)$);
\draw[gcomp] (M11) -- ($(M11)+(0,0.95)$);
\draw[gcomp] ($(M8)+(0,0.95)$) -- ($(M11)+(0,0.95)$);

\draw[gcomp] (M15) -- ($(M15)+(0,1.18)$);
\draw[gcomp] (M18) -- ($(M18)+(0,1.18)$);
\draw[gcomp] ($(M15)+(0,1.18)$) -- ($(M18)+(0,1.18)$);

\draw[gcomp] (M16) -- ($(M16)+(0,0.80)$);
\draw[gcomp] (M17) -- ($(M17)+(0,0.80)$);
\draw[gcomp] ($(M16)+(0,0.80)$) -- ($(M17)+(0,0.80)$);

\draw[bcomp] (M2)  -- ($(M2)+(0,1.82)$);
\draw[bcomp] (M13) -- ($(M13)+(0,1.82)$);
\draw[bcomp] ($(M2)+(0,1.82)$) -- ($(M13)+(0,1.82)$);

\draw[bcomp] (M4) -- ($(M4)+(0,0.86)$);
\draw[bcomp] (M6) -- ($(M6)+(0,0.86)$);
\draw[bcomp] ($(M4)+(0,0.86)$) -- ($(M6)+(0,0.86)$);

\usingle{bcomp}{9}

\coordinate (PtopL) at ($(T4)+(0,-0.32)$);
\coordinate (PtopR) at ($(T5)+(0,-0.32)$);
\coordinate (PtopM) at ($(PtopL)!0.5!(PtopR)$);

\draw[bcomp] (T4) -- (PtopL);
\draw[bcomp] (T5) -- (PtopR);
\draw[bcomp] (PtopL) -- (PtopR);

\coordinate (Pbot14) at ($(M14)+(0,1.45)$);
\coordinate (Pbot19) at ($(M19)+(0,1.45)$);
\coordinate (Pbot24) at ($(M24)+(0,1.45)$);

\draw[bcomp] (M14) -- (Pbot14);
\draw[bcomp] (M19) -- (Pbot19);
\draw[bcomp] (M24) -- (Pbot24);
\draw[bcomp] (Pbot14) -- (Pbot24);
\draw[bcomp] (PtopM) -- (Pbot19);

\draw[bcomp] (M21) -- ($(M21)+(0,0.80)$);
\draw[bcomp] (M22) -- ($(M22)+(0,0.80)$);
\draw[bcomp] ($(M21)+(0,0.80)$) -- ($(M22)+(0,0.80)$);

\usingle{rcomp}{5}

\draw[pcomp] (M20) -- ($(M20)+(0,1.12)$);
\draw[pcomp] (M23) -- ($(M23)+(0,1.12)$);
\draw[pcomp] ($(M20)+(0,1.12)$) -- ($(M23)+(0,1.12)$);

\draw[gcomp] (M1) -- ($(M1)+(0,-1.58)$);
\draw[gcomp] (M7) -- ($(M7)+(0,-1.58)$);
\draw[gcomp] (M8) -- ($(M8)+(0,-1.58)$);
\draw[gcomp] ($(M1)+(0,-1.58)$) -- ($(M8)+(0,-1.58)$);

\draw[gcomp] (B1) -- ($(B1)+(0,0.40)$);
\draw[gcomp] (B4) -- ($(B4)+(0,0.40)$);
\draw[gcomp] ($(B1)+(0,0.40)$) -- ($(B4)+(0,0.40)$);

\draw[gcomp] ($(BotMid)+(0,0.40)$) -- ($(BotMid)+(0,0.82)$)
             -- ($(M7)+(0,-1.58)$);

\lsingle{gcomp}{3}

\draw[gcomp] (M10) -- ($(M10)+(0,-1.40)$);
\draw[gcomp] (M16) -- ($(M16)+(0,-1.40)$);
\draw[gcomp] ($(M10)+(0,-1.40)$) -- ($(M16)+(0,-1.40)$);

\draw[gcomp] (M11) -- ($(M11)+(0,-1.10)$);
\draw[gcomp] (M12) -- ($(M12)+(0,-1.10)$);
\draw[gcomp] (M15) -- ($(M15)+(0,-1.10)$);
\draw[gcomp] ($(M11)+(0,-1.10)$) -- ($(M15)+(0,-1.10)$);

\draw[gcomp] (M17) -- ($(M17)+(0,-0.82)$);
\draw[gcomp] (M18) -- ($(M18)+(0,-0.82)$);
\draw[gcomp] ($(M17)+(0,-0.82)$) -- ($(M18)+(0,-0.82)$);

\draw[bcomp] (B2) -- ($(B2)+(0,0.27)$);
\draw[bcomp] (B3) -- ($(B3)+(0,0.27)$);
\draw[bcomp] ($(B2)+(0,0.27)$) -- ($(B3)+(0,0.27)$);

\draw[bcomp] (M2) -- ($(M2)+(0,-0.95)$);
\draw[bcomp] (M4) -- ($(M4)+(0,-0.95)$);
\draw[bcomp] (M6) -- ($(M6)+(0,-0.95)$);
\draw[bcomp] ($(M2)+(0,-0.95)$) -- ($(M6)+(0,-0.95)$);

\draw[bcomp] (M9)  -- ($(M9)+(0,-1.82)$);
\draw[bcomp] (M21) -- ($(M21)+(0,-1.82)$);
\draw[bcomp] ($(M9)+(0,-1.82)$) -- ($(M21)+(0,-1.82)$);

\draw[bcomp] (M13) -- ($(M13)+(0,-0.76)$);
\draw[bcomp] (M14) -- ($(M14)+(0,-0.76)$);
\draw[bcomp] ($(M13)+(0,-0.76)$) -- ($(M14)+(0,-0.76)$);

\lsingle{bcomp}{19}

\draw[bcomp] (M22) -- ($(M22)+(0,-0.82)$);
\draw[bcomp] (M24) -- ($(M24)+(0,-0.82)$);
\draw[bcomp] ($(M22)+(0,-0.82)$) -- ($(M24)+(0,-0.82)$);

\lsingle{rcomp}{5}
\lsingle{pcomp}{20}
\lsingle{pcomp}{23}

\draw[gcomp,decorate,decoration={brace,mirror,amplitude=3.6pt}]
    (7,-2.00) -- (8,-2.00);
\node[font=\scriptsize] at (7.5,-2.30) {entrance $\{7,8\}$} ;

\foreach \P/\lbl in {
    T1/{1^+},T2/{2^+},T3/{3^+},T4/{4^+},T5/{5^+}
}{
    \node[pt] at (\P) {};
    \node[toplab,above=2pt] at (\P) {$\lbl$};
}

\foreach \i in {1,...,24}{
    \node[pt] at (M\i) {};
    \node[midlab,below=2pt] at (M\i) {\(\i\)};
}

\foreach \P/\lbl in {
    B1/{1^-},B2/{2^-},B3/{3^-},B4/{4^-}
}{
    \node[pt] at (\P) {};
    \node[botlab,below=2pt] at (\P) {$\lbl$};
}

\end{tikzpicture}
\captionof{figure}{ Bi-layer non-crossing partition  partitions \(p\) and \(q\). \iffalse The blue connected component has entrance \(\{7,8\}\),\fi The singleton blocks are represented by small arcs.}
\label{fig:bilayer-nc-example}
\end{figure}

\smallskip

\noindent\emph{Comment.}
In this example, there are in fact five \((p,q)\)-connected components:
\[
\mathcal C_{\mathrm g}
=
\{1^+,2^+,3^+,1,3,7,8,10,11,12,15,16,17,18,1^-,4^-\},
\]
\[
\mathcal C_{\mathrm b}
=
\{4^+,5^+,2,4,6,9,13,14,19,21,22,24\},
\]
\[
\mathcal C_{\mathrm r}
=
\{5\},
\qquad
\mathcal C_{\mathrm v}
=
\{20,23\},
\qquad
\mathcal C_{0}
=
\{2^-,3^-\}.
\]

For the green connected component \(\mathcal C_{\mathrm g}\), there are two upper entrances:
$\{2\}$, $\{13,14\}$,
and one lower entrance $\{9\}$.
\
Their corresponding inner-point sets are
\[
I_{\{2\}}=\{2,4,5,6\},
\qquad
I_{\{9\}}=\{9\},
\qquad
I_{\{13,14\}}=\{13,14\},
\]
and
\[
M_{\{2\}}(q,p)=\{2,4,6\},
\qquad
M_{\{9\}}(q,p)=\{9\},
\qquad
M_{\{13,14\}}(q,p)=\{13,14\}.
\]

For the blue connected component \(\mathcal C_{\mathrm b}\), the displayed entrance is
\[
E=\{7,8\}.
\]
The corresponding inner-point set is
\[
I_E
=
\{3,7,8,10,11,12,15,16,17,18,20,23\},
\]
and
\[
M_E(q,p)
=
\{3,7,8,10,11,12,15,16,17,18,20,23\}.
\].

 For simplicity, we will abuse notation and write $p \in \NC_{\Lambda}(k,l)$ for $(p, \vec{t}) \in \NC_{\Lambda}(k,l)$ and $q \in \NC_{\Lambda}(l,m)$ for $(q, \vec{f}) \in NC_{\Lambda}(l,m)$ in what follows.
 By the definitions of \( T_q \) and \( T_p \), the condition \( T_q \circ T_p \neq 0 \) implies that there exist vectors \( \vec{r} \in \Gamma^k \), \( \vec{d} \in \Gamma^m \), and \( \vec{s} \in \Gamma^l \) such that $\delta_q(\vec{s}, \vec{d}) \cdot \delta_p(\vec{r}, \vec{s}) \neq 0.$  
This is equivalent to the condition that $
\theta_p^{\vec{r}} \cap \Omega_q^{\vec{d}} \neq \emptyset,$ 
where for each \( \vec{r} \in \Gamma^k \), the set \( \theta_p^{\vec{r}} \) and  $\Omega_q^{\vec{d}}$ is defined as  
$\theta_p^{\vec{r}} = \left\{ \vec{s} \in \Lambda^l \mid \delta_p(\vec{r}, \vec{s})=1 \right\}$, $\Omega_q^{\vec{d}} = \left\{ \vec{s} \in \Lambda^l \mid \delta_q(\vec{s}, \vec{d})=1 \right\}$

\begin{definition}

\textbf{1)From $(p,q)$ to a bipartite multigraph.}

List two block maps 
$$
\pi_p:[k]\sqcup[l]\to p,\qquad \pi_q:[l]\sqcup[m]\to q.
$$
Define the bipartite multigraph
$$
\mathcal{G}=(\mathcal{V},\mathcal{E}),\mathcal{V}:=[p\cap M^{b}_E(q,p)] \sqcup[q\cap M^{b}_E(q,p)],         \mathcal{E}:=\{e_i:\ i\in M_E(q,p)\},\partial(e_i)=\{\pi_p(i),\pi_q(i)\}.
$$$$
\mathcal{G}=(\mathcal{V},\mathcal{E}),\mathcal{V}:=[p\cap M^{b}_E(q,p)] \sqcup[q\cap M^{b}_E(q,p)],\qquad
\mathcal{E}:=\{e_i:\ i\in M_E(q,p)\},\qquad
\partial(e_i)=\{\pi_p(i),\pi_q(i)\}.
$$
Here $\mathcal V$ is the vertex set and $\mathcal E$ is the edge set of the bipartite multigraph $\mathcal G$. The vertices on the $p$-side are the blocks in $p\cap M_E^b(q,p)$, while those on the $q$-side are the blocks in $q\cap M_E^b(q,p)$. For each $i\in M_E(q,p)$, the symbol $e_i$ denotes the edge corresponding to $i$, and $\partial(e_i)$ is its set of endpoints, namely the two vertices $\pi_p(i)$ and $\pi_q(i)$.

\textbf{2)Cycles s in $\mathcal{G}$ and their length.}

A \emph{cycle} is a sequence
\[
\mathbf{i}=(i_1,\dots,i_{2t})\in M_E(q,p)^{2t}
\]
with pairwise distinct entries, where \(t\ge2\), such that:

(a) Alternating shared endpoint:
set $i_{2t+1}:=i_1$. For every $k=1,\dots,t$, we have
\[
\pi_p(i_{2k-1})=\pi_p(i_{2k})\in p\cap M_E^b(q,p),
\qquad
\pi_q(i_{2k})=\pi_q(i_{2k+1})\in q\cap M_E^b(q,p).
\]

(b) Side-wise vertex distinctness:
the blocks
\[
\pi_p(i_1)=\pi_p(i_2),\ \pi_p(i_3)=\pi_p(i_4),\ \dots,\ \pi_p(i_{2t-1})=\pi_p(i_{2t})
\]
are pairwise distinct in $p\cap M_E^b(q,p)$, and the blocks
\[
\pi_q(i_2)=\pi_q(i_3),\ \pi_q(i_4)=\pi_q(i_5),\ \dots,\ \pi_q(i_{2t})=\pi_q(i_1)
\]
are pairwise distinct in $q\cap M_E^b(q,p)$.

Equivalently,
\[
\pi_p(i_1)\xleftrightarrow{i_1}\pi_q(i_2)
\xleftrightarrow{i_2}\pi_p(i_3)
\xleftrightarrow{i_3}\pi_q(i_4)
\;\cdots\;
\xleftrightarrow{i_{2t}}\pi_p(i_1)
\]
is a closed alternating loop. Its length is $\ell(\mathbf{i})=2t$.

\textbf{3) The nested structure of $p$-blocks \iffalse along a loop\fi}

\emph{P-side blocks used by the cycle.}
Given the cycle   $\mathbf{i}=(i_1,\dots,i_{2t})$, define
$$
B_p(\mathbf{i})\ :=\ \big\{\, \pi_{p}(i_{2k-1})=\pi_{p}(i_{2k})\ :\ k=1,\dots,t \,\big\}\ \subseteq p\cap M_E^{b}(q,p)\subseteq [l].
$$

\emph{Label interval of a $p$-block.}
For each $B\in B_p(\mathbf{i})$, let
$$
J(B)\ :=\ \big[\,\min\{\,i_{2k-1},\,i_{2k}\,\},\ \max\{\,i_{2k-1},\,i_{2k}\,\}\big]
\quad\text{(where }\pi_p(i_{2k-1})=\pi_p(i_{2k})=B\text{)}.
$$
\emph{Nested ordered by interval inclusion.}
For $B,C\in B_p(\mathbf{i})$, write
$$
C\ \preceq_{\mathbf{i}}\ B\ \iff\ J(C)\subseteq J(B),
$$
and write $C\ \prec_{\mathbf{i}}\ B\ \iff C\ \preceq_{\mathbf{i}}\ B$ and $C\neq B$.

\medskip

\noindent\emph{Cover relation and children.}
For \(B,C\in B_p(\mathbf{i})\), we say that $B$ \emph{covers} $C$ (relative to $\mathbf{i}$) if
$$
C\ \triangleleft_{\mathbf{i}}\ B\ \iff\ C\prec_{\mathbf{i}} B\ \text{ and there is no }D\in B_P(\mathbf{i})\text{ with }C\prec_{\mathbf{i}} D\prec_{\mathbf{i}} B.
$$
Define the children set $\mathrm{Ch}_p(B)$ of $B$ by
$\mathrm{Ch}_{p,\mathbf{i}}(B)\ :=\ \{\, C\in B_p(\mathbf{i})\ :\ C\ \triangleleft_{\mathbf{i}}\ B \,\}$
Since $p$ is noncrossing,  any two intervals in the family $\{\,J(B): B\in B_p(\mathbf{i})\,\}$ are either disjoint or one contains the other. Consequently, for each $B\in B_p(\mathbf{i})$, the set of its strict upper containers is totally ordered by $\preceq_{\mathbf{i}}$.

\emph{Ancestor set .}
For $B\in B_p(\mathbf{i})$, define  $
\mathrm{Anc}_{p,\mathbf{i}}(B)\ :=\ \{\, C\in B_p(\mathbf{i})\ :\ B\ \prec_{\mathbf{i}}\ C \,\}.$

\emph{q-side analogue.} Everything stated for the $P$-side also works the same on the $q$-side: replace $p$ by $q$, $\pi_p$ by $\pi_q$, $B_p(\mathbf{i})$ by $B_q(\mathbf{i})=\{\pi_q(i_{2k})=\pi_q(i_{2k+1}):\,k=1,\dots,t\}$, and $J(\cdot)$ by $J_q(\cdot)$ (built from the pairs $(i_{2k},i_{2k+1})$); make the same replacements for the relations $\prec_{\mathbf{i}},\ \triangleleft_{\mathbf{i}}$ and the operators $\mathrm{Ch},\ \mathrm{Anc}$.

\textbf{4)Interior/exterior labels induced by the cycle.}

Label set of the cycle:

$\mathrm{Lab}(\mathbf{i})=\{i_1,\dots,i_{2t}\}\subset M_E(q,p)\subset[l],
y_1<y_2<\cdots<y_{2t}\ \text{is the increasing rearrangement of }\mathrm{Lab}(\mathbf{i}).$

Right-crossing count for $x\in[l]$:
$
\rho_{\mathbf{i}}(x)\ :=\ \bigl|\{\, y\in \mathrm{Lab}(\mathbf{i}):\ y>x\,\}\bigr|.
$

Interior / exterior:
$$
\mathrm{Int}(\mathbf{i})=\{\,x\in [l]-\mathrm{Lab}(\mathbf{i}):\ \rho_{\mathbf{i}}(x)\ \text{is odd}\,\},\qquad
\mathrm{Ext}(\mathbf{i})=\{\,x\in[l]-\mathrm{Lab}(\mathbf{i}):\ \rho_{\mathbf{i}}(x)\ \text{is even}\,\},
$$

\end{definition}

\begin{figure}[htbp]
\centering
\begin{tikzpicture}[
    x=0.42cm,
    y=0.72cm,
    >=Latex,
    pt/.style={circle,fill=black,inner sep=1.05pt},
    toplab/.style={font=\scriptsize},
    midlab/.style={font=\tiny},
    botlab/.style={font=\scriptsize},
    rowlab/.style={font=\small},
    cycarrow/.style={draw=red!80!black,line width=0.9pt,->},
    gcomp/.style={
        draw=teal!70!black,
        line width=0.9pt,
        line cap=round,
        line join=round,
        rounded corners=1.8pt
    },
    bcomp/.style={
        draw=blue!75!black,
        line width=0.9pt,
        line cap=round,
        line join=round,
        rounded corners=1.8pt
    },
    rcomp/.style={
        draw=red!75!black,
        line width=0.9pt,
        line cap=round,
        line join=round,
        rounded corners=1.8pt
    },
    pcomp/.style={
        draw=violet!75!black,
        line width=0.9pt,
        line cap=round,
        line join=round,
        rounded corners=1.8pt
    }
]

\newcommand{\usingle}[2]{%
  \draw[#1]
    (#2-0.16,0.03)
      .. controls (#2-0.16,0.21) and (#2+0.16,0.21) ..
    (#2+0.16,0.03);
}
\newcommand{\lsingle}[2]{%
  \draw[#1]
    (#2-0.16,-0.03)
      .. controls (#2-0.16,-0.21) and (#2+0.16,-0.21) ..
    (#2+0.16,-0.03);
}

\foreach \i in {1,...,24}{
    \coordinate (M\i) at (\i,0);
}

\coordinate (T1) at (0.30,2.45);
\coordinate (T2) at (0.95,2.45);
\coordinate (T3) at (1.60,2.45);
\coordinate (T4) at (18.45,2.45);
\coordinate (T5) at (19.55,2.45);

\coordinate (B1) at (0.30,-2.45);
\coordinate (B2) at (0.95,-2.45);
\coordinate (B3) at (1.60,-2.45);
\coordinate (B4) at (2.25,-2.45);

\coordinate (BotMid) at ($(B1)!0.5!(B4)$);

\node[rowlab] at (-1.2,  2.45) {$[k]$};
\node[rowlab] at (-1.2,  0.00) {$[l]$};
\node[rowlab] at (-1.2, -2.45) {$[m]$};

\draw[gcomp] (T1) -- ($(T1)+(0,-0.32)$);
\draw[gcomp] (T2) -- ($(T2)+(0,-0.32)$);
\draw[gcomp] (T3) -- ($(T3)+(0,-0.32)$);
\draw[gcomp] ($(T1)+(0,-0.32)$) -- ($(T3)+(0,-0.32)$);
\draw[gcomp] ($(T2)+(0,-0.32)$) -- ($(T2)+(0,-0.88)$)
             -- ($(M1)+(0,1.55)$) -- (M1);

\draw[gcomp] (M3)  -- ($(M3)+(0,1.45)$);
\draw[gcomp] (M7)  -- ($(M7)+(0,1.45)$);
\draw[gcomp] (M12) -- ($(M12)+(0,1.45)$);
\draw[gcomp] ($(M3)+(0,1.45)$) -- ($(M12)+(0,1.45)$);

\draw[gcomp] (M8)  -- ($(M8)+(0,0.95)$);
\draw[gcomp] (M10) -- ($(M10)+(0,0.95)$);
\draw[gcomp] (M11) -- ($(M11)+(0,0.95)$);
\draw[gcomp] ($(M8)+(0,0.95)$) -- ($(M11)+(0,0.95)$);

\draw[gcomp] (M15) -- ($(M15)+(0,1.18)$);
\draw[gcomp] (M18) -- ($(M18)+(0,1.18)$);
\draw[gcomp] ($(M15)+(0,1.18)$) -- ($(M18)+(0,1.18)$);

\draw[gcomp] (M16) -- ($(M16)+(0,0.80)$);
\draw[gcomp] (M17) -- ($(M17)+(0,0.80)$);
\draw[gcomp] ($(M16)+(0,0.80)$) -- ($(M17)+(0,0.80)$);

\draw[bcomp] (M2)  -- ($(M2)+(0,1.82)$);
\draw[bcomp] (M13) -- ($(M13)+(0,1.82)$);
\draw[bcomp] ($(M2)+(0,1.82)$) -- ($(M13)+(0,1.82)$);

\draw[bcomp] (M4) -- ($(M4)+(0,0.86)$);
\draw[bcomp] (M6) -- ($(M6)+(0,0.86)$);
\draw[bcomp] ($(M4)+(0,0.86)$) -- ($(M6)+(0,0.86)$);

\usingle{bcomp}{9}

\coordinate (PtopL) at ($(T4)+(0,-0.32)$);
\coordinate (PtopR) at ($(T5)+(0,-0.32)$);
\coordinate (PtopM) at ($(PtopL)!0.5!(PtopR)$);

\draw[bcomp] (T4) -- (PtopL);
\draw[bcomp] (T5) -- (PtopR);
\draw[bcomp] (PtopL) -- (PtopR);

\coordinate (Pbot14) at ($(M14)+(0,1.45)$);
\coordinate (Pbot19) at ($(M19)+(0,1.45)$);
\coordinate (Pbot24) at ($(M24)+(0,1.45)$);

\draw[bcomp] (M14) -- (Pbot14);
\draw[bcomp] (M19) -- (Pbot19);
\draw[bcomp] (M24) -- (Pbot24);
\draw[bcomp] (Pbot14) -- (Pbot24);
\draw[bcomp] (PtopM) -- (Pbot19);

\draw[bcomp] (M21) -- ($(M21)+(0,0.80)$);
\draw[bcomp] (M22) -- ($(M22)+(0,0.80)$);
\draw[bcomp] ($(M21)+(0,0.80)$) -- ($(M22)+(0,0.80)$);

\usingle{rcomp}{5}

\draw[pcomp] (M20) -- ($(M20)+(0,1.12)$);
\draw[pcomp] (M23) -- ($(M23)+(0,1.12)$);
\draw[pcomp] ($(M20)+(0,1.12)$) -- ($(M23)+(0,1.12)$);

\draw[gcomp] (M1) -- ($(M1)+(0,-1.58)$);
\draw[gcomp] (M7) -- ($(M7)+(0,-1.58)$);
\draw[gcomp] (M8) -- ($(M8)+(0,-1.58)$);
\draw[gcomp] ($(M1)+(0,-1.58)$) -- ($(M8)+(0,-1.58)$);

\draw[gcomp] (B1) -- ($(B1)+(0,0.40)$);
\draw[gcomp] (B4) -- ($(B4)+(0,0.40)$);
\draw[gcomp] ($(B1)+(0,0.40)$) -- ($(B4)+(0,0.40)$);

\draw[gcomp] ($(BotMid)+(0,0.40)$) -- ($(BotMid)+(0,0.82)$)
             -- ($(M7)+(0,-1.58)$);

\lsingle{gcomp}{3}

\draw[gcomp] (M10) -- ($(M10)+(0,-1.40)$);
\draw[gcomp] (M16) -- ($(M16)+(0,-1.40)$);
\draw[gcomp] ($(M10)+(0,-1.40)$) -- ($(M16)+(0,-1.40)$);

\draw[gcomp] (M11) -- ($(M11)+(0,-1.10)$);
\draw[gcomp] (M12) -- ($(M12)+(0,-1.10)$);
\draw[gcomp] (M15) -- ($(M15)+(0,-1.10)$);
\draw[gcomp] ($(M11)+(0,-1.10)$) -- ($(M15)+(0,-1.10)$);

\draw[gcomp] (M17) -- ($(M17)+(0,-0.82)$);
\draw[gcomp] (M18) -- ($(M18)+(0,-0.82)$);
\draw[gcomp] ($(M17)+(0,-0.82)$) -- ($(M18)+(0,-0.82)$);

\draw[bcomp] (B2) -- ($(B2)+(0,0.27)$);
\draw[bcomp] (B3) -- ($(B3)+(0,0.27)$);
\draw[bcomp] ($(B2)+(0,0.27)$) -- ($(B3)+(0,0.27)$);

\draw[bcomp] (M2) -- ($(M2)+(0,-0.95)$);
\draw[bcomp] (M4) -- ($(M4)+(0,-0.95)$);
\draw[bcomp] (M6) -- ($(M6)+(0,-0.95)$);
\draw[bcomp] ($(M2)+(0,-0.95)$) -- ($(M6)+(0,-0.95)$);

\draw[bcomp] (M9)  -- ($(M9)+(0,-1.82)$);
\draw[bcomp] (M21) -- ($(M21)+(0,-1.82)$);
\draw[bcomp] ($(M9)+(0,-1.82)$) -- ($(M21)+(0,-1.82)$);

\draw[bcomp] (M13) -- ($(M13)+(0,-0.76)$);
\draw[bcomp] (M14) -- ($(M14)+(0,-0.76)$);
\draw[bcomp] ($(M13)+(0,-0.76)$) -- ($(M14)+(0,-0.76)$);

\lsingle{bcomp}{19}

\draw[bcomp] (M22) -- ($(M22)+(0,-0.82)$);
\draw[bcomp] (M24) -- ($(M24)+(0,-0.82)$);
\draw[bcomp] ($(M22)+(0,-0.82)$) -- ($(M24)+(0,-0.82)$);

\lsingle{rcomp}{5}
\lsingle{pcomp}{20}
\lsingle{pcomp}{23}

\draw[gcomp,decorate,decoration={brace,mirror,amplitude=3.6pt}]
    (7,-2.00) -- (8,-2.00);
\node[font=\scriptsize] at (7.5,-2.30) {entrance $\{7,8\}$};

\foreach \P/\lbl in {
    T1/{1^+},T2/{2^+},T3/{3^+},T4/{4^+},T5/{5^+}
}{
    \node[pt] at (\P) {};
    \node[toplab,above=2pt] at (\P) {$\lbl$};
}

\foreach \i in {1,...,24}{
    \node[pt] at (M\i) {};
    \node[midlab,below=2pt] at (M\i) {\(\i\)};
}

\foreach \P/\lbl in {
    B1/{1^-},B2/{2^-},B3/{3^-},B4/{4^-}
}{
    \node[pt] at (\P) {};
    \node[botlab,below=2pt] at (\P) {$\lbl$};
}


\draw[red!80!black,line width=0.8pt] (M7)  circle [radius=0.23];
\draw[red!80!black,line width=0.8pt] (M8)  circle [radius=0.23];
\draw[red!80!black,line width=0.8pt] (M11) circle [radius=0.23];
\draw[red!80!black,line width=0.8pt] (M12) circle [radius=0.23];

\draw[cycarrow]
  ($(M7)+(0,0.42)$)
  to[bend left=18]
  node[midway,above,font=\scriptsize,red!80!black] {$p$}
  ($(M12)+(0,0.42)$);

\draw[cycarrow]
  ($(M12)+(0,-0.42)$)
  to[bend right=18]
  node[midway,below,font=\scriptsize,red!80!black] {$q$}
  ($(M11)+(0,-0.42)$);

\draw[cycarrow]
  ($(M11)+(0,0.42)$)
  to[bend left=18]
  node[midway,above,font=\scriptsize,red!80!black] {$p$}
  ($(M8)+(0,0.42)$);

\draw[cycarrow]
  ($(M8)+(0,-0.42)$)
  to[bend right=18]
  node[midway,below,font=\scriptsize,red!80!black] {$q$}
  ($(M7)+(0,-0.42)$);

\end{tikzpicture}
\caption{
Bi-layer non-crossing partitions \(p\) and \(q\).
The cycle \(\mathbf{i}=(7,12,11,8)\) is highlighted in red; it satisfies
\(\pi_p(7)=\pi_p(12)\), \(\pi_q(12)=\pi_q(11)\),
\(\pi_p(11)=\pi_p(8)\), and \(\pi_q(8)=\pi_q(7)\).
The singleton blocks are represented by small arcs.
}
\label{fig:bilayer-nc-example-cycle}
\end{figure}

\begin{lemma}\label{one-side nesting}
Let $\mathbf{i}=(i_1,\dots,i_{2t})$ be a cycle with $\ell(\mathbf{i})=2t\ge 4$. Then there exist $k\neq r$ such that $J(V_k)\subsetneq J(V_r)$ or $J(W_k)\subsetneq J(W_r)$.
\end{lemma}

\begin{proof}
Let $\mathrm{Lab}(\mathbf{i})=\{i_1,\dots,i_{2t}\}=\{y_1<\cdots<y_{2t}\}$. Assume towards a contradiction that there do not exist $k\neq r$ such that $J(V_k)\subsetneq J(V_r)$ or $J(W_k)\subsetneq J(W_r)$.

Since $p$ is non-crossing, any two intervals among $J(V_1),\dots,J(V_t)$ are either disjoint or one contains the other. By assumption, no proper containment occurs, so they are pairwise disjoint. Each interval $J(V_k)$ has endpoints in $\{y_1,\dots,y_{2t}\}$ and contains at least two points of this set. As there are $t$ such intervals and $2t$ points in total, each $J(V_k)$ contains exactly two points of $\mathrm{Lab}(\mathbf{i})$. Therefore
$\{J(V_1),\dots,J(V_t)\}=\{[y_1,y_2],[y_3,y_4],\dots,[y_{2t-1},y_{2t}]\}$.

Applying the same argument to $q$, we get
$\{J(W_1),\dots,J(W_t)\}=\{[y_1,y_2],[y_3,y_4],\dots,[y_{2t-1},y_{2t}]\}$.
Hence for every $k\in\{1,\dots,t\}$ there exists $r\in\{1,\dots,t\}$ such that $J(W_k)=J(V_r)$. Since the entries of $\mathbf{i}$ are pairwise distinct, this implies
$\{i_{2k},i_{2k+1}\}=\{i_{2r-1},i_{2r}\}$.

We now show that this is impossible when $t\ge2$.

If $k<t$, then $2k+1\le 2t-1$, so $\{i_{2k},i_{2k+1}\}$ does not involve the cyclic convention $i_{2t+1}=i_1$. Hence, if
$\{i_{2k},i_{2k+1}\}=\{i_{2r-1},i_{2r}\}$,
then by pairwise distinctness of the entries of $\mathbf{i}$ we must have
$\{2k,2k+1\}=\{2r-1,2r\}$,
which is impossible.
\end{proof}

\begin{lemma}\label{lemma:even many in between }

If $x\neq y\in [l]$ and  $\pi_p(x)=\pi_p(y)\notin B_p(\mathbf{i}) $, then there exists even many points in $\mathrm{Lab}(\mathbf{i})$ strictly between $x$ and $y$.
Analogous statements hold for the $\pi_q$ and the $B_q(\mathbf{i})$.

\begin{proof}

 Suppose there exists a point $z\in \mathrm{Lab}([\mathbf{i}])$ such that $x<z<y$. By $\pi_p(x)\neq\pi_p(z)\in B_p(\mathbf{i}) $ and the non-crossing property of $p$, we obtain $\pi_p(z)\subset(b_t, b_{t+1})$, where $\pi_p(x)=\{b_1<b_2<\cdots<b_k\}\in p$ and $x\le b_t<b_{t+1}\le y$. Since $\pi_p(z)\subset(b_t, b_{t+1})$, there exists $i_{2s}$ such that $\pi_p(z)=\pi_p(i_{2s})=\pi_p(i_{2s+1})\in p$ so they are actually two differents points $i_{2s}, i_{2s+1}\in\mathrm{Lab}([\mathbf{i}])$ stricly between $x$ and $y$.

\end{proof}
\end{lemma}

\begin{corollary}\label{crossing account}
$\rho_{\mathbf{i}}(x)$ is even, $\forall x\in \mathcal C\cap [l]$. 
\end{corollary}
\begin{proof} 

Choose $y\in \mathcal C\cap [l]$ such that $i_k<y$ for all $1\le k\le 2t$.
Then
\[
\rho_{\mathbf i}(y)=0,
\]
hence $\rho_{\mathbf i}(y)$ is even.

Fix $x\in \mathcal C\cap [l]$. Since $x,y\in \mathcal C$, there exists a joining sequence
\[
D_1,\dots,D_m\in p\sqcup q
\]
from $x$ to $y$. Since $x,y\in [l]$, we may choose points
\[
x=z_0,z_1,\dots,z_n=y
\]
in $\mathcal C\cap [l]$ such that for each $0\le r\le n-1$, either
\[
\pi_p(z_r)=\pi_p(z_{r+1})\notin B_p(\mathbf i),
\qquad\text{or}\qquad
\pi_q(z_r)=\pi_q(z_{r+1})\notin B_q(\mathbf i).
\] 
Applying Lemma~\ref{lemma:even many in between } to each consecutive pair $(z_r,z_{r+1})$, we obtain
\[
\rho_{\mathbf i}(z_r)\equiv \rho_{\mathbf i}(z_{r+1})\pmod 2
\qquad (0\le r\le n-1).
\]
Therefore
\[
\rho_{\mathbf i}(x)=\rho_{\mathbf i}(z_0)\equiv \rho_{\mathbf i}(z_n)=\rho_{\mathbf i}(y)\equiv 0\pmod 2.
\]
Thus $\rho_{\mathbf i}(x)$ is even.
\end{proof}

\begin{proposition}\label{prop:gluing pair}
 Let $p\in X_u(\mathcal C)$ and $q\in X_d(\mathcal C)$, and let $\mathcal{G}$ be the associated bipartite multigraph.
Assume that $\mathcal{G}$ contains a cycle $\mathbf{i}$ of length at least $4$.
Then there exist blocks $B,C\in B_p(\mathbf i)$ (resp.\ $B,C\in B_q(\mathbf i)$) such that $\mathrm{span}(C)\subsetneq \mathrm{span}(B),$
and 
$\nexists\,U\in\mathcal C_p\ \text{(resp.\ $U\in\mathcal C_q$)}\ \text{with }
\mathrm{span}(C)\subsetneq \mathrm{span}(U)\subsetneq \mathrm{span}(B).$

 \begin{proof}

By Lemma \ref{one-side nesting}, without loss of generality, we may assume there exist $C_0,B_0\in B_p(\mathbf i)$ such that $C_0\prec_{\mathbf i}B_0$.
Hence the ancestor set
$\mathrm{Anc}_p(C_0):=\{\,D\in B_p(\mathbf i):\ C_0\prec_{\mathbf i}D\,\}$
is nonempty. Since it is finite, we may choose 
$B:=\max_{\prec_{\mathbf i}}\mathrm{Anc}_p(C_0).$
Consider the chain of blocks 
$\{\,D\in B_p(\mathbf i):\ C_0\preceq_{\mathbf i}D\prec_{\mathbf i}B\,\}.
$
It is nonempty, finite, and totally ordered by $\prec_{\mathbf i}$.
Let $C$ be the maximal element of this chain.
Then $C\triangleleft_{\mathbf i}B$, i.e. $C\prec_{\mathbf{i}} B\ \text{ and there is no }D\in B_P(\mathbf{i})\text{ with }C\prec_{\mathbf{i}} D\prec_{\mathbf{i}}B$   . Indeed, there doesn't exist any $U\in\mathcal C_p\ \text{with }
\mathrm{span}(C)\subsetneq \mathrm{span}(U)\subsetneq \mathrm{span}(B).$ 
Assume towards a contradiction that there exists
\(U\in \mathcal C_p\) such that
\[
\mathrm{span}(C)\subsetneq \mathrm{span}(U)\subsetneq \mathrm{span}(B).
\]
Since \(C\triangleleft_{\mathbf i}B\), we know there doesn't exist any $B'\in B_p(\mathbf i)$ such that $$\mathrm{span}(U)\subsetneq \mathrm{span}(B')\subsetneq \mathrm{span}(B)$$
We have the points strictly between $\max(U)<\max(B)$, contributes no label of $\mathbf i$.

Now \(B=\max_{\prec_{\mathbf i}}\mathrm{Anc}_p(C_0)\) means that \(B\) has no strict
upper container in \(B_p(\mathbf i)\). Therefore, for every \(D\in B_p(\mathbf i)\),
either \(J(D)\subseteq J(B)\) or \(J(D)\cap J(B)=\varnothing\).

We count the cycle labels strictly to the right of \(\max(U)\), grouping them according
to the \(p\)-side cycle blocks in \(B_p(\mathbf i)\). The block \(B\) contributes exactly one such label,
namely \(\max(B)\). For every other block \(D\in B_p(\mathbf i)\setminus\{B\}\), the
number of labels of \(D\) lying strictly to the right of \(\max(U)\) is either \(0\) or
\(2\). Indeed, \(D\) is either disjoint from \(B\), in which case both labels of \(D\)
lie on the same side of \(\max(U)\), or \(J(D)\subsetneq J(B)\), in which case the two
labels of \(D\) again lie simultaneously to the right of \(\max(U)\) or not.

Hence the total number of cycle labels strictly to the right of \(\max(U)\) is odd.
Equivalently,
$\rho_{\mathbf i}(\max(U))$
is odd. On the other hand, by Corollary~\ref{crossing account}, 
$\rho_{\mathbf i}(\max(U))$
is even. This is a contradiction.

\end{proof}
\end{proposition}

\begin{corollary}
Assume that there is no pair of blocks $B,C$ in either $p\cap M_E^{b}(q,p)$ or $q\cap M_E^{b}(q,p)$
satisfying the above property. Then every cycle in $\mathcal{G}$
has length at most $2$.

\end{corollary}

\begin{definition}
  Let $(p,\mathrm{col}_p)\in \NC(k,l)$ and $(q,\mathrm{col}_q)\in \NC(l,m)$
be a pair of colored partitions satisfying the assumptions of Proposition \ref{prop:gluing pair}.
We define new colored partitions$
(\bar p,\overline{\mathrm{col}}_p)\in NC(k,\ell), (\bar q,\overline{\mathrm{col}}_q)\in NC(\ell,m)
$
as follows.

Define the intermediate family $\mathcal F_{C,B}
:=\Bigl\{\,V\in p \;:\; \mathrm{span}(C)\subsetneq \mathrm{span}(V)\subsetneq \mathrm{span}(B)\Bigr\}.
$
(As observed above, no block in $\mathcal F_{C,B}$ belongs to $\Ccal_p\sqcup \Ccal_q$.)
Set
\[
\overline{B}
:= B\ \cup\ C\ \cup\ \bigcup_{V\in \mathcal F_{C,B}} V,
\qquad
\bar p
:=\Bigl(p\setminus\bigl(\{B,C\}\cup \mathcal F_{C,B}\bigr)\Bigr)\ \cup\ \{\overline{B}\}.
\]
Define the new coloring $\overline{\mathrm{col}}$ on $\bar p$ by $\overline{\mathrm{col}}(\overline{B}) := \mathrm{col}(B)$,
$\bar{\mathrm{col}}(D):=\mathrm{col}(D)\ \ \text{for all }D\in \bar p\setminus\{\overline{B}\}.$
The construction on the $\bar{q}$-side is completely analogous. Finally, $\bar p$ (resp.\ $\bar q$) is well-defined:

\emph{(Partition property.)}
By construction, we replace a family of blocks of $p$ by their union $\overline B$ and keep all other blocks unchanged.
Since distinct blocks of a partition are disjoint, the union $\overline B$ is disjoint from every unchanged block,
and unchanged blocks remain pairwise disjoint. Moreover, every point of $[k]\sqcup[l]$ belongs either to an unchanged
block of $p$ or to one of the merged blocks, hence it belongs to exactly one block of $\bar p$.
Therefore $\bar p$ is a partition of $[k]\sqcup[l]$ (and similarly $\bar q$ is a partition of $[l]\sqcup[m]$).

\emph{(Noncrossingness.)}
Let $D$ be any unchanged block of $p$. If $D$ is disjoint from $\mathrm{span}(B)$, then $D$ is trivially noncrossing
with $\overline B$. If $D$ meets $\mathrm{span}(B)$, then by noncrossingness of $p$ its span is either contained in
$\mathrm{span}(B)$ or contains it; in either case $\mathrm{span}(D)$ is nested with $\mathrm{span}(\overline B)
=\mathrm{span}(B)$. Hence no new crossing can be created by replacing the merged family with $\overline B$.
Thus $\bar p$ is noncrossing. The same argument applies to $\bar q$.

Consequently, $(\bar p, \bar{col})\in NC(k,l)$ and $(\bar q,\bar{col})\in NC(l,m)$.

\end{definition}

\begin{property}
 $\bar{p}\in X_u$ and $\bar{q} \in X_d$
 \begin{proof}
 By the previous argument, every $U\in\mathcal C_p$ (resp.\ $U\in\mathcal C_q$) satisfies
$U\notin \{B,C\}\cup \mathcal F_{C,B}.$
Hence, by the definition of $\bar p$ (resp.\ $\bar q$), the block $U$ is unchanged and in particular $
U\in \bar p \quad (\text{resp.\ } U\in \bar q).$

\end{proof}
 
\end{property}

\begin{corollary}
 $\mathcal F_{C,B}\subsetneq I_E^{b}(q,p)$
 \end{corollary}
 
 \begin{property}
 $|\bar{p}\cap M^{b}_E(\bar{q},\bar{p})|\leq|p\cap M^{b}_E(q,p)|$, $|\bar{q}\cap M^{b}_E(\bar{q},\bar{p})|\leq|q\cap M^{b}_E(q,p)|$ and $| M^{b}_E(\bar{q},\bar{p})|<|M^{b}_E(q,p)|$
 \begin{proof}
 By the previous property, $\bar p\in X_u$ and $\bar q\in X_d$. Hence $\Ccal$ is also a
$(\bar p,\bar q)$-connected component, and $E$ remains an entrance for $(\bar p,\bar q)$.
On the $p$-side, $\bar p$ is obtained from $p$ by merging the family $\{B,C\}\cup \mathcal F_{C,B}$
into the single block $\overline B$. All blocks outside this family are unchanged, hence their membership in
$M_E^b$ is unchanged. Inside the modified family, the block $B$ is replaced by $\overline B$, whereas the distinct
block $C\in M_E^b(q,p)$ is absorbed into $\overline B$ and therefore disappears as a separate block. Thus at least
one $M_E^b$-block is lost on the $p$-side, and so $|\bar p\cap M_E^b(\bar q,\bar p)|<|p\cap M_E^b(q,p)|.$
Similarly,     $|\bar q\cap M_E^b(\bar q,\bar p)|\le |q\cap M_E^b(q,p)|.$
Consequently,
$|M_E^b(\bar q,\bar p)|<|M_E^b(q,p)|.$

 \end{proof}
 \end{property}

\begin{lemma}
 Assume every cycle $\mathbf i$ in bipartite multigraph $\mathcal{G}$ has at most length $2$ and there is no pair of blocks $B,C$ in either $B_p(\mathbf i)$ or $B_q(\mathbf i)$
satisfying the property in Proposition \ref{prop:gluing pair}. Then there exist  $V\in p\cap M_E^{b}(q,p)$ and $V'\in q\cap M_E^{b}(q,p)$ such that one of the following conditions holds:

 \begin{itemize}
\item $\min V=\min V'$ with $\max V\in V'$ or $\max V'\in V$
\item $\max V=\max V'$ with $\min V'\in V$ or $\min V\in V'$
\end{itemize}

 \begin{proof}

Take a longest simple path $V_1-\cdots-V_{n-1}-V_n$ in $\mathcal G$.(By a \emph{path} in $\mathcal G$ we mean no repeated vertices, disregarding edge multiplicities.
) WLOG, assume
$V:=V_{n-1}\in q\cap M_E^{b}(q,p)$.

If $\pi_p(\min V)=\pi_p(\max V)=:U$, then $U$ is the unique $p$-neighbour of $V$, and $U=V_n$:
indeed, if $U'\neq U$ satisfies $U'\cap V\neq\varnothing$, then noncrossingness forces
$\mathrm{span}(U')\subsetneq \mathrm{span}(U)$. $\{\min V,\max V\}$ forms a $2$-cycle, by Corollary \ref{crossing account}, we have $\mathcal C\cap[\min V,\max V]=\varnothing$. Then there doesn't exist any $B\in\mathcal C_p\ \text{with }
\mathrm{span}(U')\subsetneq \mathrm{span}(B)\subsetneq \mathrm{span}(U)$,
which contradicts the non-existence of such a nested pair. 
Otherwise, suppose that $\min U<\min V$ or $\max V<\max U$.
Then $\pi_q(\min U)\neq V$ or $\pi_q(\max U)\neq V$, so $U$ has a $q$-neighbour $W\neq V$.
Since $U=V_n$, if $W\notin \{V_1,\dots,V_{n-1}\}$, then $V_1-\cdots-V_n-W$
is a longer simple path, contradicting the maximality of $V_1-\cdots-V_n$.
Hence necessarily $W=V_j$ for some $j\le n-2$.
But then
$V_j-V_{j+1}-\cdots-V_n-V_j$
is a cycle of length at least $4$, contradicting the assumption that every cycle in $\mathcal G$
has length at most $2$.
Therefore we must have $\min U=\min V$ and $\max U=\max V$, and we are done.

On the other hand, assume that$
U_1=\pi_p(\min V)\neq \pi_p(\max V)=U_2.$
Suppose that both
\[
\min U_1<\min V
\qquad\text{and}\qquad
\max V<\max U_2
\]
hold. Then $U_1$ and $U_2$ are two distinct $p$-neighbours of $V$.

We first claim that neither $U_1$ nor $U_2$ can be equal to $V_n$.
Indeed, assume for instance that $U_1=V_n$.
Since $\min U_1<\min V$, we have $\pi_q(\min U_1)\neq V$, so $U_1$ has a $q$-neighbour
$W\neq V$.
If $W\notin\{V_1,\dots,V_n\}$, then
\[
V_1-\cdots-V_n-W
\]
is a simple path longer than $V_1-\cdots-V_n$, contradicting the maximality of the latter.
Hence $W=V_i$ for some $i\le n-3$.
But then the subpath from $V_i$ to $V_n$, together with the extra edge $V_n-W$,
forms a cycle of length at least $4$, contradicting the assumption that every cycle in
$\mathcal G$ has length at most $2$.
Thus $U_1\neq V_n$. The same argument shows that $U_2\neq V_n$.

Next, we show that at most one of $U_1,U_2$ can lie on the path $V_1-\cdots-V_n$.
Indeed, if for instance $U_1=V_i$ for some $i\le n-3$, then the subpath from $V_i$ to
$V=V_{n-1}$, together with the extra edge $U_1\!-\!V$, forms a cycle of length at least $4$,
again a contradiction.
Hence among the two distinct $p$-neighbours $U_1,U_2$ of $V$, at most one can coincide with
the p-neighbour $V_{n-2}$.
Therefore there exists
\[
\widetilde U\in\{U_1,U_2\}\setminus\{V_{n-2}\}.
\]
Since $\widetilde U\neq V_n$ and $\widetilde U$ cannot coincide with any $V_i$ for $i\le n-3$,
we obtain
\[
\widetilde U\notin\{V_1,\dots,V_n\}.
\]

Now choose a $q$-neighbour $\widetilde W$ of $\widetilde U$ distinct from $V$.
Such a neighbour exists because either $\widetilde U=U_1$ and then $\min U_1<\min V$, or
$\widetilde U=U_2$ and then $\max V<\max U_2$.
If $\widetilde W\in\{V_1,\dots,V_n\}$, say $\widetilde W=V_i$, then necessarily $i\le n-3$,
and the subpath from $V_i$ to $V$, together with the two extra edges
\[
V-\widetilde U
\qquad\text{and}\qquad
\widetilde U-\widetilde W,
\]
forms a cycle of length at least $4$, again impossible.
Therefore $\widetilde W\notin\{V_1,\dots,V_n\}$, and hence
\[
V_1-\cdots-V-\widetilde U-\widetilde W
\]
is a simple path strictly longer than $V_1-\cdots-V_n$, a contradiction.

Therefore the case
\[
\min U_1<\min V
\qquad\text{and}\qquad
\max V<\max U_2
\]
cannot occur.

\end{proof}
\end{lemma}

 \begin{definition}\label{def:truncation-aligned-pair}
Let $p\in \NC(k,l)$ and $q\in \NC(l,m)$, and fix an adjacent pair $(V,V')$
satisfying one of the  conditions of the previous lemma.

We first define a new pair of \emph{underlying} partitions $(p',q')$.

Assume for instance that
\[
V\in p\cap M_E^{b}(q,p),\qquad
V'\in q\cap M_E^{b}(q,p),\qquad
\max V=\max V'
\quad\text{and}\quad
\min V'\in V.
\]
Write
\[
V=\{v_1<\cdots<v_n\},
\]
and let $k$ be such that $v_k=\min V'$. Define
\[
V_l:=\{v_1,\dots,v_{k-1}\},
\qquad
V_r:=\{v_k,\dots,v_n\}.
\]
Then we set
\[
p':=(p\setminus\{V\})\cup\{V_l,V_r\},
\qquad
q':=q.
\]

The other three cases are defined similarly: one splits the
block whose endpoint properly contains the corresponding endpoint of the other
block, and leaves the other partition unchanged.
\end{definition}
 
 \begin{definition}\label{def:recoloring-truncation}
Assume now that $(p,col_p)\in \NC_\Lambda(k,l)$ and
$(q,col_q)\in \NC_\Lambda(l,m)$, and let $(p',q')$ be obtained from
Definition~\ref{def:truncation-aligned-pair}.

In the situation
\[
V\in p\cap M_E^{b}(q,p),\qquad
V'\in q\cap M_E^{b}(q,p),\qquad
\max V=\max V',
\quad
\min V'\in V,
\]
with
\[
V=V_l\sqcup V_r
\]
as above, let $\mathcal S$ be the family of $p$-blocks lying strictly between
$\max(V_l)$ and $\min(V_r)=\min(V')$ in the sense of
Definition~\ref{blocks}. Let $g$
be the ordered product of the labels of the relative outer blocks of
$\mathcal S$ (See Definition \ref{boundary product}).

Since the interval $[\min V,\max V]$ decomposes into the interval of $V_l$, the
spans of those relative outer blocks, and the interval $[\min V',\max V']$, we
define
\[
\gamma_l:=col_p(V)\,col_q(V')\,g^{-1},
\qquad
\gamma_r:=col_q(V')^{-1}.
\]
We then define the new colors by
\[
col'_p(V_l):=\gamma_l,
\qquad
col'_p(V_r):=\gamma_r,
\]
and keep all other labels unchanged:
\[
col'_p(B):=col_p(B)\quad(B\in p,\ B\neq V),
\qquad
col'_q(C):=col_q(C)\quad(C\in q).
\]

The other three endpoint-alignment cases are defined analogously, by the same
principle: the block that is split receives one new label copied from the
aligned block on the other side (with the appropriate inverse when coming from
$q$), and the remaining new label is determined by the corresponding
factorization of the original block label.
\end{definition}
 
 \begin{lemma}\label{lem:truncation-preserves-middle-constraints}
Let $(p,col_p)$ and $(q,col_q)$ be as above, and let
$(p',col'_p)$, $(q',col'_q)$ be obtained from
Definitions~\ref{def:truncation-aligned-pair}
and~\ref{def:recoloring-truncation}. Then:

\begin{enumerate}
    \item One has
    \[
    |M_E^{b}(q',p')|<|M_E^{b}(q,p)|.
    \]

    \item For every choice of external labels $\vec r,\vec d$,
    \[
    \theta_{p}^{\vec r}\cap\Omega_{q}^{\vec d}
    =
    \theta_{p'}^{\vec r}\cap\Omega_{q'}^{\vec d}.
    \]
\end{enumerate}
\end{lemma}

\begin{proof}
We prove the displayed case
\[
V\in p,\qquad V'\in q,\qquad \max V=\max V',
\qquad \min V'\in V,
\]
the other three cases being analogous.

For (1), the only change in the underlying pair is that the block $V$ is
replaced by $V_l$ and $V_r$, while $q$ is unchanged. By construction, the block
$V_r$ has the same endpoint interval as $V'$, hence it no longer contributes a
new boundary block. Thus the old boundary block $V$ is replaced only by $V_l$,
and therefore
\[
|M_E^{b}(q',p')|<|M_E^{b}(q,p)|.
\]

For (2), let $\vec s\in \Lambda^l$. All block conditions except the one on $V$
are unchanged. Thus it is enough to compare the old condition on $V$ with the
new conditions on $V_l$ and $V_r$.

Since $V'$ is unchanged and
\[
[\min V',\max V']=[\min V_r,\max V_r],
\]
the condition coming from $q$ is
\[
\prod_{i\in[\min V_r,\max V_r]} s_i
=
col_q(V')^{-1}
=
col'_p(V_r).
\]
On the other hand, by construction of $g$, the interval $[\min V,\max V]$
decomposes as the interval of $V_l$, followed by the spans of the relative
outer blocks of $\mathcal S$, followed by the interval of $V_r$. Hence
\[
\prod_{i\in[\min V,\max V]} s_i
=
\left(\prod_{i\in[\min V_l,\max V_l]} s_i\right)\,
g\,
\left(\prod_{i\in[\min V_r,\max V_r]} s_i\right).
\]
Therefore the old condition
\[
\prod_{i\in[\min V,\max V]} s_i=col_p(V)
\]
is equivalent, using
\[
\prod_{i\in[\min V_r,\max V_r]} s_i=col_q(V')^{-1},
\]
to
\[
\prod_{i\in[\min V_l,\max V_l]} s_i
=
col_p(V)\,col_q(V')\,g^{-1}
=
\gamma_l
=
col'_p(V_l).
\]

So the single old condition on $V$ is equivalent to the two new conditions on
$V_l$ and $V_r$. Since all other block conditions are unchanged, we obtain
\[
\theta_{p}^{\vec r}\cap\Omega_{q}^{\vec d}
=
\theta_{p'}^{\vec r}\cap\Omega_{q'}^{\vec d}.
\]
This proves the lemma.
\end{proof}

\begin{proposition}\label{constant entrance }
Fix $p_0\in \NC(k,l)$ and $q_0\in \NC(l,m)$, and let $\mathcal{C}$ be a 
$(p_0,q_0)$–connected component with entrance $E$. 
For any $(p,col_p)\in NC_{\Lambda}(k,l)$ with $p\in X_u(\mathcal{C})$ and any 
$(q,col_q)\in NC_{\Lambda}(l,m)$ with $q\in X_d(\mathcal{C})$, the product 
$\prod \vec{s}_{|E}$ is constant over all $\vec s\in \mathcal{S}_E(q,p) $
where $\mathcal{S}_E(q,p)
=\bigl\{\vec y\in\Lambda^{l}\,\big|\,
\prod_{i\in [\min B,\max B]} y_i=\operatorname{col}(B)^{\varepsilon(B)}\ 
\forall B\in\mathcal{M}^b_E(q,p)
\bigr\}$, $\varepsilon(B):=
\begin{cases}
+1,& B\in p,\\
-1,& B\in q.
\end{cases}$
Moreover, the value of this constant depends only on $\widetilde{\mathcal{M}}^b_E(q,p)$,
viewed as the family of ordered pairs $(B,\mathrm{col}(B))$ consisting of a block and 
its assigned color.

\end{proposition}

 \begin{proof} We proceed by induction on $|M_{E}^{b}(q,p)|$. 
When $|M_{E}^{b}(q,p)| = 1$, the only element in $M_{E}^{b}(q,p)=\mathcal{M}^b_E(q,p)$ must be the consecutive block $E$ (otherwise, we would have $|M_{E}^{b}(q,p)| > 1$). It follows that for any $\vec{s} \in \mathcal{S}_E(q,p)$, we have $\prod\vec{s}_ {\restriction E} \equiv {\mathrm{col}(E)}^{\varepsilon(B)}$.

 Suppose the statement holds for all $p \in X_u$ and $q \in X_d$ with $|M_{E}^{b}(q,p)|\le n$
 
 We now consider the case where $|M_{E}^{b}(q,p)|=n+1$. If there is no pair of blocks $B,C$ in either $p\cap M_E^{b}(q,p)$ or $q\cap M_E^{b}(q,p)$
satisfying the property in Proposition    \ref{prop:gluing pair}, then every cycle in bipartite multigraph $\mathcal{G}$ has at most length $2$ and there exists $p'\in X_u$ and $q'\in X_d$ such that $|M_{E}^{b}({q}^{\prime},{p}^{\prime})|\leq n$, then by the induction hypothesis, $\prod{\vec{s}_{\restriction{E}}\equiv h_{E}({q}}^{\prime},{p}^{\prime})$, for  any $\vec{s}\in\mathcal{S}_E(q^{\prime},p^{\prime})$. By Lemma \ref{lem:truncation-preserves-middle-constraints}, we have $\mathcal{S}_E(q,p)=\mathcal{S}_E(q^{\prime},p^{\prime})$, hence $\prod{\vec{s}_{\restriction{E}}\equiv h_{E}({q}}^{\prime},{p}^{\prime})$, for any $\vec{s}\in\mathcal{S}_E(q,p)$.

If there is a pair of blocks $B,C$ in either $p\cap M_E^{b}(q,p)$ or $q\cap M_E^{b}(q,p)$
satisfying the property in Proposition    \ref{prop:gluing pair} ,  there exist $\bar{p}\in X_u$ and $\bar{q}\in X_d$ and we have $|M_{E}^{b}(\bar{q},\bar{p})|\leq n$, then by the induction hypothesis, 
 $\prod\vec{s}_{\restriction{E}}\equiv h_{E}(\bar{q},\bar{p})$, for  any $\vec{s} \in \mathcal{S}_E(\bar{q},\bar{p})$. Since $\mathcal{S}_E(q,p)\subseteq \mathcal{S}_E(\bar{q},\bar{p})$,
 we have $\prod\vec{s}_{\restriction{E}}\equiv h_{E}({\bar{q}},{\bar{p}})$, for any $\vec{s}\in\mathcal{S}_E(q,p)$. 
 \medskip

\end{proof}

 \begin{definition}

If \(\mathrm{Th}_p(\mathcal C)\neq\emptyset\), we denote by
\(T_l^{\uparrow}\) and \(T_r^{\uparrow}\) the minimal and maximal elements of
\(\mathrm{Th}_p(\mathcal C)\) with respect to  order \(\lhd\)(See Definition \ref{through block order}), respectively.
If \(\mathrm{Th}_q(\mathcal C)\neq\emptyset\), we denote by
\(T_l^{\downarrow}\) and \(T_r^{\downarrow}\) the minimal and maximal elements of
\(\mathrm{Th}_q(\mathcal C)\) with respect to \(\lhd\), respectively.

If \(\mathcal C\) is upper-half (resp. lower-half), we simply write
\(T_l,T_r\) for \(T_l^{\uparrow},T_r^{\uparrow}\) (resp.
\(T_l^{\downarrow},T_r^{\downarrow}\)). If \(\mathcal C\) is of through type, both
pairs \(T_l^{\uparrow},T_r^{\uparrow}\) and \(T_l^{\downarrow},T_r^{\downarrow}\)
are defined.
\end{definition}

\begin{definition}\label{def:frame}
For each \((p,q)\)-connected component \(\mathcal C\in K(q,p)\) with
\(\mathcal C\cap[l]\neq\emptyset\), 

The (middle-row) interval
$
F_{\mathcal{C}} := [\min( T_l\cap[l]),\max( T_r\cap[l])]\subseteq [l]
$ is called the \emph{frame} of the upper-half (resp.\ lower-half) component
$\mathcal{C}$.

We say that
$\mathcal{C}$ is an \emph{outer upper-half} (resp.\ \emph{outer lower-half})
component if there is no distinct upper-half (resp.\ lower-half) component
$\mathcal{C}'$ such that  $
F_{\mathcal{C}}\subseteq F_{\mathcal{C}'}.
$
\end{definition}

\begin{corollary}\label{cor: constant interval}
Let $(p,\vec{t})\in \NC_{\Lambda}(k,l)$ and $(q,\vec{z})\in \NC_{\Lambda}(l,m)$ such that $T_{(q,\vec{z})} \circ T_{(p,\vec{t})} \neq 0$
\begin{enumerate}

\item Let \(\mathcal C\) be a through connected component. The restricted products
\[
\vec s\longmapsto \prod \vec{s}_{\restriction[\max(V_r^{\uparrow}) + 1,\ \max(\mathcal{C}\cap [l])]},
\qquad
\vec s\longmapsto \prod \vec{s}_{\restriction[\max(V_r^{\downarrow}) + 1,\ \max(\mathcal{C}\cap [l])]}
\]
are constant on \(\bigcup_{\vec r,\vec d}(\theta_p^{\vec r}\cap\Omega_q^{\vec d})\); denote their values by \(h\) and \(\mu\), respectively.

\item  Let $\mathcal{C}$ be a upper-half(or lower-half)connected component. We have $\prod \vec{s}_{\restriction{F_{\mathcal{C}}}}\equiv f_{\mathcal{C}}$, for all $\vec{s} \in \bigcup_{\vec{r}, \vec{d}} (\theta_p^{\vec{r}} \cap \Omega_q^{\vec{d}})$.
\end{enumerate}
\end{corollary}

\begin{proof}

 \textup{(1)} By Lemma \ref{lem:two-sided-middle-decomposition} and the definition of $V_r^{\uparrow}$, any
$x\in [\max(V_r^{\uparrow}\cap [l])+1,\max(\Ccal\cap[l])]$ either lies in an upper
entrance of $\Ccal$ or in $\mathrm{span}(D)$ for some lower outer block
$D\in\Ccal_p$. By Definition \ref{Entrance}, the upper entrances of $\Ccal$ and the spans of lower outer blocks in $\Ccal_p$ are pairwise disjoint, and each
of them is a consecutive integer interval. Hence they form a partition of $[\max(V_r^{\uparrow}\cap [l])+1,\max(\Ccal\cap[l])]$ into disjoint consecutive
intervals. We denote this family by $\mathcal{J}$. We thus endow the family $\mathcal{J}$ with a strict total order $\prec$ by
declaring that, for intervals $I,J$ in this family, $I\prec J$ if and only if $\max I < \max J$.
By Proposition \ref{constant entrance }(1), for each upper entrance $E\in\mathcal{J}$, let
$h_E\in\Lambda$ be the element given by Proposition \ref{constant entrance }(1). For each lower outer block
$D\in\Ccal_p$ with $\mathrm{span}(D)\in\mathcal{J}$, we already have
a label $t_D\in\Lambda$ coming from the colouring of $p$. We define a family $(u_I)_{I\in\mathcal{J}}$ in $\Lambda$ by
$u_I:=h_E$ if $I=E$ is an upper entrance, and $u_I:=t_D$ if $I=\mathrm{span}(D)$ for a lower outer block $D$.
Write $\mathcal{J}=\{I_1\prec\cdots\prec I_s\}$. We then define
$h:=\prod_{I\in\mathcal{J}}^{\prec}u_I=u_{I_1}\cdots u_{I_s}$, that is, $h$ is the product of the entrance-labels $h_E$ and the lower-outer-block labels $t_D$, taken in increasing $\prec$ order. By definition of $\bigcup_{\vec r,\vec d}(\theta_p^{\vec r}\cap\Omega_q^{\vec d})$, we have $\bigcup_{\vec r,\vec d}(\theta_p^{\vec r}\cap\Omega_q^{\vec d})\subset \mathcal S_E(q,p)$ for any entrance $E$, hence
$\prod \vec s_{\restriction[\max(V_r^{\uparrow}\cap [l])+1,\max(\mathcal C\cap [l])]}\equiv h$
for all $\vec s\in \bigcup_{\vec r,\vec d}(\theta_p^{\vec r}\cap\Omega_q^{\vec d})$.
The remaining part of the proof of (1) is completely analogous.

\textup{(2)}
Without loss of generality, assume that $\mathcal C$ is an upper-half
connected component; the lower-half case is symmetric. Write
$F_{\mathcal C}=[a,b]$, where $a:=\min(V_l\cap [l])$ and
$b:=\max(V_r\cap [l])$, and let $c:=\min(\mathcal C\cap[l])$ and
$d:=\max(\mathcal C\cap[l])$. Then
$[c,a-1]\sqcup F_{\mathcal C}\sqcup [b+1,d]=[c,d]$,
omitting empty intervals if necessary.

By Lemma \ref{lem:two-sided-middle-decomposition} and by the definition of the
frame, every point of $[c,a-1]$ or $[b+1,d]$ either lies in the span of a no $p$-covered lower
outer block or belongs to an upper entrance of $\mathcal C$. Hence each of the intervals $[c,a-1]$ and $[b+1,d]$ is partitioned
into subintervals of these two types.

Likewise, every point of $[c,d]$ either lies in the span of an no $q$-covered upper outer block, or belongs to a upper entrance of $\mathcal C$.
Thus $[c,d]$ is partitioned into subintervals of these two types.

If $I$ is the span of a lower outer block $B$ of $p$, then
$p\in \NC_{\Lambda}(k,l)$ implies $\prod_I \vec s\,'=col_p(B),  \forall\vec s\,'\in \bigcup_{\vec{r}, \vec{d}} (\theta_p^{\vec{r}} \cap \Omega_q^{\vec{d}})$.
If $I$ is the span of an upper outer block $D$ of $q$, then
$q\in \NC_{\Lambda}(l,m)$ implies $\prod_I \vec s\,'=col_p(D),  \forall\vec s\,'\in \bigcup_{\vec{r}, \vec{d}} (\theta_p^{\vec{r}} \cap \Omega_q^{\vec{d}})$.
Finally, if $I$ is an upper entrance or a lower entrance of the connected
component $\mathcal C$, then Proposition \ref{constant entrance } yields $\prod_I \vec s = h_E(q,p),
\text{for any } \vec s\in \bigcup_{\vec{r}, \vec{d}} (\theta_p^{\vec{r}} \cap \Omega_q^{\vec{d}})\subsetneq\mathcal S_E(q,p)$.

Therefore there exist $\alpha,\beta,\gamma\in \Lambda$ such that
\[
\prod_{[c,a-1]} \vec s\,'=\alpha,\qquad
\prod_{[b+1,d]} \vec s\,'=\beta,\qquad
\prod_{[c,d]} \vec s\,'=\gamma
\]
for every
\[
\vec s\,'\in \bigcup_{\vec r,\vec d}(\theta_p^{\vec r}\cap \Omega_q^{\vec d}).
\]
Indeed, each of these intervals is partitioned into subintervals on which the product is constant,
so the product over the whole interval is constant as well.

Since
\[
[c,d]=[c,a-1]\sqcup F_{\mathcal C}\sqcup [b+1,d]
\]
is an ordered decomposition into consecutive intervals, we obtain
\[
\prod_{[c,d]} \vec s\,'
=
\left(\prod_{[c,a-1]} \vec s\,'\right)
\left(\prod_{F_{\mathcal C}} \vec s\,'\right)
\left(\prod_{[b+1,d]} \vec s\,'\right).
\]
Hence
\[
\prod_{F_{\mathcal C}} \vec s\,'=\alpha^{-1}\gamma\beta^{-1},
\]
which shows that the product over $F_{\mathcal C}$ is independent of $\vec s\,'$.
\end{proof}

\begin{definition}\label{gap}
Let \(V\in \mathrm{Th}(\pi)\). If there exists \(W\in \mathrm{Th}(\pi)\) such that \(W \lhd V\), we define the
\emph{left-adjacent through-block} of \(V\) to be the unique block \(V'\in \mathrm{Th}(\pi)\) such that
\[
V'\lhd V
\qquad\text{and}\qquad
\nexists\,W\in \mathrm{Th}(\pi)\ \text{with}\ V'\lhd W \lhd V.
\]
Otherwise, we say that \(V\) has no left-adjacent through-block.
Otherwise, we say that \(V\) has no left-adjacent through-block.

If the left-adjacent through-block \(V'\) of \(V\) exists, set
\[
\ell^{\uparrow}_{\pi}(V):=\max(V'\cap[m]),
\qquad
\ell^{\downarrow}_{\pi}(V):=\max(V'\cap[n]);
\]
otherwise set
\[
\ell^{\uparrow}_{\pi}(V):=0,
\qquad
\ell^{\downarrow}_{\pi}(V):=0.
\]

We then define
\[
\mathcal P_{\pi}^{\uparrow}(V)
:=
\{\,B\in \pi:\ B\text{ is an upper single-layer block and }
\ell_{\pi}^{\uparrow}(V)<\max(B\cap[m])<\min(V\cap[m])\,\},
\]
\[
\mathcal P_{\pi}^{\downarrow}(V)
:=
\{\,B\in \pi:\ B\text{ is a lower single-layer block and }
\ell_{\pi}^{\downarrow}(V)<\max(B\cap[n])<\min(V\cap[n])\,\}.
\]

\end{definition}

\begin{corollary}\label{colour of components}
Let \(p\in \NC_{\Lambda}(k,l)\) and \(q\in \NC_{\Lambda}(l,m)\) such that
\(T_q\circ T_p\neq 0\) and \(\theta_p^{\vec r}\cap \Omega_q^{\vec d}\neq\varnothing\).
Let \(\mathcal C\) be a \((p,q)\)-connected component.

\begin{enumerate}
\item If \(\mathcal C\) is of upper-half type, let \(T_l,T_r\in \mathrm{Th}(p)\) be,
respectively, the leftmost and rightmost through-blocks of \(\mathcal C\), and let
\(T_l'\) be the left-adjacent through-block of \(T_l\), if it exists. Write
\[
t_{\mathcal C}:=t_{T_r},
\qquad
t'_{\mathcal C}:=
\begin{cases}
t_{T_l'},& \text{if }T_l'\text{ exists},\\
1_\Lambda,& \text{otherwise}.
\end{cases}
\]
Set
\[
g_{\mathcal C}:=\mathcal T^{\downarrow}(\mathcal P_{p}^{\downarrow}(T_l)),
\qquad
\alpha_{\mathcal C}:=\mathcal T^{\uparrow}(\mathcal P_{p}^{\uparrow}(T_l)).
\]
Equivalently, if
\[
\partial^\downarrow(\mathcal P_{p}^{\downarrow}(T_l))=\{W_1\prec\cdots\prec W_t\},
\qquad
\partial^\uparrow(\mathcal P_{p}^{\uparrow}(T_l))=\{U_1\prec\cdots\prec U_s\},
\]
then
\[
g_{\mathcal C}=\mathrm{col}(W_1)\cdots \mathrm{col}(W_t),
\qquad
\alpha_{\mathcal C}=\bigl(\mathrm{col}(U_1)\cdots \mathrm{col}(U_s)\bigr)^{-1}.
\]
Thus \(g\) is the ordered product of the labels of the relative lower outer blocks of \(p\)
lying between \(T_l'\) and \(T_l\), while \(\alpha\) is the inverse of the ordered product
of the labels of the relative upper outer blocks of \(p\) lying between \(T_l'\) and \(T_l\).
Then
\[
\prod_{\mathcal C\cap[k]} \vec r
=
\Bigl(t_{\mathcal C}\, f_{\mathcal C}^{-1}\, g_{\mathcal C}^{-1}\, (t'_{\mathcal C})^{-1}\, \alpha_{\mathcal C}\Bigr)^{-1}.
\]

\item If \(\mathcal C\) is of lower-half type, let \(T_l,T_r\in \mathrm{Th}(q)\) be,
respectively, the leftmost and rightmost through-blocks of \(\mathcal C\), and let
\(V_l'\) be the left-adjacent through-block of \(V_l\), if it exists. Write
\[
\iota_{\mathcal C}:=\iota_{V_r},
\qquad
\iota'_{\mathcal C}:=
\begin{cases}
\iota_{V_l'},& \text{if }V_l'\text{ exists},\\
1_\Lambda,& \text{otherwise}.
\end{cases}
\]
Set
\[
\beta_{\mathcal C}:=\mathcal T^{\downarrow}(\mathcal P_{q}^{\downarrow}(V_l)),
\qquad
b_{\mathcal C}:=\mathcal T^{\uparrow}(\mathcal P_{q}^{\uparrow}(V_l)).
\]
Equivalently, if
\[
\partial^\downarrow(\mathcal P_{q}^{\downarrow}(V_l))=\{W_1\prec\cdots\prec W_t\},
\qquad
\partial^\uparrow(\mathcal P_{q}^{\uparrow}(V_l))=\{U_1\prec\cdots\prec U_s\},
\]
then
\[
\beta_{\mathcal C}=\mathrm{col}(W_1)\cdots \mathrm{col}(W_t),
\qquad
b_{\mathcal C}=\bigl(\mathrm{col}(U_1)\cdots \mathrm{col}(U_s)\bigr)^{-1}.
\]
Thus \(\beta\) is the ordered product of the labels of the relative lower outer blocks of \(q\)
lying between \(T_l'\) and \(T_l\), while \(b\) is the inverse of the ordered product of the
labels of the relative upper outer blocks of \(q\) lying between \(T_l'\) and \(T_l\).
Then
\[
\prod_{\mathcal C\cap[m]} \vec d
=
\beta^{-1}_{\mathcal C}\,(\iota'_{\mathcal C})^{-1}\, b_{\mathcal C}\, f_{\mathcal C}\,\iota_{\mathcal C}.
\]

\item If \(\mathcal C\) is of through type, let
\(T_r^{\uparrow}\in\mathrm{Th}(p)\) and \(T_r^{\downarrow}\in\mathrm{Th}(q)\)
be, respectively, the rightmost upper and lower through-blocks of \(\mathcal C\).
Let \(h_{\mathcal C}\) be the constant on

$[\max(T_r^{\uparrow}\cap[l])+1,\ \max(\mathcal C\cap[l])],$
and \(\mu_{\mathcal C}\) the constant on

$[\max(T_r^{\downarrow}\cap[l])+1,\ \max(\mathcal C\cap[l])],$
as in Corollary~\ref{cor: constant interval}. Write
\[
t_{\mathcal C}:=t_{V_r^{\uparrow}},
\qquad
\iota_{\mathcal C}:=\iota_{V_r^{\downarrow}}.
\]
Then
\[
\Bigl(\prod_{j=1}^{\max(\mathcal C\cap[k])} r_j\Bigr)(t_{\mathcal C} h_{\mathcal C} \mu^{-1}_{\mathcal C}\iota_{\mathcal C})
=
\prod_{j=1}^{\max(\mathcal C\cap[m])} d_j.
\]

\item If \(\mathcal C\) is of upper trivial type (resp. lower trivial type), then
\(\vec r\) (resp. \(\vec d\)) satisfies the condition induced by the corresponding
block of \(p\) (resp. of \(q\)).
\end{enumerate}
\end{corollary}

\begin{definition}\label{def: colour}
Let $(p,\vec{t} )\in \NC_{\Lambda}(k,l)$ and $(q,\vec{s})\in \NC_{\Lambda}(l,m)$ with
$T_{(p,\vec{t} )} \circ T_{(q,\vec{z})} \neq 0$. For every no-cycle $(p,q)$-connected component $\Ccal$ we define its label(colour) by
$$
\mathrm{lab}(\Ccal)
:=
\begin{cases}
t_{\Ccal}\cdot f_{\Ccal}^{-1}\cdot g^{-1}_{\mathcal C}\cdot (t_{\Ccal}')^{-1}\cdot \alpha_{\mathcal C},
& \text{if $\Ccal$ is of upper-half type},\\[0.3em]
\beta^{-1}_{\mathcal C}\cdot (\iota_{\Ccal}')^{-1}\cdot b_{\mathcal C}\cdot f_{\Ccal}\cdot \iota_{\Ccal},
& \text{if $\Ccal$ is of lower-half type},\\[0.3em]
t_{\Ccal}\,h_{\Ccal}\,\mu_{\Ccal}^{-1}\,\iota_{\Ccal},
& \text{if $\Ccal$ is of through type},\\[0.3em]
\text{label of the corresponding block in $p$},
& \text{if $\Ccal$ is of upper trivial type},\\[0.3em]
\text{label of the corresponding block in $q$},
& \text{if $\Ccal$ is of lower trivial type}.
\end{cases}
$$
Here $f_{\Ccal}$, $h_{\Ccal}$ and $\mu_{\Ccal}$ are as in Corollary~\ref{cor: constant interval}, while $t_{\Ccal}$, $t_{\Ccal}'$, $g$, $\alpha$, $\beta$ and $b$ are as in Corollary~\ref{colour of components}.
\end{definition}

\begin{remark}\label{rem:label-represented-by-middle-vector}
Whenever
\[
\vec x\in \bigcup_{\vec r,\vec d}\bigl(\theta_p^{\vec r}\cap\Omega_q^{\vec d}\bigr),
\]
the factors \(f_{\Ccal}, g, b, h_{\Ccal}, \mu_{\Ccal}\) appearing in
Definition~\ref{def: colour} may be replaced by the corresponding interval
products of \(\vec x\), where the relevant intervals are exactly those given in
Corollary~\ref{cor: constant interval} and Corollary~\ref{colour of components}.
In particular, the label \(\mathrm{lab}(\Ccal)\) may be computed using any such
vector \(\vec x\), and the resulting value is independent of the choice of
\(\vec x\).
\end{remark}

 \begin{definition}
 We now define a new coloured   non-crossing partition $(q\cdot p,col_{q\cdot p})$ on
$[k]\sqcup[m]$ as follows. Let $\{\Ccal_\alpha\}$ be the family of
$(p,q)$-connected components. For each $\alpha$ set
$
B_\alpha
\;:=\;
\Ccal_\alpha\cap\bigl([k]\sqcup[m]\bigr).
$
We let $q\cdot p$ be the collection of all non-empty sets $B_\alpha$:
$$
q\cdot p
\;:=\;
\{\,B_\alpha\subseteq [k]\sqcup[m] : \Ccal_\alpha \text{ a $(p,q)$-connected
component},\ B_\alpha\neq\emptyset\,\}.
$$
 The colouring  is induced from the labels of
$(p,q)$-connected components defined in Definition~6.32: if
$B=\Ccal\cap([k]\sqcup[m])$ for a $(p,q)$-connected component $\Ccal$, then we set $col_{q\cdot p}(B):= \mathrm{lab}(\Ccal),
$ where $\mathrm{lab}(\Ccal)\in\Lambda$ is as in Definition    \ref{colour of components}.
\end{definition}

 \begin{lemma}\label{lem:vertical-composition-NCgh}
Let $\vec g\in \Gamma^k$, $\vec h'\in \Gamma^l$, and $\vec h\in \Gamma^m$.
If $p\in \NC(\vec g,\vec h')$ and $q\in \NC(\vec h',\vec h)$, then
$q\cdot p\in \NC(\vec g,\vec h)$.
\end{lemma}

\begin{proof}
This is a special case of the argument in the proof of
\cite[Theorem~4.11]{FS18}, using the composition formula of
\cite[Proposition~4.6]{FS18}, specialised to $C=\NC$.
\end{proof}

 \begin{proposition}
Suppose $(p,col_p)\in \NC_{\Lambda}(k,l)$ and $(q,col_q)\in \NC_{\Lambda}(l,m)$ with
$T_q \circ T_p \neq 0$. Then $(q\cdot p,col_{q\cdot p})\in NC_{\Lambda}(k,m)$.
\end{proposition}

\begin{proof}

\textbf{Case 1.}
Suppose that $\mathcal T_0$ is the rightmost through connected component, in the sense that $\mathcal T_0$ intersects $[k]$, $[l]$, and $[m]$, and for every other connected component $\mathcal T$ intersecting all three intervals, $\max(\mathcal T\cap[l])\le \max(\mathcal T_0\cap[l]).$
Since through-blocks admit no nesting, noncrossingness implies that this is equivalent to $\max(\mathcal T\cap[j])\le \max(\mathcal T_0\cap[j])(j=k,m).$

\medskip
\noindent\textbf{Step 1: Outer upper-half and lower-half components to the right of $\mathcal T_0$.}
Denote by $\mathcal C_1,\dots,\mathcal C_{u_n}$ (resp.\ $\mathcal O_1,\dots,\mathcal O_{d_n}$)
all outer upper-half (resp.\ lower-half) connected components to the right of $\mathcal T_0$,
i.e.\ all upper-half connected components $\mathcal C$ (resp.\ lower-half connected
components $\mathcal O$) such that there is no distinct upper-half connected component
$\mathcal C'$ (resp.\ lower-half connected component $\mathcal O'$) with
$F_{\mathcal C}\subseteq F_{\mathcal C'}$ (resp.\ $F_{\mathcal O}\subseteq F_{\mathcal O'}$),
and such that
\[
\max(\mathcal T_0\cap[l])<\max(\mathcal C\cap[l])
\quad
(\text{resp.\ }\max(\mathcal T_0\cap[l])<\max(\mathcal O\cap[l])).
\]
List them so that $\max(\mathcal C_i\cap[l])<\max(\mathcal C_{i+1}\cap[l]) (1\le i\le u_n-1),$ and $\max(\mathcal O_j\cap[l])<\max(\mathcal O_{j+1}\cap[l])(1\le j\le d_n-1).$

\medskip
\noindent\textbf{Step 2: Labels of the outer upper(lower)-half components.}
(a) For each $1\le i\le u_n$, let $T_l(\mathcal C_i),T_r(\mathcal C_i)$ be respectively
the leftmost and rightmost upper through-blocks of $\mathcal C_i$, and let
$T_r^{\uparrow}(\mathcal T_0)$ be the rightmost upper through-block of $\mathcal T_0$. Write $t_i:=t_{T_r(\mathcal C_i)}(1\le i\le u_n),t_0:=t_{T_r^{\uparrow}(\mathcal T_0)}.$
Also write $f_i:=f_{\mathcal C_i}(1\le i\le u_n),$ where $f_{\mathcal C_i}$ is the constant attached to the frame $F_{\mathcal C_i}$ in Corollary~\ref{cor: constant interval}.

For each $0\le i\le u_n-1$, set  $g_i:=\mathcal T^\downarrow\bigl(\mathcal P_{p}^{\downarrow}(T_l(\mathcal C_{i+1}))\bigr), \alpha_i:=\mathcal T^\uparrow\bigl(\mathcal P_{p}^{\uparrow}(T_l(\mathcal C_{i+1}))\bigr).$
Thus $g_i$ and $\alpha_i$ are exactly the lower and upper boundary contributions,
respectively, attached to the leftmost upper through-block $V_l(\mathcal C_{i+1})$.
Equivalently, if

$\partial^\downarrow\bigl(\mathcal P_{p}^{\downarrow}(T_l(\mathcal C_{i+1}))\bigr)
=\{W_1\prec\cdots\prec W_s\},$ and  $\partial^\uparrow\bigl(\mathcal P_{p}^{\uparrow}(T_l(\mathcal C_{i+1}))\bigr)
=\{U_1\prec\cdots\prec U_t\},$
then
\[
g_i=\mathrm{col}(W_1)\cdots \mathrm{col}(W_s),
\qquad
\alpha_i=\bigl(\mathrm{col}(U_1)\cdots \mathrm{col}(U_t)\bigr)^{-1}.
\]

We claim that $T_l(\mathcal C_{i+1})'=T_r(\mathcal C_i)(1\le i\le u_n-1),$
and also $T_l(\mathcal C_1)'=T_r^{\uparrow}(\mathcal T_0).$  Indeed, since $\mathcal C_i$ and $\mathcal C_{i+1}$ are outer upper-half connected components,
their frames are disjoint, hence $\max F_{\mathcal C_i}<\min F_{\mathcal C_{i+1}}.$
Equivalently, $\max\bigl(T_r(\mathcal C_i)\cap[l]\bigr)
<
\min\bigl(T_l(\mathcal C_{i+1})\cap[l]\bigr).$
Thus $T_r(\mathcal C_i)\lhd T_l(\mathcal C_{i+1}).$ If there were a through-block $W\in\mathrm{Th}(p)$ such that $T_r(\mathcal C_i)\lhd W\lhd T_l(\mathcal C_{i+1}),$
then the connected component $\mathcal D$ containing $W$ would be an upper-half connected
component lying to the right of $\mathcal T_0$, and its frame would lie strictly between
$F_{\mathcal C_i}$ and $F_{\mathcal C_{i+1}}$.
If $\mathcal D$ is outer upper-half, this contradicts the fact that
$\mathcal C_i$ and $\mathcal C_{i+1}$ are consecutive in the ordered family
$\{\mathcal C_1,\dots,\mathcal C_{u_n}\}$.
If $\mathcal D$ is not outer, let $\mathcal E$ be the outer upper-half connected component
with $F_{\mathcal D}\subseteq F_{\mathcal E}$; then $F_{\mathcal E}$ still lies strictly
between $F_{\mathcal C_i}$ and $F_{\mathcal C_{i+1}}$, again a contradiction.

Therefore there is no through-block strictly between
$T_r(\mathcal C_i)$ and $T_l(\mathcal C_{i+1})$ in the order $\lhd$.
By Definition~\ref{gap}, the left-adjacent through-block of $T_l(\mathcal C_{i+1})$
is exactly $T_r(\mathcal C_i)$, that is, $T_l'(\mathcal C_{i+1})=T_r(\mathcal C_i).$ The proof of $T_l'(\mathcal C_1)=T_r^{\uparrow}(\mathcal T_0)$
is identical.

Hence, by Corollary~\ref{colour of components}(1) and Definition~\ref{def: colour}, $\mathrm{lab}(\mathcal C_i)
=t_i\,f_i^{-1}\,g_{i-1}^{-1}\,t_{i-1}^{-1}\,\alpha_{i-1}(1\le i\le u_n).$
In particular, $\mathrm{lab}(\mathcal C_{u_n})=t_{u_n}f_{u_n}^{-1}g_{u_n-1}^{-1}t_{u_n-1}^{-1}\alpha_{u_n-1},\dots,\mathrm{lab}(\mathcal C_1)=t_1f_1^{-1}g_0^{-1}t_0^{-1}\alpha_0.$ 

(b) The lower-half case is entirely analogous. For each $1\le j\le d_n$, let
$T_l(\mathcal O_j),T_r(\mathcal O_j)$ be respectively the leftmost and rightmost
lower through-blocks of $\mathcal O_j$, and let
$T_r^{\downarrow}(\mathcal T_0)$ be the rightmost lower through-block of $\mathcal T_0$.
Write $\iota_j:=\iota_{T_r(\mathcal O_j)}\ (1\le j\le d_n),\iota_0:=\iota_{T_r^{\downarrow}(\mathcal T_0)},$
and let $\phi_j:=f_{\mathcal O_j}$.

For each $0\le j\le d_n-1$, define $\beta_j:=\mathcal T^\downarrow\bigl(\mathcal P_{q}^{\downarrow}(T_l(\mathcal O_{j+1}))\bigr), b_j:=\mathcal T^\uparrow\bigl(\mathcal P_{q}^{\uparrow}(T_l(\mathcal O_{j+1}))\bigr).$
Then, by the same argument as in {\rm(a)}, $T_l'(\mathcal O_{j+1})=T_r(\mathcal O_j)\quad(1\le j\le d_n-1),T_l'(\mathcal O_1)=T_r^{\downarrow}(\mathcal T_0),$
and therefore, by Corollary~\ref{colour of components}(2) and
Definition~\ref{def: colour}, $\mathrm{lab}(\mathcal O_j)
=\beta_{j-1}^{-1}\,\iota_{j-1}^{-1}\,b_{j-1}\,\phi_j\,\iota_j(1\le j\le d_n).$

\medskip
\noindent\textbf{Step 3}
Notice that, by Corollary~\ref{colour of components} and Definition~\ref{def: colour},
$(\alpha_i)^{-1}$ is exactly the ordered product of the labels of the upper boundary
blocks of $q\cdot p$ induced by trivial connected components lying between
$\mathcal C_i$ and $\mathcal C_{i-1}$ in the boundary order, while
$\beta_i$ is exactly the ordered product of the labels of the lower boundary
blocks of $q\cdot p$ induced by trivial connected components lying between
$\mathcal O_{i-1}$ and $\mathcal O_i$ in the boundary order. Moreover,
\[
\mathrm{lab}(\mathcal T_0)=t_0h_0\mu_0^{-1}\iota_0.
\]

Let $\gamma_1$ (resp.\ $\gamma_2$) be the ordered product of the labels of all boundary
blocks of $q\cdot p$ strictly larger (resp.\ strictly smaller) than
$\mathcal C_{u_n}$ (resp.\ $\mathcal O_{d_n}$) in the boundary order. Then the full
boundary product can be written as
\begin{align*}
\prod_{\partial(q\cdot p)}^{\prec}\mathrm{col}
&=
\gamma_1^{-1}\,
\mathrm{lab}(\mathcal C_{u_n})\,(\alpha_{u_n-1})^{-1}\,
\mathrm{lab}(\mathcal C_{u_n-1})\cdots
(\alpha_1)^{-1}\,\mathrm{lab}(\mathcal C_1)\,(\alpha_0)^{-1} \\
&\qquad\cdot\,
\mathrm{lab}(\mathcal T_0)\,
\beta_0\,\mathrm{lab}(\mathcal O_1)\,
\beta_1\cdots
\mathrm{lab}(\mathcal O_{d_n-1})\,\beta_{d_n-1}\,
\mathrm{lab}(\mathcal O_{d_n})\,\gamma_2 .
\end{align*}
Substituting the expressions for the labels gives
\begin{align*}
\prod_{\partial(q\cdot p)}^{\prec}\mathrm{col}
&=
\gamma_1^{-1}\,
\bigl(t_{u_n}f_{u_n}^{-1}g_{u_n-1}^{-1}t_{u_n-1}^{-1}\alpha_{u_n-1}\bigr)
(\alpha_{u_n-1})^{-1}
\bigl(t_{u_n-1}f_{u_n-1}^{-1}g_{u_n-2}^{-1}t_{u_n-2}^{-1}\alpha_{u_n-2}\bigr)
\cdots \\
&\qquad\cdots
(\alpha_0)^{-1}\,
\bigl(t_0h_0\mu_0^{-1}\iota_0\bigr)\,
\beta_0\,
\bigl(\beta_0^{-1}\iota_0^{-1}b_0\phi_1\iota_1\bigr)\,
\beta_1\cdots
\bigl(\beta_{d_n-1}^{-1}\iota_{d_n-1}^{-1}b_{d_n-1}\phi_{d_n}\iota_{d_n}\bigr)\,
\gamma_2 \\
&=
\gamma_1^{-1}\,
t_{u_n}f_{u_n}^{-1}g_{u_n-1}^{-1}f_{u_n-1}^{-1}\cdots
f_1^{-1}g_0^{-1}h_0\mu_0^{-1}
b_0\phi_1b_1\phi_2\cdots b_{d_n-1}\phi_{d_n}\iota_{d_n}\,
\gamma_2.
\end{align*}

\medskip \noindent\textbf{Step 4: } 
Set $w:=\max\bigl(T_r^{\uparrow}(\mathcal T_0)\cap[l]\bigr).$
For $1\le i\le u_n$, put $a_i:=\min\bigl(T_l(\mathcal C_i)\cap[l]\bigr), b_i:=\max\bigl(T_r(\mathcal C_i)\cap[l]\bigr),$
and set $b_0:=w$. Then $[w+1,l]
=G_0\sqcup F_{\mathcal C_1}\sqcup G_1
\sqcup F_{\mathcal C_2}
\sqcup\cdots\sqcup
G_{u_n-1}
\sqcup F_{\mathcal C_{u_n}}
\sqcup G_{u_n},$
where $G_i:=[\,b_i+1,a_{i+1}-1\,](0\le i\le u_n-1),$ and $G_{u_n}:=[\,b_{u_n}+1,l\,]$. Here, as usual, the product over an empty interval is understood to be $1$.
By Definition\ref{def:frame}, $F_{\mathcal C_i}=[\,a_i,b_i\,](1\le i\le u_n).$

Let $\vec s\in\theta_p^{\vec r}\cap\Omega_q^{\vec d}$.
By the definition of $g_i$ and the adjacency relations $T_l(\mathcal C_{i+1})'=T_r(\mathcal C_i)
(1\le i\le u_n-1),T_l(\mathcal C_1)'=T_r^{\uparrow}(\mathcal T_0),$
we have $\prod_{I_i}\vec s=g_i
(0\le i\le u_n-1).$
On the other hand, by Corollary~\ref{cor: constant interval}, one has
$\prod_{F_{\mathcal C_i}}\vec s=f_i(1\le i\le u_n).$
Finally, set $h:=\prod_{G_{u_n}}\vec s=
\prod_{[\,b_{u_n}+1,l\,]}\vec s.$
Therefore,,
$\prod_{[w+1,l]}\vec s=
g_0\,f_1\,g_1\,f_2\,\cdots\,g_{u_n-1}\,f_{u_n}\,h.$

The proof for the interval $[w'+1,l]$ is identical, with upper and lower halves
exchanged. More precisely, using the notation from Step 2(b), set
$\mu:=\prod_{[\,\max(T_r(\mathcal O_{d_n})\cap[l])+1,l\,]}\vec s .$

Then, for $\vec s\in\theta_p^{\vec r}\cap\Omega_q^{\vec d}$, we have $\prod_{[w'+1,l]}\vec s
=b_0\,\phi_1\,b_1\,\phi_2\cdots b_{d_n-1}\,\phi_{d_n}\,\mu .$

\medskip \noindent\textbf{Step 5: Comparing the products} Let  $x:=\max(\mathcal T_0\cap[k]),x':=\max(\mathcal T_0\cap[m]).$ By Corollary~\ref{colour of components}(3), applied to the through component $\mathcal T_0$, we have  $\Bigl(\prod_{j=1}^{x} r_j\Bigr)\,t_0h_0\mu_0^{-1}\iota_0 = \prod_{j=1}^{x'} d_j,$ equivalently, $\Bigl(\prod_{j=1}^{x} r_j\Bigr)\,t_0h_0\mu_0^{-1} = \Bigl(\prod_{j=1}^{x'} d_j\Bigr)\,\iota_0^{-1},$ for every $\vec r,\vec d$ such that $\theta_p^{\vec r}\cap\Omega_q^{\vec d}\neq\varnothing$. On the other hand, since $[w+1,l]$ is partitioned by the frames $F_{\mathcal C_1},\dots,F_{\mathcal C_{u_n}}$ together with the spans of the lower outer blocks of $p$ between them, and since $[w'+1,l]$ is partitioned by the frames $F_{\mathcal O_1},\dots,F_{\mathcal O_{d_n}}$ together with the spans of the upper outer blocks of $q$ between them, we obtain, for every $\vec s\in\theta_p^{\vec r}\cap\Omega_q^{\vec d}$, \[\Bigl(\prod_{j=1}^{x} r_j\Bigr)\,t_0\,g_0\,f_1\,g_1\,f_2\cdots g_{u_n-1}f_{u_n}\,h = \prod \vec s_{\restriction[1,l]} = \Bigl(\prod_{j=1}^{x'} d_j\Bigr)\,\iota_0^{-1}\,b_0\,\phi_1\,b_1\,\phi_2\cdots b_{d_n-1}\phi_{d_n}\,\mu. \] 

Combining the two displayed equalities, we get  $f_{u_n}^{-1}g_{u_n-1}^{-1}f_{u_n-1}^{-1}\cdots f_1^{-1}g_0^{-1}\,h_0\,\mu_0^{-1}\, b_0\,\phi_1\,b_1\,\phi_2\cdots b_{d_n-1}\phi_{d_n} = h\,\mu^{-1}.$

\noindent\textbf{Step 5: Conclusion via the boundary conditions of \iffalse$p$ and $q$.\fi}
Therefore the ordered product of the labels of all boundary blocks of $q\cdot p$ is $\gamma_1^{-1}\,t_{u_n}\,h\,\mu^{-1}\,\iota_{d_n}\,\gamma_2.$

By the boundary condition for $(p,\operatorname{col}_p)$, we have $\gamma_1^{-1}\,t_{u_n}\,h=1,$ and by the boundary condition for $(q,\operatorname{col}_q)$, we have $\mu^{-1}\,\iota_{d_n}\,\gamma_2=1.$
Therefore $\gamma_1^{-1}\,t_{u_n}\,h\,\mu^{-1}\,\iota_{d_n}\,\gamma_2=1.$
Thus the ordered product of the labels of all boundary blocks of $q\cdot p$ is equal to $1$.
Hence $(q\cdot p,\operatorname{col}_{q\cdot p})\in \NC_\Lambda(k,m)$.

\medskip
\noindent\textbf{Case 2.} Assume that there is no through connected component.

The proof is identical to that of \textbf{Case 1}, except that all quantities attached to the rightmost through connected component lose their meaning and are replaced by $1$. In particular, the identity $\Bigl(\prod_{j=1}^{x} r_j\Bigr)\,t_0h_0\mu_0^{-1}\iota_0=\prod_{j=1}^{x'} d_j$
is replaced by the trivial identity $1=1$.

Likewise, in \textbf{Step 5}, the quantities $w$ and $w'$ are no longer defined. The corresponding argument is carried out directly on the whole middle interval $[1,l]$ (instead of the tails $[w+1,l]$ and $[w'+1,l]$). All remaining steps are unchanged.
\end{proof}

\begin{corollary}

Let $\vec{g} \in \Lambda^k$, ${\vec{h}}^{\prime}, \in \Lambda^l$, and $\vec h, \in \Lambda^m$. If $(p,col_p) \in \NC_{\Lambda}(\vec{g}, {\vec{h}}^{\prime})$ and $(q ,col_q)\in \NC_{\Lambda}({\vec{h}}^{\prime}, {\vec h})$, then their composition $(q \cdot p, col_{q\cdot p}  ) \in\NC_{\Lambda}(\vec{g}, \vec h)$.

\end{corollary}

\subsection{Stability under Vertical Composition}

\begin{lemma}\label{Equation}
Let $p\in \NC_{\Lambda}(k,l)$ and $q\in \NC_{\Lambda}(l,m)$ with
$T_q\circ T_p\neq 0$. Then $\theta_p^{\vec r}\cap\Omega_q^{\vec d}\neq\varnothing$
if and only if the system $\mathcal{E}_{\vec r,\vec d}^{p,q}$ admits a solution.
, where the $\mathcal{E}_{\vec r,\vec d}^{p,q}$ is defined as follows:
{\small
\[
\mathcal{E}_{\vec r,\vec d}^{p,q}:=
\left\{
\begin{aligned}
\prod_{i=\min V^{-}}^{\max V^{-}} x_i
&= g_V^{-1}(t_{V'})^{-1}\alpha_V
   \Bigl(\prod_{i=\min V^{-}}^{\max V^{-}} r_i\Bigr)t_V
:= \kappa_{\vec r,\vec d}(V)
\quad (V\in\mathrm{Th}(p)),\\[-0.2em]
\prod_{i=\min V^{+}}^{\max V^{+}} x_i
&= b_V^{-1}\iota_{V'}\beta_V
   \Bigl(\prod_{i=\min V^{+}}^{\max V^{+}} d_i\Bigr)(\iota_V)^{-1}
:= \kappa_{\vec r,\vec d}(V)
\quad (V\in\mathrm{Th}(q)),\\[-0.2em]
\prod_{i=\min V}^{\max V} x_i
&= \mathrm{col}_q(V^+)^{-1}:=\kappa_{\vec r,\vec d}(V)
\quad (\text{$V$ upper single-layer block of $q$}),\\[-0.2em]
\prod_{i=\min V}^{\max V} x_i
&= \mathrm{col}_p(V^{-}):=\kappa_{\vec r,\vec d}(V)
\quad (\text{$V$ lower single-layer block of $p$}).
\end{aligned}
\right.
\]}
 For \(V\in \mathrm{Th}(p)\) (resp. \(V\in \mathrm{Th}(q)\)), \(t_V\) (resp. \(\iota_V\)) denotes its label, and \(t_{V'}\) (resp. \(\iota_{V'}\)) denotes the label of its left-adjacent through-block \(V'\). For $V\in\mathrm{Th}(p)$, $g_V=\mathcal{T}^{\downarrow}(\mathcal{P}_{p}^{\downarrow}(V))$, $\alpha_V=\mathcal T^{\uparrow}(\mathcal{P}_{p}^{\uparrow}(V))$. For  $V\in\mathrm{Th}(q)$, $\beta_V=\mathcal{T}^{\downarrow}(\mathcal{P}_{q}^{\downarrow}(V))$, $b_V=\mathcal T^{\uparrow}(\mathcal{P}_{q}^{\uparrow}(V))$.
 
\end{lemma}

\begin{proof}
By Definition~\ref{block relations}, the condition
$\theta_p^{\vec r}\cap\Omega_q^{\vec d}\neq\varnothing$
is equivalent to the existence of $\vec x\in\Lambda^l$ satisfying all block
relations induced by $(p,\vec r)$ and $(q,\vec d)$ on the middle level $[l]$.
We rewrite these block relations in terms of the interval products introduced above.
Throughout, the product over an empty interval is understood to be $e$.

\medskip

\noindent
\textbf{Through-blocks of $p$.}
Let $V\in\mathrm{Th}(p)$, and let $V'$ be its left-adjacent through-block, if it exists.
If $V'$ does not exist, we use the convention   $\ell(V)=0,t_{V'}=e.$
Otherwise, set $\ell(V):=\max(V'\cap[l]).$
The block relations for $V$ and $V'$ give $\prod_{i=1}^{\max V^-}x_i
=\Bigl(\prod_{i=1}^{\max V^+}r_i\Bigr)t_V,$
$\prod_{i=1}^{\ell(V)}x_i
=
\Bigl(\prod_{i=1}^{\max (V')^+}r_i\Bigr)t_{V'}.$
Dividing these two relations yields $\prod_{i=L(V)+1}^{\max V^-}x_i=(t_{V'})^{-1}
\Bigl(\prod_{i=\max (V')^+ +1}^{\max V^+}r_i\Bigr)t_V.$
By the definitions of $g_V$ and $\alpha_V$, this becomes $g_V\prod_{V^-}\vec x
=(t_{V'})^{-1}\alpha_V
\Bigl(\prod_{V^+}\vec r\Bigr)t_V.$
Hence $\prod_{V^-}\vec x
=g_V^{-1}(t_{V'})^{-1}\alpha_V
\Bigl(\prod_{V^+}\vec r\Bigr)t_V.$

\medskip

\noindent
\textbf{Through-blocks of $q$.}
The proof is identical, with the roles of upper and lower rows exchanged, and with
\[
(p,\vec r,t_V,g_V,\alpha_V,V^-,V^+)
\quad\text{replaced by}\quad
(q,\vec d,\iota_V,b_V,\beta_V,V^+,V^-).
\]
Thus, for every $V\in\mathrm{Th}(q)$, one obtains $\prod_{V^+}\vec x
=b_V^{-1}\iota_{V'}\beta_V
\Bigl(\prod_{V^-}\vec d\Bigr)(\iota_V)^{-1}.$
\noindent
\textbf{Single-layer blocks.}
Finally, if $V$ is an upper single-layer block of $q$, respectively a lower
single-layer block of $p$, the block relations give directly
\[
\prod_V\vec x=\mathrm{col}(V)^{-1},
\qquad\text{respectively}\qquad
\prod_V\vec x=\mathrm{col}(V).
\]

Collecting the through-block and single-layer relations gives exactly the system
$\mathcal E_{\vec r,\vec d}^{p,q}$. Therefore
$\theta_p^{\vec r}\cap\Omega_q^{\vec d}\neq\varnothing$
if and only if $\mathcal E_{\vec r,\vec d}^{p,q}$ admits a solution.
\end{proof}

\begin{definition}\label{(p^{-}, q^{+})}
\[
p^{-}:=\{\,V^{-}: V\in p,\ V^{-}\neq\varnothing\,\},
\qquad
q^{+}:=\{\,V^{+}: V\in q,\ V^{+}\neq\varnothing\,\}.
\]

\end{definition}
\begin{tikzpicture}[
    x=0.42cm,
    y=0.62cm,
    pt/.style={circle,fill=black,inner sep=0.9pt},
    toplab/.style={font=\scriptsize},
    midlab/.style={font=\tiny},
    rowlab/.style={font=\scriptsize},
    gcomp/.style={
        draw=teal!70!black,
        line width=0.75pt,
        line cap=round,
        line join=round,
        rounded corners=1.4pt
    },
    bcomp/.style={
        draw=blue!75!black,
        line width=0.75pt,
        line cap=round,
        line join=round,
        rounded corners=1.4pt
    },
    rcomp/.style={
        draw=red!75!black,
        line width=0.75pt,
        line cap=round,
        line join=round,
        rounded corners=1.4pt
    },
    pcomp/.style={
        draw=violet!75!black,
        line width=0.75pt,
        line cap=round,
        line join=round,
        rounded corners=1.4pt
    }
]

\newcommand{\usingle}[2]{%
  \draw[#1]
    (#2-0.14,0.03)
      .. controls (#2-0.14,0.18) and (#2+0.14,0.18) ..
    (#2+0.14,0.03);
}
\newcommand{\lsingle}[2]{%
  \draw[#1]
    (#2-0.14,-0.03)
      .. controls (#2-0.14,-0.18) and (#2+0.14,-0.18) ..
    (#2+0.14,-0.03);
}

\foreach \i in {1,...,24}{
    \coordinate (M\i) at (\i,0);
}

\node[rowlab] at (-0.8,0) {$[l]$};

\node[font=\large] at (12.5, 1.95) {$p^{-}$};
\node[font=\large] at (12.5,-2.05) {$q^{+}$};


\usingle{gcomp}{1}

\draw[gcomp] (M3)  -- ($(M3)+(0,1.18)$);
\draw[gcomp] (M7)  -- ($(M7)+(0,1.18)$);
\draw[gcomp] (M12) -- ($(M12)+(0,1.18)$);
\draw[gcomp] ($(M3)+(0,1.18)$) -- ($(M12)+(0,1.18)$);

\draw[gcomp] (M8)  -- ($(M8)+(0,0.78)$);
\draw[gcomp] (M10) -- ($(M10)+(0,0.78)$);
\draw[gcomp] (M11) -- ($(M11)+(0,0.78)$);
\draw[gcomp] ($(M8)+(0,0.78)$) -- ($(M11)+(0,0.78)$);

\draw[gcomp] (M15) -- ($(M15)+(0,0.96)$);
\draw[gcomp] (M18) -- ($(M18)+(0,0.96)$);
\draw[gcomp] ($(M15)+(0,0.96)$) -- ($(M18)+(0,0.96)$);

\draw[gcomp] (M16) -- ($(M16)+(0,0.66)$);
\draw[gcomp] (M17) -- ($(M17)+(0,0.66)$);
\draw[gcomp] ($(M16)+(0,0.66)$) -- ($(M17)+(0,0.66)$);

\draw[bcomp] (M2)  -- ($(M2)+(0,1.45)$);
\draw[bcomp] (M13) -- ($(M13)+(0,1.45)$);
\draw[bcomp] ($(M2)+(0,1.45)$) -- ($(M13)+(0,1.45)$);

\draw[bcomp] (M4) -- ($(M4)+(0,0.70)$);
\draw[bcomp] (M6) -- ($(M6)+(0,0.70)$);
\draw[bcomp] ($(M4)+(0,0.70)$) -- ($(M6)+(0,0.70)$);

\usingle{bcomp}{9}

\coordinate (Pbot14) at ($(M14)+(0,1.18)$);
\coordinate (Pbot19) at ($(M19)+(0,1.18)$);
\coordinate (Pbot24) at ($(M24)+(0,1.18)$);

\draw[bcomp] (M14) -- (Pbot14);
\draw[bcomp] (M19) -- (Pbot19);
\draw[bcomp] (M24) -- (Pbot24);
\draw[bcomp] (Pbot14) -- (Pbot24);

\draw[bcomp] (M21) -- ($(M21)+(0,0.66)$);
\draw[bcomp] (M22) -- ($(M22)+(0,0.66)$);
\draw[bcomp] ($(M21)+(0,0.66)$) -- ($(M22)+(0,0.66)$);

\usingle{rcomp}{5}

\draw[pcomp] (M20) -- ($(M20)+(0,0.90)$);
\draw[pcomp] (M23) -- ($(M23)+(0,0.90)$);
\draw[pcomp] ($(M20)+(0,0.90)$) -- ($(M23)+(0,0.90)$);


\draw[gcomp] (M1) -- ($(M1)+(0,-1.28)$);
\draw[gcomp] (M7) -- ($(M7)+(0,-1.28)$);
\draw[gcomp] (M8) -- ($(M8)+(0,-1.28)$);
\draw[gcomp] ($(M1)+(0,-1.28)$) -- ($(M8)+(0,-1.28)$);

\lsingle{gcomp}{3}

\draw[gcomp] (M10) -- ($(M10)+(0,-1.12)$);
\draw[gcomp] (M16) -- ($(M16)+(0,-1.12)$);
\draw[gcomp] ($(M10)+(0,-1.12)$) -- ($(M16)+(0,-1.12)$);

\draw[gcomp] (M11) -- ($(M11)+(0,-0.90)$);
\draw[gcomp] (M12) -- ($(M12)+(0,-0.90)$);
\draw[gcomp] (M15) -- ($(M15)+(0,-0.90)$);
\draw[gcomp] ($(M11)+(0,-0.90)$) -- ($(M15)+(0,-0.90)$);

\draw[gcomp] (M17) -- ($(M17)+(0,-0.66)$);
\draw[gcomp] (M18) -- ($(M18)+(0,-0.66)$);
\draw[gcomp] ($(M17)+(0,-0.66)$) -- ($(M18)+(0,-0.66)$);

\draw[bcomp] (M2) -- ($(M2)+(0,-0.78)$);
\draw[bcomp] (M4) -- ($(M4)+(0,-0.78)$);
\draw[bcomp] (M6) -- ($(M6)+(0,-0.78)$);
\draw[bcomp] ($(M2)+(0,-0.78)$) -- ($(M6)+(0,-0.78)$);

\draw[bcomp] (M9)  -- ($(M9)+(0,-1.45)$);
\draw[bcomp] (M21) -- ($(M21)+(0,-1.45)$);
\draw[bcomp] ($(M9)+(0,-1.45)$) -- ($(M21)+(0,-1.45)$);

\draw[bcomp] (M13) -- ($(M13)+(0,-0.62)$);
\draw[bcomp] (M14) -- ($(M14)+(0,-0.62)$);
\draw[bcomp] ($(M13)+(0,-0.62)$) -- ($(M14)+(0,-0.62)$);

\lsingle{bcomp}{19}

\draw[bcomp] (M22) -- ($(M22)+(0,-0.66)$);
\draw[bcomp] (M24) -- ($(M24)+(0,-0.66)$);
\draw[bcomp] ($(M22)+(0,-0.66)$) -- ($(M24)+(0,-0.66)$);

\lsingle{rcomp}{5}
\lsingle{pcomp}{20}
\lsingle{pcomp}{23}

\draw[gcomp,decorate,decoration={brace,mirror,amplitude=2.8pt}]
    (7,-1.62) -- (8,-1.62);
\node[font=\tiny] at (7.5,-1.88) {entrance $\{7,8\}$};

\foreach \i in {1,...,24}{
    \node[pt] at (M\i) {};
    \node[midlab,below=1.5pt] at (M\i) {\(\i\)};
}

\end{tikzpicture}

\begin{definition}[The graph associated with $(p^{-},q^{+})$]
Let $V_{p,q}:=\{0,1,\dots,l\}$. For $B\in p^{-}\sqcup q^{+}$, set
$\ell(B):=\min B-1,\ r(B):=\max B$, and let $e_B$ be the unoriented edge with
endpoint set $\partial e_B:=\{\ell(B),r(B)\}$. Set
\[
\overline G_{p,q}:=(V_{p,q},\overline E_{p,q}),
\qquad
\overline E_{p,q}:=\{\,e_B:B\in p^{-}\sqcup q^{+}\,\}.
\]

A \emph{path} in $\overline G_{p,q}$ is a sequence
$\rho=(e_{B_1},\dots,e_{B_N})$ such that there exist pairwise distinct vertices
$v_0,\dots,v_N\in V_{p,q}$ with $\partial e_{B_i}=\{v_{i-1},v_i\}(1\le i\le N).$ In particular, the blocks $B_1,\dots,B_N$ are pairwise distinct. A \emph{simple cycle} in $\overline G_{p,q}$ is a sequence
$\gamma=(e_{B_1},\dots,e_{B_N})$ such that there exist vertices
$v_0,\dots,v_N\in V_{p,q}$ with $v_0=v_N$, $v_0,\dots,v_{N-1}$ pairwise distinct,
and $\partial e_{B_i}=\{v_{i-1},v_i\}(1\le i\le N).$

We also use the oriented double
\[
G_{p,q}:=(V_{p,q},E_{p,q},s,t),
\qquad
E_{p,q}:=\{B^{+},B^{-}:B\in p^{-}\sqcup q^{+}\},
\]
where $B^{+}:\ell(B)\to r(B)$ and $B^{-}:r(B)\to\ell(B)$. Thus
\[
s(B^{+})=\ell(B),\quad t(B^{+})=r(B),
\qquad
s(B^{-})=r(B),\quad t(B^{-})=\ell(B).
\]

A \emph{directed lift} of a path
$\rho=(e_{B_1},\dots,e_{B_N})$ is a sequence
$\widetilde\rho=(B_1^{\sigma_1},\dots,B_N^{\sigma_N})$, with
$\sigma_i\in\{+,-\}$, such that $t(B_i^{\sigma_i})=s(B_{i+1}^{\sigma_{i+1}})(1\le i<N).$

If $\rho$ is a simple cycle, we require in addition $t(B_N^{\sigma_N})=s(B_1^{\sigma_1}).$

\end{definition}

\begin{definition}[Gain graph associated to $\mathcal E^{p,q}_{\vec r,\vec d}$]
For each $B\in p^{-}\sqcup q^{+}$, denote by
$\kappa_{\vec r,\vec d}(B)\in\Lambda$ the coefficient prescribed by the corresponding
equation in $\mathcal E^{p,q}_{\vec r,\vec d}$.

Define
$\Phi_{\vec r,\vec d}:E_{p,q}\to\Lambda$ by
\[
\Phi_{\vec r,\vec d}(B^{+}):=\kappa_{\vec r,\vec d}(B),
\qquad
\Phi_{\vec r,\vec d}(B^{-}):=\kappa_{\vec r,\vec d}(B)^{-1}.
\]
Set
\[
\mathcal G(\mathcal E^{p,q}_{\vec r,\vec d}):=(V^{p,q},E^{p,q},s,t,\Phi_{\vec r,\vec d}).
\]
When no confusion is possible, we simply write $\mathcal G=(V,E,s,t,\Phi)$.
For a directed lift
$\widetilde\rho=(B_1^{\sigma_1},\dots,B_N^{\sigma_N})$, define its gain by $\Phi(\widetilde\rho):=\prod_{i=1}^N\Phi(B_i^{\sigma_i}).$
A simple cycle $\gamma$ in $\overline G_{p,q}$ is \emph{balanced} if
$\Phi(\widetilde\gamma)=1_\Lambda$ for one, equivalently every, directed lift
$\widetilde\gamma$ of $\gamma$. We say that $\mathcal G$ is \emph{balanced} if
every simple cycle in $\overline G_{p,q}$ is balanced.

A \emph{potential function} on $\mathcal G$ is a map $\pi:V\to\Lambda$ such that
$\pi(s(e))^{-1}\pi(t(e))=\Phi(e)$ for all $e\in E$. We write
$\mathrm{Pot}_0(\mathcal G)$ for the set of such $\pi$ with $\pi(0)=1_\Lambda$.

For $x\in\{1,\dots,l\}$, say that $B\in p^{-}\sqcup q^{+}$ \emph{crosses} $x$
if $\min B\le x\le\max B$. For
$\gamma=(B_1^{\sigma_1},\dots,B_N^{\sigma_N})$, define
\[
\chi_{p^{-}}^{\gamma}(x)
:=\#\{\,i:\ B_i\in p^{-},\ \min B_i\le x\le\max B_i\,\},
\chi_{q^{+}}^{\gamma}(x)
:=\#\{\,i:\ B_i\in q^{+},\ \min B_i\le x\le\max B_i\,\}.
\]
\end{definition}

\begin{figure}[htbp]
\centering
\begin{tikzpicture}[
    scale=0.93,
    transform shape,
    x=0.37cm,
    y=0.82cm,
    vtx/.style={circle,fill=black,inner sep=0.75pt},
    num/.style={font=\tiny},
    tag/.style={font=\scriptsize},
    gcomp/.style={draw=teal!70!black,line width=0.78pt},
    bcomp/.style={draw=blue!75!black,line width=0.78pt},
    rcomp/.style={draw=red!75!black,line width=0.78pt},
    pcomp/.style={draw=violet!75!black,line width=0.78pt}
]

\foreach \i in {0,...,24}{
    \coordinate (V\i) at (\i,0);
    \node[vtx] at (V\i) {};
}

\draw[black!35] (V0) -- (V24);

\foreach \i in {0,...,24}{
    \pgfmathtruncatemacro{\odd}{mod(\i,2)}
    \ifnum\odd=0
        \node[num,below=1.8pt] at (V\i) {$\i$};
    \else
        \node[num,below=6.2pt] at (V\i) {$\i$};
    \fi
}

\node[tag,left] at (-0.7,1.80) {$p^{-}$};
\node[tag,left] at (-0.7,-1.80) {$q^{+}$};

\newcommand{\uarc}[4]{%
  \draw[#1]
    (V#2) .. controls +(0,#4) and +(0,#4) .. (V#3);
}
\newcommand{\darc}[4]{%
  \draw[#1]
    (V#2) .. controls +(0,-#4) and +(0,-#4) .. (V#3);
}

\uarc{gcomp}{0}{1}{0.32}      
\uarc{bcomp}{1}{13}{1.72}     
\uarc{gcomp}{2}{12}{1.38}     
\uarc{bcomp}{3}{6}{0.62}      
\uarc{rcomp}{4}{5}{0.26}      
\uarc{gcomp}{7}{11}{0.80}     
\uarc{bcomp}{8}{9}{0.28}      
\uarc{bcomp}{13}{24}{1.58}    
\uarc{gcomp}{14}{18}{0.72}    
\uarc{gcomp}{15}{17}{0.43}    
\uarc{pcomp}{19}{23}{0.72}    
\uarc{bcomp}{20}{22}{0.43}    

\darc{gcomp}{0}{8}{1.32}      
\darc{bcomp}{1}{6}{0.76}      
\darc{gcomp}{2}{3}{0.26}      
\darc{rcomp}{4}{5}{0.23}      
\darc{bcomp}{8}{21}{1.72}     
\darc{gcomp}{9}{16}{1.12}     
\darc{gcomp}{10}{15}{0.80}    
\darc{bcomp}{12}{14}{0.36}    
\darc{gcomp}{16}{18}{0.43}    
\darc{bcomp}{18}{19}{0.23}    
\darc{pcomp}{19}{20}{0.23}    
\darc{bcomp}{21}{24}{0.48}    
\darc{pcomp}{22}{23}{0.23}    

\end{tikzpicture}
\caption{The underlying graph associated with the middle restrictions \((p^{-},q^{+})\).
Edges coming from \(p^{-}\) are drawn above the line, while edges coming from
\(q^{+}\) are drawn below it.}

\end{figure}

\begin{lemma}[Balancedness and potentials]\label{lem:balanced-pot}
Let $\mathcal G=(V,E,s,t,\Phi)$ be a gain graph.
Then $\mathrm{Pot}(\mathcal G)\neq\emptyset$ if and only if $\mathcal G$ is balanced.
\end{lemma}

\begin{lemma}\label{p+q}
For every simple  cycle $\gamma$ in $\overline G_{p,q}$ and every $x\in\{1,\dots,l\}$,
the integer $\chi_{p^{-}}^{\gamma}(x)+\chi_{q^{+}}^{\gamma}(x)$ is even.
\end{lemma}

\begin{proof}
Define the cut indicator $\varepsilon_x(v):=\mathbf 1_{\{v\ge x\}} (v\in V_{p,q}).$
We claim that for any block \(B\in p^{-}\sqcup q^{+}\),
$\min B\le x\le \max B\iff\varepsilon_x(\ell(B))\neq \varepsilon_x(r(B)).$
Indeed, since \(\ell(B)=\min B-1\) and \(r(B)=\max B\), $\varepsilon_x(\ell(B))\neq \varepsilon_x(r(B))
\iff
\ell(B)<x\le r(B)
\iff
\min B\le x\le \max B.$

Write $\gamma=(e_{B_1},\dots,e_{B_N}),$
and choose vertices \(v_0,\dots,v_N\in V_{p,q}\) such that
$v_N=v_0,\partial e_{B_k}=\{v_{k-1},v_k\}(1\le k\le N).$
Since $\partial e_{B_k}=\{\ell(B_k),r(B_k)\},$
for each \(k\) we have $B_k\text{ crosses }x
\iff
\varepsilon_x(\ell(B_k))\neq \varepsilon_x(r(B_k))
\iff
\varepsilon_x(v_{k-1})\neq \varepsilon_x(v_k).$ 
Therefore
\[
\chi_{p^{-}}^{\gamma}(x)+\chi_{q^{+}}^{\gamma}(x)
=
\#\{\,k\in\{1,\dots,N\}:\ \varepsilon_x(v_{k-1})\neq \varepsilon_x(v_k)\,\}
=: \Delta.
\]
For each \(k=1,\dots,N\), define
\[
\delta_k:=
\begin{cases}
1,& \varepsilon_x(v_{k-1})\neq \varepsilon_x(v_k),\\
0,& \varepsilon_x(v_{k-1})=\varepsilon_x(v_k).
\end{cases}
\]
Then \(\Delta=\sum_{k=1}^N\delta_k\), and since
\(\varepsilon_x(v_{k-1}),\varepsilon_x(v_k)\in\{0,1\}\),$\delta_k\equiv \varepsilon_x(v_{k-1})+\varepsilon_x(v_k)\pmod 2.$
Hence$\Delta
\equiv
\sum_{k=1}^N\bigl(\varepsilon_x(v_{k-1})+\varepsilon_x(v_k)\bigr)
\pmod 2.$
But
\[
\sum_{k=1}^N \varepsilon_x(v_{k-1})=\sum_{j=0}^{N-1}\varepsilon_x(v_j),
\qquad
\sum_{k=1}^N \varepsilon_x(v_k)=\sum_{j=1}^{N}\varepsilon_x(v_j).
\]
Since \(v_N=v_0\), every term \(\varepsilon_x(v_j)\) for \(0\le j\le N-1\) appears twice in the total sum. Therefore $\Delta\equiv 0\pmod 2.$
So \(\Delta\) is even.
\end{proof}

\begin{corollary}\label{mod2}
Let $\gamma$ be a simple cycle in $\overline G_{p,q}$, and let $\mathcal C$ be a
$(p^{-},q^{+})$-connected component. Then $\chi_{p^{-}}^{\gamma}(a)\equiv \chi_{p^{-}}^{\gamma}(b)\pmod 2
\qquad(a,b\in\mathcal C).$
\end{corollary}

\begin{proof}
Write $\gamma=(e_{B_1},\dots,e_{B_N})$. First suppose that $a,b$ belong to the
same block $D\in q^{+}$. By noncrossingness of $q^{+}$, for every
$B_i\in q^{+}$ one has $\min B_i\le a\le \max B_i
\Longleftrightarrow\min B_i\le b\le \max B_i.$ Hence $\chi_{q^{+}}^{\gamma}(a)=\chi_{q^{+}}^{\gamma}(b).$
By Lemma~\ref{p+q},
\[
\chi_{p^{-}}^{\gamma}(a)+\chi_{q^{+}}^{\gamma}(a)\equiv 0\pmod 2,
\qquad
\chi_{p^{-}}^{\gamma}(b)+\chi_{q^{+}}^{\gamma}(b)\equiv 0\pmod 2.
\]
Thus
\[
\chi_{p^{-}}^{\gamma}(a)-\chi_{p^{-}}^{\gamma}(b)
\equiv
-\bigl(\chi_{q^{+}}^{\gamma}(a)-\chi_{q^{+}}^{\gamma}(b)\bigr)
\equiv 0\pmod 2.
\]

On the other hand, if $a,b$ belong to the same block $D\in p^{-}$, then by
noncrossingness of $p^{-}$, for every $B_i\in p^{-}$ one has $\min B_i\le a\le \max B_i\Longleftrightarrow
\min B_i\le b\le \max B_i,$
and therefore
\[
\chi_{p^{-}}^{\gamma}(a)=\chi_{p^{-}}^{\gamma}(b).
\]

Finally, if $a,b\in\mathcal C$, then there exists a chain
\[
a=a_0,a_1,\dots,a_M=b
\]
such that $a_{j-1}$ and $a_j$ belong to a common block of $p^{-}$ or of $q^{+}$
for every $j$. Applying the two cases above along the chain gives $\chi_{p^{-}}^{\gamma}(a)\equiv \chi_{p^{-}}^{\gamma}(b)\pmod 2.$
\end{proof}

\begin{lemma}\label{outer cycle}
Let $p\in \NC_{\Lambda}(k,l)$ and $q\in \NC_{\Lambda}(l,m)$ with
$T_q\circ T_p\neq 0$. If $|K(q,p)|=1$, then $\overline G_{p,q}$ has a unique
simple cycle. Its edges are exactly those indexed by the outer blocks of
$p^{-}\sqcup q^{+}$.
\end{lemma}

\begin{proof}
Let $B_1,\dots,B_r$ and $C_1,\dots,C_s$ be the outer blocks of $p^{-}$ and
$q^{+}$, respectively, listed from left to right. Since the spans of the outer
blocks of a noncrossing partition are pairwise disjoint and cover $[1,l]$, we
have
\[
\ell(B_1)=0,\qquad r(B_i)=\ell(B_{i+1})\ (1\le i<r),\qquad r(B_r)=l,
\]
and similarly
\[
\ell(C_1)=0,\qquad r(C_j)=\ell(C_{j+1})\ (1\le j<s),\qquad r(C_s)=l.
\]
Hence the outer blocks of $p^{-}$ form the simple path
\[
\rho_p:\quad
0=\ell(B_1),\ r(B_1)=\ell(B_2),\ \dots,\ r(B_{r-1})=\ell(B_r),\ r(B_r)=l,
\]
and the outer blocks of $q^{+}$ form the simple path
\[
\rho_q:\quad
0=\ell(C_1),\ r(C_1)=\ell(C_2),\ \dots,\ r(C_{s-1})=\ell(C_s),\ r(C_s)=l.
\]

We claim that these two paths meet only at $0$ and $l$. Indeed, if
$x\in\{1,\dots,l-1\}$ were a common vertex, then no block of $p^{-}$ crosses
$x$, and no block of $q^{+}$ crosses $x$. Hence $\{1,\dots,x\}$ and
$\{x+1,\dots,l\}$ would lie in different $(p^{-},q^{+})$-connected components,
contradicting $|K(q,p)|=1$. Therefore the two outer paths form a simple cycle.

It remains to prove uniqueness. Let $\gamma=(e_{D_1},\dots,e_{D_N})$
be any simple cycle in $\overline G_{p,q}$, and let $x_\gamma$ be the maximal
vertex appearing in $\gamma$. Since $\gamma$ is simple, exactly two edges of
$\gamma$ are incident with $x_\gamma$. By maximality of $x_\gamma$, both have
right endpoint $x_\gamma$. At most one of them comes from $p^{-}$ and at most
one from $q^{+}$, since two distinct blocks of the same partition cannot have
the same right endpoint. Hence $\chi_{p^{-}}^{\gamma}(x_\gamma)=1,\chi_{q^{+}}^{\gamma}(x_\gamma)=1.$
Since $|K(q,p)|=1$, all points of $[l]$ lie in the same
$(p^{-},q^{+})$-connected component. By Corollary~\ref{mod2},
\[
\chi_{p^{-}}^{\gamma}(x)\equiv 1\pmod 2
\qquad(1\le x\le l).
\]
Then Lemma~\ref{p+q} gives also
\[
\chi_{q^{+}}^{\gamma}(x)\equiv 1\pmod 2
\qquad(1\le x\le l).
\]

We now show that every outer block belongs to $\gamma$. Let $B$ be an outer
block of $p^{-}$ and set $b:=\max B$. Then $B$ is the unique block of $p^{-}$
crossing $b$. Hence $\chi_{p^{-}}^{\gamma}(b)$ is either $0$ or $1$, according
as $e_B$ does not belong to $\gamma$ or belongs to $\gamma$. Since
$\chi_{p^{-}}^{\gamma}(b)$ is odd, we get $e_B\in\gamma$. The same argument
shows that every outer block of $q^{+}$ belongs to $\gamma$.

It remains to rule out inner blocks.
Suppose first that $\gamma$ contains an inner block $D\in p^{-}$.
Choose such a block with maximal right endpoint, and set $d:=\max D$.
Let $B$ be the unique outer block of $p^{-}$ such that $\operatorname{span}(D)\subsetneq \operatorname{span}(B).$
By the previous paragraph, $e_B\in\gamma$.
Now both $B$ and $D$ cross $d$.
Moreover, by maximality of $d$, no other inner block of $p^{-}$ belonging to $\gamma$
crosses $d$, and no outer block of $p^{-}$ other than $B$ crosses $d$.

Hence $\chi_{p^{-}}^{\gamma}(d)=2,$ contradicting the fact that $\chi_{p^{-}}^{\gamma}(d)$ is odd.
Thus $\gamma$ contains no inner block of $p^{-}$.
By the same argument, $\gamma$ contains no inner block of $q^{+}$.

Therefore every simple cycle has exactly the outer blocks as its edges. Since
the outer blocks already form a simple cycle, this cycle is unique.
\end{proof}

\begin{lemma}[Potential functions and solutions]
Let $\mathcal G(\mathcal E^{p,q}_{\vec r,\vec d})$ be the gain graph associated to $(p^{-},q^{+})$, and let $\mathrm{Sol}$
be the solution set of the system $\mathcal E^{p,q}_{\vec r,\vec d}$.
Then there is a canonical bijection $\mathrm{Pot}_0(\mathcal G)\leftrightarrow \mathrm{Sol}$:
\[
F(\pi)=(x_k)_{k=1}^l,\ \ x_k=\pi(k-1)^{-1}\pi(k),
\qquad
G(x)=\pi,\ \ \pi(0)=1_\Lambda,\ \ \pi(m)=\prod_{k=1}^m x_k\ (1\le m\le l).
\]
\end{lemma}

\begin{proof}
Let $\pi\in\mathrm{Pot}_0(\mathcal G)$ and put $x_k=\pi(k-1)^{-1}\pi(k)$. For any block $B\in p^{-}\sqcup q^{+}$,
telescoping gives
\[
\prod_{k=\min B}^{\max B}x_k=\pi(\ell(B))^{-1}\pi(r(B))=\Phi(B^{+})=\mathrm{col}(B),
\]
hence $F(\pi)\in\mathrm{Sol}$.

Conversely, let $x=(x_k)_{k=1}^l\in\mathrm{Sol}$ and define $\pi(0)=1_\Lambda$, $\pi(m)=\prod_{k=1}^m x_k$.
Then for any block $B$,
\[
\pi(\ell(B))^{-1}\pi(r(B))
=
\Bigl(\prod_{k=1}^{\ell(B)}x_k\Bigr)^{-1}\Bigl(\prod_{k=1}^{r(B)}x_k\Bigr)
=
\prod_{k=\min B}^{\max B}x_k
=
\mathrm{col}(B)
=
\Phi(B^{+}),
\]
so $\pi\in\mathrm{Pot}_0(\mathcal G)$.

Finally, $G(F(\pi))(m)=\prod_{k=1}^m(\pi(k-1)^{-1}\pi(k))=\pi(m)$ and
$F(G(x))_k=\pi(k-1)^{-1}\pi(k)=x_k$, hence $F$ and $G$ are inverse bijections.
\end{proof}

\begin{lemma}\label{|K(q,p)|=1}
Assume that $p\in \NC_{\Lambda}(k,l)$ and $q\in \NC_{\Lambda}(l,m)$ satisfy
$T_q\circ T_p\neq 0$ and $|K(q,p)|=1$.
Let $\mathcal G(\mathcal E^{p,q}_{\vec r,\vec d})$ be the gain graph associated
with $\mathcal E^{p,q}_{\vec r,\vec d}$. By Lemma~\ref{outer cycle}, $\overline G_{p,q}$ has a unique simple cycle;
denote it by $\gamma_{\mathrm{out}}$. Let $\widetilde\gamma_{\mathrm{out}}$ be
the directed lift of $\gamma_{\mathrm{out}}$ following the order in which the
boundary word is read in the definition of $\delta_{q\cdot p}(\vec r,\vec d)$.
Then $\delta_{q\cdot p}(\vec r,\vec d)=1
\quad\Longrightarrow\quad
\Phi_{\vec r,\vec d}(\widetilde\gamma_{\mathrm{out}})=1_{\Lambda}.$ Equivalently, under the condition $\delta_{q\cdot p}(\vec r,\vec d)=1$, the $\mathcal{E}_{\vec r,\vec d}^{p,q}$ has solution 
.

\end{lemma}

\begin{proof}
There are two possible types of non-trivial $(p,q)$-connected components $\mathcal C$ with
$\mathcal C\cap [l]\neq\varnothing$:
\begin{enumerate}
\item[(1)] $\mathcal C$ is a half-connected component, i.e. an upper-half or lower-half connected component;
\item[(2)] $\mathcal C$ is a through connected component.
\end{enumerate}

We prove the upper-half case; the lower-half case is analogous. Let
\(T_1,\dots,T_r\) be the elements of \(\mathrm{Th}_p(\mathcal C)\), listed from
left to right. Since \(\mathcal C\) is upper-half and \(|K(q,p)|=1\), we have
$\mathrm{Th}_q(\mathcal C)
=\mathrm{Ent}^{\uparrow}(\mathcal C)=\mathrm{Ent}^{\downarrow}(\mathcal C)
=\varnothing.$
Thus, by Proposition~\ref{lem:two-sided-middle-decomposition},
\[
H_{\mathcal C}
=
\left(\bigsqcup_{T\in \mathrm{Th}_p(\mathcal C)} I(T)\right)
\sqcup
\left(\bigsqcup_{D\in \mathrm{FreeOut}_p^{\downarrow}(\mathcal C)} I(D)\right),
\qquad
H_{\mathcal C}
=
\bigsqcup_{D\in \mathrm{FreeOut}_q^{\uparrow}(\mathcal C)} I(D).
\]

On the \(p^{-}\)-side, let \(T_1,\dots,T_t\) be the elements of \(\mathrm{Th}_p(\mathcal C)\), listed
from left to right. Put \(T_i^-:=T_i\cap[l]\in p^{-}\), and define $\mathcal A_{\mathcal C}^{p}:=\{\,T^-:T\in\mathrm{Th}_p(\mathcal C)\,\}\sqcup
\mathrm{FreeOut}_p^{\downarrow}(\mathcal C)\subset p^{-}.$
List the elements of \(\mathcal A_{\mathcal C}^{p}\) from left to right as $A_1\prec\cdots\prec A_N.$ By Proposition~\ref{lem:two-sided-middle-decomposition}, $H_{\mathcal C}
=\bigsqcup_{j=1}^{N} I(A_j).$ 

Hence  $\ell(A_1)=\min(\mathcal C\cap[l])-1, r(A_j)=\ell(A_{j+1})(1\le j<N), r(A_N)=\max(\mathcal C\cap[l]),$
so the edges \(e_{A_1},\dots,e_{A_N}\) form a path
$\rho^{p}:=(e_{A_1},\dots,e_{A_N})$
in \(\overline G_{p,q}\). Write $A_{\nu_i}=T_i^-(1\le i\le t).$

Define the subpaths $\rho^{L}:=(e_{A_1},\dots,e_{A_{\nu_1-1}}),\rho^{M}:=(e_{A_{\nu_1}},\dots,e_{A_{\nu_t}}),
\rho^{R}:=(e_{A_{\nu_t+1}},\dots,e_{A_N}).$
Empty paths are allowed. 
 Thus \(\rho^{p}\) is the concatenation of $\rho^{L},\rho^{M},\rho^{R}$. For their directed lifts with the left-to-right
orientation, \[\Phi_{\vec r,\vec d}(\widetilde\rho^{L})=
\prod_{j<\nu_1}\kappa_{\vec r,\vec d}(A_j),
\qquad
\Phi_{\vec r,\vec d}(\widetilde\rho^{M})
=
\prod_{\nu_1\le j\le\nu_t}\kappa_{\vec r,\vec d}(A_j),
\qquad
\Phi_{\vec r,\vec d}(\widetilde\rho^{R})
=
\prod_{j>\nu_t}\kappa_{\vec r,\vec d}(A_j),
\]
with the convention that empty products are \(1_\Lambda\).

On the \(q^{+}\)-side, put $\mathcal A_{\mathcal C}^{q}:=
\mathrm{FreeOut}_q^{\uparrow}(\mathcal C)
\subset q^{+}.$
List its elements from left to right as $Q_1\prec\cdots\prec Q_M.$
By Proposition~\ref{lem:two-sided-middle-decomposition}, $H_{\mathcal C}
=\bigsqcup_{j=1}^{M} I(Q_j).$
Hence $\ell(Q_1)=\min(\mathcal C\cap[l])-1,r(Q_j)=\ell(Q_{j+1})(1\le j<M),r(Q_M)=\max(\mathcal C\cap[l]).$
Thus the edges \(e_{Q_1},\dots,e_{Q_M}\) form a path $\rho^{q}:=(e_{Q_1},\dots,e_{Q_M})$
in \(\overline G_{p,q}\). In \(\overline G_{p,q}\). For its right-to-left directed lift, $\Phi_{\vec r,\vec d}(\widetilde\rho^q)
=
(\prod_{j=1}^{M}\kappa_{\vec r,\vec d}(Q_j))^{-1}.$ Hence $\widetilde\gamma_{\mathrm{out}}$ is the concatenation of $\widetilde\rho^{L}, \widetilde\rho^{M}, \widetilde\rho^{R}, \widetilde\rho^q$. Therefore $\Phi_{\vec r,\vec d}(\widetilde\gamma_{\mathrm{out}})=\Phi_{\vec r,\vec d}(\widetilde\rho^{L})\Phi_{\vec r,\vec d}(\widetilde\rho^{M})\Phi_{\vec r,\vec d}(\widetilde\rho^{R})\Phi_{\vec r,\vec d}(\widetilde\rho^q)$

We show that \(\widetilde\gamma_{\mathrm{out}}\) is balanced in
\(\mathcal G(\mathcal E^{p,q}_{\vec r,\vec d})\). Since \(|K(q,p)|=1\), the
upper boundary contributions attached to the through-blocks of \(p\) are
trivial, that is, $\alpha_V=1 (V\in\mathrm{Th}_p(\mathcal C)).$ Using \(\delta_{q\cdot p}(\vec r,\vec d)=1\) and
Corollary~\ref{colour of components}(1), a direct computation gives $\Phi_{\vec r,\vec d}(\widetilde\rho^{M})=f_{\mathcal C}.$

On the other hand, by Corollary~\ref{cor: constant interval}(2), there exist
boundary data \((\vec r',\vec d')\) and
\(\vec x\in\theta_p^{\vec r'}\cap\Omega_q^{\vec d'}\) such that $\prod_{F_{\mathcal C}}\vec x
=f_{\mathcal C}.$
Equivalently, $\Phi_{\vec r',\vec d'}(\widetilde\rho^{M})=f_{\mathcal C}=\Phi_{\vec r,\vec d}(\widetilde\rho^{M}).$

Moreover, the coefficients along the left part, the right part, and the
\(q^{+}\)-side path are unchanged when passing from
\((\vec r,\vec d)\) to \((\vec r',\vec d')\). Hence $\Phi_{\vec r,\vec d}(\widetilde\rho^{L})
=\Phi_{\vec r',\vec d'}(\widetilde\rho^{L}),\Phi_{\vec r,\vec d}(\widetilde\rho^{R})
=\Phi_{\vec r',\vec d'}(\widetilde\rho^{R}),$
and $\Phi_{\vec r,\vec d}(\widetilde\rho^{q})
=\Phi_{\vec r',\vec d'}(\widetilde\rho^{q}).$
Therefore $\Phi_{\vec r,\vec d}(\widetilde\gamma_{\mathrm{out}})=\Phi_{\vec r',\vec d'}(\widetilde\gamma_{\mathrm{out}}).$

Since \(\vec x\in\theta_p^{\vec r'}\cap\Omega_q^{\vec d'}\), the system
\(\mathcal E^{p,q}_{\vec r',\vec d'}\) admits a solution. Hence the associated
gain graph is balanced, and in particular $\Phi_{\vec r',\vec d'}(\widetilde\gamma_{\mathrm{out}})=1_\Lambda.$
Consequently, $\Phi_{\vec r,\vec d}(\widetilde\gamma_{\mathrm{out}})=1_\Lambda.$

The lower-half case is completely parallel, with the roles of \(p^{-}\) and
\(q^{+}\) exchanged and with Corollary~\ref{colour of components}(2) replacing
Corollary~\ref{colour of components}(1). The through case is handled in the
same way, applying Proposition~\ref{lem:two-sided-middle-decomposition} to the
two sides of \(H_{\mathcal C}\). In both cases the same gain comparison, together
with \(\delta_{q\cdot p}(\vec r,\vec d)=1\), yields $\Phi_{\vec r,\vec d}(\widetilde\gamma_{\mathrm{out}})=1_\Lambda.$

\end{proof}

\begin{lemma}[Counting normalized potentials on a balanced gain graph]
Let $\mathcal G=(V,E,s,t,\Phi)$ be a gain graph with gains in the finite group $\Lambda$.
Assume that $\mathcal G$ is \emph{balanced}, i.e.\ $\Phi(\gamma)=1_\Lambda$ for every simple directed cycle $\gamma$.
Let $c=c(\mathcal G)$ be the number of connected components of the underlying (undirected) graph, and
let $C_1,\dots,C_c$ be these components with $0\in C_1$.
Then $\mathrm{Pot}_0(\mathcal G)\neq\emptyset$ and
$|\mathrm{Pot}_0(\mathcal G)|=|\Lambda|^{c-1}.$
More precisely, choosing roots $r_1:=0\in C_1$ and $r_j\in C_j$ for $j\ge2$, the map
\[
\Theta_0:\mathrm{Pot}_0(\mathcal G)\longrightarrow \Lambda^{c-1},
\qquad
\Theta_0(\pi):=\bigl(\pi(r_2),\dots,\pi(r_c)\bigr)
\]
is a bijection. In particular, if $\mathcal G$ is not balanced, then $\mathrm{Pot}_0(\mathcal G)=\emptyset$.
\end{lemma}

\begin{proof}
Fix $j\in\{1,\dots,c\}$ and a root $r_j\in C_j$.
Balancedness implies that the gain of any closed walk in $C_j$ equals $1_\Lambda$.
Hence we may define $\pi^{(0)}_j:C_j\to\Lambda$ by $\pi^{(0)}_j(r_j)=1_\Lambda$ and
$\pi^{(0)}_j(v):=\Phi(P_{r_j\to v})$, where $P_{r_j\to v}$ is any path from $r_j$ to $v$.
This is well-defined, and satisfies $\bigl(\pi^{(0)}_j(u)\bigr)^{-1}\pi^{(0)}_j(v)=\Phi(e)$ for every edge $e:u\to v$ in $C_j$.

\noindent\emph{Surjectivity of $\Theta_0$.}
Fix $(g_2,\dots,g_c)\in\Lambda^{c-1}$ and define $\pi:V\to\Lambda$ by
$\pi|_{C_1}:=\pi^{(0)}_1,\pi|_{C_j}:=g_j\,\pi^{(0)}_j\ \ (j=2,\dots,c).$ Then $\pi(0)=1_\Lambda$, the edge relation holds on each $C_j$, hence
$\pi\in\mathrm{Pot}_0(\mathcal G)$ and $\Theta_0(\pi)=(g_2,\dots,g_c)$.
\noindent\emph{Injectivity of $\Theta_0$.}
Let $\pi,\pi'\in\mathrm{Pot}_0(\mathcal G)$ with $\Theta_0(\pi)=\Theta_0(\pi')$.
Fix $v\in C_j$ and choose a path $P$ from $r_j$ to $v$. Multiplying the edge relations along $P$ yields $\pi(v)=\pi(r_j)\,\Phi(P)$, $\pi'(v)=\pi'(r_j)\,\Phi(P).$
If $j=1$, then $\pi(r_1)=\pi'(r_1)=\pi(0)=\pi'(0)=1_\Lambda$; if $j\ge2$, then $\pi(r_j)=\pi'(r_j)$ by assumption.
Hence $\pi(v)=\pi'(v)$ for all $v$, so $\pi=\pi'$ and $\Theta_0$ is bijective.
\end{proof}

\begin{corollary}[Number of solutions]\label{Number of solutions}
If the system \(\mathcal E_{\vec r,\vec d}^{p,q}\) admits a solution (for the parameters \(\vec r,\vec d\)),
then the solution set has cardinality
$\#\mathrm{Sol}\bigl(\mathcal E_{\vec r,\vec d}^{p,q}\bigr)=|\Lambda|^{\,c-1}.$
In particular, whenever \(\mathcal E_{\vec r,\vec d}^{p,q}\) is solvable, the number of solutions depends only on \(c\)
(and \(|\Lambda|\)), not on \(\vec r,\vec d\).
\end{corollary}

\begin{lemma}[Existence of a gluing step]\label{lem:gluing-step}
Let $p\in \NC_{\Lambda}(k,l)$ and $q\in \NC_{\Lambda}(l,m)$ with
$T_q\circ T_p\neq 0$. Assume that $|K(q,p)|\ge 2$ and that the intervals
$H_{\mathcal C}$, $\mathcal C\in K(q,p)$, are not pairwise disjoint. For every block $X$ of $\mathcal C_p$ or $\mathcal C_q$ satisfying
$X\cap[l]\neq\varnothing$, define $I_X:=[\min(X\cap[l]),\max(X\cap[l])].$ Then there exist distinct connected components
$\mathcal C,\mathcal C'\in K(q,p)$ such that
\[
H_{\mathcal C}\cap H_{\mathcal C'}\neq\varnothing,
\qquad
x:=\min(H_{\mathcal C}\cap H_{\mathcal C'})=\min H_{\mathcal C'}.
\]
Moreover, one of the following two situations occurs.

\medskip

\noindent
\textup{(p-case)} Let $B:=\pi_p(x)$. Then $
B\in \mathrm{FreeOut}_p^{\downarrow}(\mathcal C').$ (See Proposation \ref{lem:two-sided-middle-decomposition})
There exists a block $D\in \mathcal C_p$ such that $\operatorname{span}(B)\subsetneq I_D,$
and such that there is no block $V\in p$ with $$\operatorname{span}(B)\subsetneq \operatorname{span}(V)\subsetneq I_D.$$

\medskip

\noindent
\textup{(q-case)} Let $B:=\pi_q(x)$. Then $B\in \mathrm{FreeOut}_q^{\uparrow}(\mathcal C').$
There exists a block $D\in \mathcal C_q$ such that $\operatorname{span}(B)\subsetneq I_D,$
and such that there is no block $V\in q$ with
\[
\operatorname{span}(B)\subsetneq \operatorname{span}(V)\subsetneq I_D.
\]
\end{lemma}

\begin{proof}
Since the intervals $H_{\mathcal C}$, $\mathcal C\in K(q,p)$, are not pairwise
disjoint, there exist distinct connected components
$\mathcal O,\mathcal O'\in K(q,p)$ such that $H_{\mathcal O}\cap H_{\mathcal O'}\neq\varnothing.$
Write $H_{\mathcal O}=[a,b], H_{\mathcal O'}=[c,d].$
Then $H_{\mathcal O}\cap H_{\mathcal O'}
=
[\max\{a,c\},\,\min\{b,d\}],$
so that $x:=\min(H_{\mathcal O}\cap H_{\mathcal O'})=\max\{a,c\}.$
Since $x=a$ or $x=c$, after relabelling $\mathcal O$ and $\mathcal O'$ if necessary,
we may assume that $x=\min H_{\mathcal O'}.$
Set $\mathcal C':=\mathcal O'.$

We first show that
\[
\pi_q(x)=V_l^{\downarrow}(\mathcal C')
\qquad\text{and}\qquad
\pi_p(x)=V_l^{\uparrow}(\mathcal C')
\]
cannot hold simultaneously. Indeed, otherwise the \(p\)-block and the \(q\)-block
containing \(x\) would both be through blocks. Since through blocks admit no
nesting in a noncrossing partition, this would force
\[
\max H_{\mathcal O}<x,
\]
contradicting \(x\in H_{\mathcal O}\cap H_{\mathcal O'}\).

Hence at least one of \(\pi_p(x)\) and \(\pi_q(x)\) is not a through block.

Assume first that \(\pi_p(x)\) is not a through block. Then \(\pi_p(x)\) is a lower
\(p\)-block of \(\mathcal C'\), and in fact outer: otherwise there would exist a lower
\(p\)-block \(V\in \mathcal C'_p\) with \(\operatorname{span}{\pi_p(x)}\subsetneq I_V\), which would force
\(\min(D\cap [l])<x\), contradicting \(x=\min H_{\mathcal C'}\).

We claim that \(\pi_p(x)\) is not \(p\)-covered (See Proposation \ref{lem:two-sided-middle-decomposition}). Otherwise, by
Proposition~\ref{lem:two-sided-middle-decomposition}, there would exist
\(T\in \mathcal T_p^{\uparrow}(\mathcal C')\) such that $\operatorname{span}(\pi_p(x))\subset I_T.$
Since \(x\in \pi_p(x)\cap [l]\), we have \(x\in \operatorname{span}(\pi_p(x))\subset I_T\). If \(x\in T\),
then \(T=\pi_p(x)\), contradicting the assumption that \(\pi_p(x)\) is not a through
block. Hence \(x\notin T\), and therefore
\[
\min(T\cap [l])<x,
\]
contradicting the fact that \(x\) is the leftmost point of \(\mathcal C'\cap [l]\).
Thus 
\[
\pi_p(x)\in \mathrm{FreeOut}_p^{\downarrow}(\mathcal C').
\]

The case where \(\pi_q(x)\) is not a through block is symmetric, and yields
\[
\pi_q(x)\in \mathrm{FreeOut}_q^{\uparrow}(\mathcal C').
\]

Consequently, $\pi_p(x)\in \mathrm{FreeOut}_p^{\downarrow}(\mathcal C')$ or $\pi_q(x)\in \mathrm{FreeOut}_q^{\uparrow}(\mathcal C').$

We next claim that either there exists \(D_0\in O_p\) such that
\[
\operatorname{span}(\pi_p(x))
\subsetneq I_{D_0}
=
[\min(D_0\cap [l]),\,\max(D_0\cap [l])]
\quad\text{and}\quad
\pi_p(x)\in \mathrm{FreeOut}_p^{\downarrow}(\mathcal C'),
\]
or there exists \(D_0\in O_q\) such that
\[
\operatorname{span}(\pi_q(x))
\subsetneq I_{D_0}
=
[\min(D_0\cap [l]),\,\max(D_0\cap [l])]
\quad\text{and}\quad
\pi_q(x)\in \mathrm{FreeOut}_q^{\uparrow}(\mathcal C').
\]

Indeed, by \(x\in H_{\mathcal O}\cap H_{\mathcal O'}\), either there exists
\(D_0\in O_p\) such that $\operatorname{span}(\pi_p(x))\subsetneq I_{D_0},$
or there exists \(D_0\in O_q\) such that $\operatorname{span}(\pi_q(x))\subsetneq I_{D_0}.$

Assume first that there exists \(D_0\in O_p\) with $\operatorname{span}(\pi_p(x))\subsetneq I_{D_0}.$
By the noncrossing property, \(\pi_p(x)\) cannot be a through block. Hence, by the
previous paragraph,
$\pi_p(x)\in \mathrm{FreeOut}_p^{\downarrow}(\mathcal C').$
The case of \(q\) is symmetric. This proves the claim.

Now assume that the first alternative in the claim holds, and set
\[
B:=\pi_p(x).
\]
Since \(\operatorname{span}(B)\subsetneq I_{D_0}\), there exists at least one \(p\)-block
\(D\) such that $\operatorname{span}(B)\subsetneq I_D.$

Choose such a \(p\)-block \(D\) so that \(I_D\) is minimal with respect to inclusion
among all \(p\)-blocks satisfying this property. Let \(\mathcal C\) be the connected
component containing \(D\).

Then \(D\notin \mathcal C'_p\). Indeed, otherwise \(D\) would be a \(p\)-block of
\(\mathcal C'\) s.t. $I_{D}$ strictly contains \(\operatorname{span}(B)\), contradicting $B\in \mathrm{FreeOut}_p^{\downarrow}(\mathcal C').$
Hence $\mathcal C\neq \mathcal C'.$

By the minimality of \(I_D\), there is no \(p\)-block \(V\) such that
\[
\operatorname{span}(B)\subsetneq \operatorname{span}(V)\subsetneq I_D.
\]
Thus we are in the \textup{(p-case)} of the lemma. If instead the second alternative
in the claim holds, the same argument with \(q\) in place of \(p\) yields the
\textup{(q-case)}. This proves the lemma.

\end{proof}

\begin{corollary}\label{cor:gluing-consequences}
Keep the notation of Lemma~\ref{lem:gluing-step}. In the \textup{(p-case)}, define
\[
\bar p:=(p\setminus\{B,D\})\cup\{D\cup B\},
\]
and colour the new block $D\cup B$ by $\operatorname{col}_p(D)$; then
\[
\bar p\in \NC_{\Lambda}(k,l).
\]
In the \textup{(q-case)}, define
\[
\bar q:=(q\setminus\{B,D\})\cup\{D\cup B\},
\]
and colour the new block $D\cup B$ by $\operatorname{col}_q(D)$; then
\[
\bar q\in \NC_{\Lambda}(l,m).
\]

Moreover, the only connected components affected by this operation are
$\mathcal C$ and $\mathcal C'$, which are merged into one. More precisely,
\[
K(q,\bar p)
=
\bigl(K(q,p)\setminus\{\mathcal C,\mathcal C'\}\bigr)\cup\{\mathcal C\cup\mathcal C'\}
\qquad\text{in the \textup{(p-case)},}
\]
\[
K(\bar q,p)
=
\bigl(K(q,p)\setminus\{\mathcal C,\mathcal C'\}\bigr)\cup\{\mathcal C\cup\mathcal C'\}
\qquad\text{in the \textup{(q-case)}.}
\]
In particular,
\[
|K(q,\bar p)|=|K(q,p)|-1
\qquad\text{or}\qquad
|K(\bar q,p)|=|K(q,p)|-1,
\]
according to the case.

Furthermore, for every $\vec r\in\Gamma^k$ and $\vec d\in\Gamma^m$,
\[
\mathcal E_{\vec r,\vec d}^{\bar p,q}
=
\mathcal E_{\vec r,\vec d}^{p,q}
\setminus
\left\{
\prod_{i=\min B}^{\max B}x_i=\operatorname{col}_p(B)
\right\}
\qquad\text{in the \textup{(p-case)},}
\]
\[
\mathcal E_{\vec r,\vec d}^{p,\bar q}
=
\mathcal E_{\vec r,\vec d}^{p,q}
\setminus
\left\{
\prod_{i=\min B}^{\max B}x_i=\operatorname{col}_q(B)^{-1}
\right\}
\qquad\text{in the \textup{(q-case)}.}
\]
Consequently,
\[
\theta_{\bar p}^{\vec r}\cap\Omega_q^{\vec d}
\supseteq
\theta_p^{\vec r}\cap\Omega_q^{\vec d}
\qquad\text{or}\qquad
\theta_p^{\vec r}\cap\Omega_{\bar q}^{\vec d}
\supseteq
\theta_p^{\vec r}\cap\Omega_q^{\vec d},
\]
according to the case.
\end{corollary}

\begin{proof}
We prove the \textup{(p-case)}; the \textup{(q-case)} is identical.

Define
\[
\bar p:=(p\setminus\{B,D\})\cup\{D\cup B\},
\]
where the new block $D\cup B$ is coloured by $\operatorname{col}_p(D)$.

We first check that $\bar p\in \NC_{\Lambda}(k,l)$. Since
\[
\operatorname{span}(B)\subsetneq \operatorname{span}(D),
\]
we have
\[
\operatorname{span}(D\cup B)=\operatorname{span}(D).
\]
Let $E\in p\setminus\{B,D\}$. We claim that $E$ does not cross $D\cup B$.
Indeed, if $E$ crossed $D\cup B$, then, since $E$ does not cross $D$ and
$\operatorname{span}(D\cup B)=\operatorname{span}(D)$, we would necessarily have
\[
\operatorname{span}(E)\subsetneq \operatorname{span}(D)
\quad\text{and}\quad
\operatorname{span}(E)\cap \operatorname{span}(B)\neq\varnothing.
\]
As $E$ and $B$ are blocks of the noncrossing partition $p$, their spans are
either disjoint or nested; hence
\[
\operatorname{span}(B)\subsetneq \operatorname{span}(E)\subsetneq \operatorname{span}(D),
\]
contradicting Lemma~\ref{lem:gluing-step}. Therefore no block
$E\in p\setminus\{B,D\}$ crosses $D\cup B$, and so
\[
\bar p=(p\setminus\{B,D\})\cup\{D\cup B\}\in \NC_{\Lambda}(k,l).
\]

Since the only blocks modified are $B$ and $D$, and no other block lies between
them, the only connected components affected by this operation are
$\mathcal C$ and $\mathcal C'$. These two connected components are merged into a
single connected component, while all other connected components remain unchanged.
Therefore
\[
K(q,\bar p)
=
\bigl(K(q,p)\setminus\{\mathcal C,\mathcal C'\}\bigr)\cup\{\mathcal C\cup\mathcal C'\},
\]
and in particular
\[
|K(q,\bar p)|=|K(q,p)|-1.
\]

Now let $\vec r\in\Gamma^k$ and $\vec d\in\Gamma^m$. Since
\[
\operatorname{span}(B)\subsetneq \operatorname{span}(D),
\]
we have
\[
\min(D\cup B)=\min D,
\qquad
\max(D\cup B)=\max D.
\]
Moreover, by construction, the new block $D\cup B$ is coloured by
$\operatorname{col}_p(D)$. Hence the equation attached to $D\cup B$ in
$\mathcal E_{\vec r,\vec d}^{\bar p,q}$ is exactly the same as the equation
attached to $D$ in $\mathcal E_{\vec r,\vec d}^{p,q}$. All other block equations
remain unchanged, except that the equation corresponding to $B$ disappears.
Therefore
\[
\mathcal E_{\vec r,\vec d}^{\bar p,q}
=
\mathcal E_{\vec r,\vec d}^{p,q}
\setminus
\left\{
\prod_{i=\min B}^{\max B}x_i=\operatorname{col}_p(B)
\right\}.
\]
Consequently,
\[
\theta_{\bar p}^{\vec r}\cap\Omega_q^{\vec d}
\supseteq
\theta_p^{\vec r}\cap\Omega_q^{\vec d}.
\]
\end{proof}

\begin{figure}[htbp]
\centering
\begin{tikzpicture}[
    scale=0.80,
    transform shape,
    x=0.42cm,
    y=0.72cm,
    pt/.style={circle,fill=black,inner sep=1.05pt},
    toplab/.style={font=\scriptsize},
    midlab/.style={font=\tiny},
    botlab/.style={font=\scriptsize},
    rowlab/.style={font=\small},
    gcomp/.style={
        draw=teal!70!black,
        line width=0.9pt,
        line cap=round,
        line join=round,
        rounded corners=1.8pt
    },
    bcomp/.style={
        draw=teal!70!black,
        line width=0.9pt,
        line cap=round,
        line join=round,
        rounded corners=1.8pt
    },
    rcomp/.style={
        draw=red!75!black,
        line width=0.9pt,
        line cap=round,
        line join=round,
        rounded corners=1.8pt
    },
    pcomp/.style={
        draw=violet!75!black,
        line width=0.9pt,
        line cap=round,
        line join=round,
        rounded corners=1.8pt
    }
]

\newcommand{\usingle}[2]{%
  \draw[#1]
    (#2-0.16,0.03)
      .. controls (#2-0.16,0.21) and (#2+0.16,0.21) ..
    (#2+0.16,0.03);
}
\newcommand{\lsingle}[2]{%
  \draw[#1]
    (#2-0.16,-0.03)
      .. controls (#2-0.16,-0.21) and (#2+0.16,-0.21) ..
    (#2+0.16,-0.03);
}

\foreach \i in {1,...,24}{
    \coordinate (M\i) at (\i,0);
}

\coordinate (T1) at (0.30,2.45);
\coordinate (T2) at (0.95,2.45);
\coordinate (T3) at (1.60,2.45);
\coordinate (T4) at (18.45,2.45);
\coordinate (T5) at (19.55,2.45);

\coordinate (B1) at (0.30,-2.45);
\coordinate (B2) at (0.95,-2.45);
\coordinate (B3) at (1.60,-2.45);
\coordinate (B4) at (2.25,-2.45);

\coordinate (BotMid) at ($(B1)!0.5!(B4)$);

\node[rowlab] at (-1.2,  2.45) {$[k]$};
\node[rowlab] at (-1.2,  0.00) {$[l]$};
\node[rowlab] at (-1.2, -2.45) {$[m]$};

\draw[gcomp] (T1) -- ($(T1)+(0,-0.32)$);
\draw[gcomp] (T2) -- ($(T2)+(0,-0.32)$);
\draw[gcomp] (T3) -- ($(T3)+(0,-0.32)$);
\draw[gcomp] ($(T1)+(0,-0.32)$) -- ($(T3)+(0,-0.32)$);
\draw[gcomp] ($(T2)+(0,-0.32)$) -- ($(T2)+(0,-0.88)$)
             -- ($(M1)+(0,1.55)$) -- (M1);

\draw[gcomp] (M3)  -- ($(M3)+(0,1.45)$);
\draw[gcomp] (M7)  -- ($(M7)+(0,1.45)$);
\draw[gcomp] (M12) -- ($(M12)+(0,1.45)$);
\draw[gcomp] ($(M3)+(0,1.45)$) -- ($(M12)+(0,1.45)$);

\draw[gcomp] (M8)  -- ($(M8)+(0,0.95)$);
\draw[gcomp] (M10) -- ($(M10)+(0,0.95)$);
\draw[gcomp] (M11) -- ($(M11)+(0,0.95)$);
\draw[gcomp] ($(M8)+(0,0.95)$) -- ($(M11)+(0,0.95)$);

\draw[gcomp] (M15) -- ($(M15)+(0,1.18)$);
\draw[gcomp] (M18) -- ($(M18)+(0,1.18)$);
\draw[gcomp] ($(M15)+(0,1.18)$) -- ($(M18)+(0,1.18)$);

\draw[gcomp] (M16) -- ($(M16)+(0,0.80)$);
\draw[gcomp] (M17) -- ($(M17)+(0,0.80)$);
\draw[gcomp] ($(M16)+(0,0.80)$) -- ($(M17)+(0,0.80)$);

\draw[bcomp] (M2)  -- ($(M2)+(0,1.82)$);
\draw[bcomp] (M13) -- ($(M13)+(0,1.82)$);
\draw[bcomp] ($(M2)+(0,1.82)$) -- ($(M13)+(0,1.82)$);

\draw[bcomp] (M4) -- ($(M4)+(0,0.86)$);
\draw[bcomp] (M6) -- ($(M6)+(0,0.86)$);
\draw[bcomp] ($(M4)+(0,0.86)$) -- ($(M6)+(0,0.86)$);

\usingle{bcomp}{9}

\coordinate (PtopL) at ($(T4)+(0,-0.32)$);
\coordinate (PtopR) at ($(T5)+(0,-0.32)$);
\coordinate (PtopM) at ($(PtopL)!0.5!(PtopR)$);

\draw[bcomp] (T4) -- (PtopL);
\draw[bcomp] (T5) -- (PtopR);
\draw[bcomp] (PtopL) -- (PtopR);

\coordinate (Pbot14) at ($(M14)+(0,1.45)$);
\coordinate (Pbot19) at ($(M19)+(0,1.45)$);
\coordinate (Pbot24) at ($(M24)+(0,1.45)$);

\draw[bcomp] (M14) -- (Pbot14);
\draw[bcomp] (M19) -- (Pbot19);
\draw[bcomp] (M24) -- (Pbot24);
\draw[bcomp] (Pbot14) -- (Pbot24);
\draw[bcomp] (PtopM) -- (Pbot19);

\draw[bcomp] (M21) -- ($(M21)+(0,0.80)$);
\draw[bcomp] (M22) -- ($(M22)+(0,0.80)$);
\draw[bcomp] ($(M21)+(0,0.80)$) -- ($(M22)+(0,0.80)$);

\usingle{rcomp}{5}

\draw[pcomp] (M20) -- ($(M20)+(0,1.12)$);
\draw[pcomp] (M23) -- ($(M23)+(0,1.12)$);
\draw[pcomp] ($(M20)+(0,1.12)$) -- ($(M23)+(0,1.12)$);

\coordinate (QTopL) at ($(M1)+(0,-1.58)$);
\coordinate (QTopR) at ($(M8)+(0,-1.58)$);
\coordinate (QTopMid) at ($(QTopL)!0.5!(QTopR)$);

\draw[gcomp] (M1) -- ($(M1)+(0,-1.58)$);
\draw[gcomp] (M2) -- ($(M2)+(0,-1.58)$);
\draw[gcomp] (M4) -- ($(M4)+(0,-1.58)$);
\draw[gcomp] (M6) -- ($(M6)+(0,-1.58)$);
\draw[gcomp] (M7) -- ($(M7)+(0,-1.58)$);
\draw[gcomp] (M8) -- ($(M8)+(0,-1.58)$);
\draw[gcomp] (QTopL) -- (QTopR);

\draw[gcomp] (B1) -- ($(B1)+(0,0.40)$);
\draw[gcomp] (B4) -- ($(B4)+(0,0.40)$);
\draw[gcomp] ($(B1)+(0,0.40)$) -- ($(B4)+(0,0.40)$);

\draw[gcomp] ($(BotMid)+(0,0.40)$) -- ($(BotMid)+(0,0.82)$)
             -- (QTopMid);

\lsingle{gcomp}{3}

\draw[gcomp] (M10) -- ($(M10)+(0,-1.40)$);
\draw[gcomp] (M16) -- ($(M16)+(0,-1.40)$);
\draw[gcomp] ($(M10)+(0,-1.40)$) -- ($(M16)+(0,-1.40)$);

\draw[gcomp] (M11) -- ($(M11)+(0,-1.10)$);
\draw[gcomp] (M12) -- ($(M12)+(0,-1.10)$);
\draw[gcomp] (M15) -- ($(M15)+(0,-1.10)$);
\draw[gcomp] ($(M11)+(0,-1.10)$) -- ($(M15)+(0,-1.10)$);

\draw[gcomp] (M17) -- ($(M17)+(0,-0.82)$);
\draw[gcomp] (M18) -- ($(M18)+(0,-0.82)$);
\draw[gcomp] ($(M17)+(0,-0.82)$) -- ($(M18)+(0,-0.82)$);

\draw[bcomp] (B2) -- ($(B2)+(0,0.27)$);
\draw[bcomp] (B3) -- ($(B3)+(0,0.27)$);
\draw[bcomp] ($(B2)+(0,0.27)$) -- ($(B3)+(0,0.27)$);

\draw[bcomp] (M9)  -- ($(M9)+(0,-1.82)$);
\draw[bcomp] (M21) -- ($(M21)+(0,-1.82)$);
\draw[bcomp] ($(M9)+(0,-1.82)$) -- ($(M21)+(0,-1.82)$);

\draw[bcomp] (M13) -- ($(M13)+(0,-0.76)$);
\draw[bcomp] (M14) -- ($(M14)+(0,-0.76)$);
\draw[bcomp] ($(M13)+(0,-0.76)$) -- ($(M14)+(0,-0.76)$);

\lsingle{bcomp}{19}

\draw[bcomp] (M22) -- ($(M22)+(0,-0.82)$);
\draw[bcomp] (M24) -- ($(M24)+(0,-0.82)$);
\draw[bcomp] ($(M22)+(0,-0.82)$) -- ($(M24)+(0,-0.82)$);

\lsingle{rcomp}{5}
\lsingle{pcomp}{20}
\lsingle{pcomp}{23}

\foreach \P/\lbl in {
    T1/{1^+},T2/{2^+},T3/{3^+},T4/{4^+},T5/{5^+}
}{
    \node[pt] at (\P) {};
    \node[toplab,above=2pt] at (\P) {$\lbl$};
}

\foreach \i in {1,...,24}{
    \node[pt] at (M\i) {};
    \node[midlab,below=2pt] at (M\i) {\(\i\)};
}

\foreach \P/\lbl in {
    B1/{1^-},B2/{2^-},B3/{3^-},B4/{4^-}
}{
    \node[pt] at (\P) {};
    \node[botlab,below=2pt] at (\P) {$\lbl$};
}

\end{tikzpicture}
\caption{Illustration of the \textup{(q-case)} of Corollary~\ref{cor:gluing-consequences}. It is obtained from Figure~\ref{fig:bilayer-nc-example} by gluing the two \(q\)-blocks
\(
B=\{2,4,6\}
\)
and
\(
D=\{1,7,8\}
\)
into the single block
\(
D\cup B=\{1,2,4,6,7,8\},
\)
that is, \(\bar q=(q\setminus\{B,D\})\cup\{D\cup B\}\).
Accordingly, the only connected components affected are the two components containing \(B\) and \(D\), which merge into one, illustrating
\(K(\bar q,p)=\bigl(K(q,p)\setminus\{\mathcal C,\mathcal C'\}\bigr)\cup\{\mathcal C\cup\mathcal C'\}\).
The singleton blocks are represented by small arcs.}
\label{fig:bilayer-nc-example-2}

\end{figure}

\begin{proposition}[Stability of $\delta$ under gluing]
\label{prop:delta-under-gluing}
Keep the notation of Lemma~\ref{lem:gluing-step} and
Corollary~\ref{cor:gluing-consequences}.

\medskip

\noindent
\textup{(p-case)} For every $\vec r\in \Gamma^k$ and $\vec d\in \Gamma^m$, $\delta_{q\cdot p}(\vec r,\vec d)=1
\Longrightarrow\delta_{q\cdot \bar p}(\vec r,\vec d)=1.$

\noindent
\textup{(q-case)} For every $\vec r\in \Gamma^k$ and $\vec d\in \Gamma^m$,
$\delta_{q\cdot p}(\vec r,\vec d)=1
\Longrightarrow\delta_{\bar q\cdot p}(\vec r,\vec d)=1.$
\end{proposition}
\begin{proof}
We prove the \textup{(p-case)}; the argument for the \textup{(q-case)} is entirely similar.

Fix $\vec r\in \Gamma^k$ and $\vec d\in \Gamma^m$ such that $\delta_{q\cdot p}(\vec r,\vec d)=1.$
We must show that $\delta_{q\cdot \bar p}(\vec r,\vec d)=1.$

Since all connected components of $q\cdot p$ other than $\mathcal C$ and
$\mathcal C'$ remain unchanged in $q\cdot \bar p$, it suffices to verify the
defining relation for the unique connected component $\mathcal U:=\mathcal C\cup \mathcal C'$of $q\cdot \bar p$ obtained by merging $\mathcal C$ and $\mathcal C'$.

We distinguish cases according to the relative position and the types of
\(\mathcal C\) and \(\mathcal C'\).

\medskip
\noindent\textbf{Case 1.} \(\mathcal U=\mathcal C \cup \mathcal C'\) is an upper(lower)-half $(\bar{p},q)$ connected component.

\smallskip
\noindent
\begin{minipage}[t]{0.08\textwidth}
\textbf{(a$^{\uparrow}$)}
\end{minipage}%
\begin{minipage}[t]{0.90\textwidth}
\vspace{-\baselineskip}
\begin{enumerate}
\item Both \(\mathcal C\) and \(\mathcal C'\) are of upper-half type with $[\min(\mathcal C\cap[k]),\max(\mathcal C\cap[k])]
\subset
[\min(\mathcal C'\cap[k]),\max(\mathcal C'\cap[k])].$

\item Up to interchanging \(\mathcal C\) and \(\mathcal C'\), the component
\(\mathcal C\) is of cycle type and \(\mathcal C'\) is of upper-half type.
\end{enumerate}
\end{minipage}

Then $\min(\mathcal U\cap[k])=\min(\mathcal C'\cap[k]),\max(\mathcal U\cap[k])=\max(\mathcal C'\cap[k])$ and $F_{\mathcal U}=F_{\mathcal C'}.$
Take $\vec x\in \bigcup_{\vec r,\vec d}(\theta_p^{\vec r}\cap\Omega_q^{\vec d})\subset\bigcup_{\vec r,\vec d}(\theta_{\bar p}^{\vec r}\cap\Omega_q^{\vec d}).$
By Corollary~\ref{cor: constant interval}, Definition~\ref{def: colour} and Remark  \ref{rem:label-represented-by-middle-vector}, the
color of \(\mathcal U\) in \(q\cdot\bar p\) and the color of \(\mathcal C'\) in
\(q\cdot p\) can be represented by \(\vec x\). 

Then $\mathrm{lab}_{q\cdot\bar p}(\mathcal U)=
t_{\mathcal C'}\Bigl(\prod \vec x_{\restriction F_{\mathcal C'}}\Bigr)^{-1}  g_{\mathcal C'}^{-1}
(t_{\mathcal C'}')^{-1}=t_{\mathcal C'} f_{\mathcal C'}^{-1}  g_{\mathcal C'}^{-1}
(t_{\mathcal C'}')^{-1}
=\mathrm{lab}_{q\cdot p}(\mathcal C').$
Thus the defining relation for \(\mathcal U\) is identical to that for
\(\mathcal C'\). Therefore, for any $\vec r$ satisfying
\(\delta_{q\cdot p}(\vec r,\vec d)=1\),
\[
\prod_{i=\min(\mathcal U\cap[k])}^{\max(\mathcal U\cap[k])} r_i
=\mathrm{lab}_{q\cdot p}(\mathcal C')^{-1}
=\mathrm{lab}_{q\cdot\bar p}(\mathcal U)^{-1}=\prod_{i=\min(\mathcal C'\cap[k])}^{\max(\mathcal C'\cap[k])} r_i,
\]
so the required condition holds for \(\mathcal U\).

\smallskip
\noindent
\begin{minipage}[t]{0.08\textwidth}
\textbf{(a$^{\downarrow}$)}
\end{minipage}%
\begin{minipage}[t]{0.90\textwidth}
\vspace{-\baselineskip}
\begin{enumerate}
\item Both \(\mathcal C\) and \(\mathcal C'\) are of lower-half type with $[\min(\mathcal C\cap[m]),\max(\mathcal C\cap[m])]
\subset
[\min(\mathcal C'\cap[m]),\max(\mathcal C'\cap[m])].$

\item Up to interchanging \(\mathcal C\) and \(\mathcal C'\), the component
\(\mathcal C\) is of cycle type and \(\mathcal C'\) is of lower-half type.
\end{enumerate}
\end{minipage}

The proof is identical to that of \textbf{(a$^{\uparrow}$)} and is omitted.

\smallskip
\noindent\textbf{(b$^{\uparrow}$)} If $\mathcal C_1=\mathcal C,\mathcal C_n=\mathcal C'$ are of upper-half type
and let \(\mathcal C_2,\dots,\mathcal C_{n-1}\)  be  all the outer upper-half
components \(\mathcal D\) with $\max(\mathcal C\cap[k])<\max(\mathcal D\cap[k])<\max(\mathcal C'\cap[k]),$
listed in increasing order of \(\max(\mathcal D\cap[k])\). WLOG, assume that
there are no upper-trivial components between \(\mathcal C\) and \(\mathcal C'\). We have: $\min(\mathcal U\cap[k])=\min(\mathcal C_1\cap[k]),
\max(\mathcal U\cap[k])=\max(\mathcal C_n\cap[k]).$
Moreover,$F_{\mathcal C_j}=[\min(V_l(\mathcal C_{j})\cap[l]), \max(V_r(\mathcal C_{j})\cap[l])]$
and $F_{\mathcal U}=[\min(V_l(\mathcal C)\cap[l]),\max(V_r(\mathcal C')\cap[l])].$   

Take $\vec x\in \bigcup_{\vec r,\vec d}(\theta_p^{\vec r}\cap\Omega_q^{\vec d})
\subset
\bigcup_{\vec r,\vec d}(\theta_{\bar p}^{\vec r}\cap\Omega_q^{\vec d}).$ By Corollary \ref{cor: constant interval}, Definition \ref{def: colour} and Remark  \ref{rem:label-represented-by-middle-vector}, the color of \(\mathcal U\) in \(q\cdot\bar p\) and the color of $\mathcal C_j$ in $q\cdot p$ can be represented by \(\vec x\). Since  there is no upper-half component \(\mathcal D\) satisfying $\max(\mathcal C_{j-1}\cap[k])<\max(\mathcal D\cap[k])<\max(\mathcal C_{j}\cap[k])$, we have:
\begin{equation}
\begin{aligned}
\mathrm{lab}_{q\cdot \bar p}(\mathcal U)
&=t_n\Bigl(\prod \vec x_{\restriction F_{\mathcal U}}\Bigr)^{-1}g_0^{-1}t_0^{-1}=t_nS_n^{-1} g_{n}^{-1} S_{n-1}^{-1} g_{n-1}^{-1} \cdots
   S_1^{-1} g_1^{-1} t_0^{-1} \\
&= (t_n S_n^{-1} g_{n}^{-1} t_{n-1}^{-1})\,
   (t_{n-1} S_{n-1}^{-1} g_{n-1}^{-1} t_{n-2}^{-1})
   \cdots
   (t_1 S_1^{-1} g_1^{-1} t_0^{-1}) \\
&=\mathrm{lab}_{q\cdot p}(\mathcal C')\,\mathrm{lab}_{q\cdot p}(\mathcal C_{n-1})\cdots
\mathrm{lab}_{q\cdot p}(\mathcal C).
\end{aligned}
\end{equation}

where
\[
S_j:=\prod_{i=\min(V_l(\mathcal C_{j})\cap[l])}^{\max(V_r(\mathcal C_{j})\cap[l])} x_i=f_{\mathcal C_j},
\qquad
g_{j}=\prod_{i=\max(V_r(\mathcal C_{j-1})\cap[l])+1}^{\min(V_l(\mathcal C_{j})\cap[l])-1} x_i=g_{\mathcal C_{j}} .
\]

On the other hand, since \(\delta_{q\cdot p}(\vec r,\vec d)=1\), for each
\(j=1,\dots,n\),
$\prod_{i=\min(\mathcal C_j\cap[k])}^{\max(\mathcal C_j\cap[k])} r_i
=\mathrm{lab}_{q\cdot p}(\mathcal C_j)^{-1}=
col_{q\cdot p}(\mathcal C_j)^{-1}.$ Multiplying these identities yields
$\prod_{i=\min(\mathcal U\cap[k])}^{\max(\mathcal U\cap[k])}r_i=\mathrm{lab}_{q\cdot\bar p}(\mathcal U)^{-1}=
col_{q\cdot\bar p}(\mathcal U)^{-1},$
which is exactly the defining relation for the upper-half component \(\mathcal U\).

\smallskip
\noindent\textbf{(b$^{\downarrow}$)}
If $\mathcal C_1=\mathcal C,\mathcal C_n=\mathcal C'$ are of lower-half type
and let \(\mathcal C_2,\dots,\mathcal C_{n-1}\) be  all the outer lower-half
components \(\mathcal D\) with $\max(\mathcal C\cap[k])<\max(\mathcal D\cap[k])<\max(\mathcal C'\cap[k]),$ listed in increasing order of \(\max(\mathcal D\cap[k])\).
The proof is identical to that of \textbf{(b)} and is omitted.

\medskip
\noindent\textbf{Case 2.}
\(\mathcal U=\mathcal C \cup \mathcal C'\) is a through type $(\bar{p},q)$ connected component.

\noindent\textbf{(a)} If \(\mathcal C\) is of through type while \(\mathcal C'\) is of upper-half and $\max(\mathcal C\cap[k])<\max(\mathcal C'\cap[k]).$
Let $\mathcal C_1=\mathcal C,\mathcal C_n=\mathcal C'$,
and let \(\mathcal C_2,\dots,\mathcal C_{n-1}\) be the outer upper-half
components \(\mathcal D\) with $\max(\mathcal C\cap[k])<\max(\mathcal D\cap[k])<\max(\mathcal C'\cap[k]),$
listed in increasing order of \(\max(\mathcal D\cap[k])\). WLOG, assume that
there are no upper-trivial components between \(\mathcal C\) and \(\mathcal C'\).

Then $\max(\mathcal U\cap[k])=\max(\mathcal C'\cap[k]), \max(\mathcal U\cap[m])=\max(\mathcal C\cap[m]),\max(\mathcal U\cap[l])=\max(\mathcal C'\cap[l])$ and $ V_r^{\uparrow}(\mathcal U)=V_r(\mathcal C')=V_r(\mathcal C_n),V_r^{\downarrow}(\mathcal U)=V_r^{\downarrow}(\mathcal C)=V_r^{\downarrow}(\mathcal C_1).$

Set $a:=\max(V_r^{\uparrow}(\mathcal U)), b:=\max(V_r^{\downarrow}(\mathcal U)),$
and define $\varepsilon(a,b):=
\begin{cases}
1,& a\le b,\\
-1,& a>b.
\end{cases}$
Further, let \(h,\mu\) denote the constants given by
Corollary~\ref{cor: constant interval} for the through component
\(\mathcal C=\mathcal C_1\), while \(h',\mu'\) denote the corresponding
constants for the through component \(\mathcal U\).

Take$\vec x\in \bigcup_{\vec r,\vec d}(\theta_p^{\vec r}\cap\Omega_q^{\vec d})
\subset
\bigcup_{\vec r,\vec d}(\theta_{\bar p}^{\vec r}\cap\Omega_q^{\vec d}).$
By Corollary~\ref{cor: constant interval} and Definition~\ref{def: colour},
the colors of \(\mathcal U\) in \(q\cdot\bar p\) and of \(\mathcal C_j\) in
\(q\cdot p\) are represented by the same vector \(\vec x\). Moreover, $h'(\mu')^{-1}
=
\left(
\prod_{i=\min\{a,b\}+1}^{\max\{a,b\}} x_i
\right)^{\varepsilon(a,b)}.$
Hence, by Definition~\ref{def: colour}, $\mathrm{lab}_{q\cdot\bar p}(\mathcal U)
=t_n\left(
\prod_{i=\min\{a,b\}+1}^{\max\{a,b\}} x_i
\right)^{\varepsilon(a,b)}\iota_1.$

For fixed $(\vec r,\vec d)\in \Gamma^k \times \Gamma^m$ satisfying $\delta_{q\cdot p}(\vec r,\vec d)=1$,
we consider the difference $X$ between $\prod_{i=1}^{\max(\mathcal C'\cap[k])} r_i$ and $\prod_{i=1}^{\max(\mathcal C\cap[m])} d_i$, i.e. $\Bigl(\prod_{i=1}^{\max(\mathcal C'\cap[k])} r_i\Bigr)X=\Bigl(\prod_{i=1}^{\max(\mathcal C\cap[m])} d_i\Bigr).$

Since \(\mathcal C_1=\mathcal C\) is of through type and
\(\delta_{q\cdot p}(\vec r,\vec d)=1\), its defining relation gives
\[
\Bigl(\prod_{i=1}^{\max(\mathcal C_1\cap[k])} r_i\Bigr)^{-1}
\Bigl(\prod_{i=1}^{\max(\mathcal C_1\cap[m])} d_i\Bigr)
=
\mathrm{lab}_{q\cdot p}(\mathcal C_1)
=
t_1h\mu^{-1}\iota_1.
\]
Using this relation together with the defining relations of the upper-half
components \(\mathcal C_2,\dots,\mathcal C_n\), we get
\begin{align*}
X
&=(\prod_{i=1}^{\max(\mathcal C_n\cap[k])} r_i)^{-1}
  \Bigl(\prod_{i=1}^{\max(\mathcal C_1\cap[m])} d_i\Bigr)=\Bigl[\prod_{j=2}^{n}\Bigl(\prod_{i=\min(\mathcal C_j\cap[k])}^{\max(\mathcal C_j\cap[k])} r_i\Bigr)\Bigr]^{-1}
  \Bigl(\prod_{i=1}^{\max(\mathcal C_1\cap[k])} r_i\Bigr)^{-1}
  \Bigl(\prod_{i=1}^{\max(\mathcal C_1\cap[m])} d_i\Bigr)\\
&=\Bigl[\prod_{j=2}^{n}\Bigl(\prod_{i=\min(\mathcal C_j\cap[k])}^{\max(\mathcal C_j\cap[k])} r_i\Bigr)\Bigr]^{-1}
  \,t_1h\mu^{-1}\iota_1= \mathrm{lab}_{q\cdot p}(\mathcal C_n)\,
   \mathrm{lab}_{q\cdot p}(\mathcal C_{n-1})\cdots
   \mathrm{lab}_{q\cdot p}(\mathcal C_2)\,
   t_1h\mu^{-1}\iota_1 \\
&= \bigl(t_nS_n^{-1}g_{n}^{-1}t_{n-1}^{-1}\bigr)
   \bigl(t_{n-1}S_{n-1}^{-1}g_{n-1}^{-1}t_{n-2}^{-1}\bigr)\cdots
   \bigl(t_2S_2^{-1}g_2^{-1}t_1^{-1}\bigr)\,
   t_1h\mu^{-1}\iota_1\\
&= t_nS_n^{-1}g_{n}^{-1}S_{n-1}^{-1}g_{n-1}^{-1}\cdots
   S_2^{-1}g_2^{-1}h\mu^{-1}\iota_1= t_n
   \left(
   \prod_{i=\min\{a,b\}+1}^{\max\{a,b\}} x_i
   \right)^{\varepsilon(a,b)}\iota_1
 = \mathrm{lab}_{q\cdot\bar p}(\mathcal U).
\end{align*}
Here $S_j:=\prod_{i=\min(V_l(\mathcal C_j)\cap[l])}^{\max(V_r(\mathcal C_j)\cap[l])}x_i=f_{\mathcal C_j},
\qquad
g_{j}:=\prod_{i=\max(V_r(\mathcal C_{j-1})\cap[l])+1}^{\min(V_l(\mathcal C_j)\cap[l])-1}x_i=g_{\mathcal C_j}.$

which is exactly the defining relation for the through component \(\mathcal U\).

\noindent\textbf{(a$'$)} If $\mathcal C$ is of through type while $\mathcal C'$ is of lower-half type and $\max(\mathcal C\cap[m])<\max(\mathcal C'\cap[m]).$ The proof is identical to that of \textbf{(a)} and is omitted.

\noindent\textbf{(b)} If \(\mathcal C\) is of upper-half type and \(\mathcal C'\) is of lower-half type,
let \(\mathcal T\) be the through-type connected component immediately preceding
\(\mathcal C\) on \([k]\), that is, $\max(\mathcal T\cap[k])<\max(\mathcal C\cap[k])$
and there is no through-type connected component \(\mathcal D\) with $\max(\mathcal T\cap[k])<\max(\mathcal D\cap[k])<\max(\mathcal C\cap[k]).$
Since noncrossingness preserves the order of through-type connected components on
\([k]\) and \([m]\), the same component \(\mathcal T\) is precisely the
through-type connected component immediately preceding \(\mathcal C'\) on \([m]\).

Let $\mathcal A_r=\mathcal C,$
where \(\mathcal A_1,\dots,\mathcal A_{r-1}\) are exactly the upper-half connected
components \(\mathcal D\) satisfying $\max(\mathcal T\cap[k])<\max(\mathcal D\cap[k])<\max(\mathcal C\cap[k]),$
listed in increasing order of \(\max(\mathcal D\cap[k])\).

Similarly, let $\mathcal B_s=\mathcal C',$
where \(\mathcal B_1,\dots,\mathcal B_{s-1}\) are exactly the lower-half connected
components \(\mathcal D\) satisfying $\max(\mathcal T\cap[m])<\max(\mathcal D\cap[m])<\max(\mathcal C'\cap[m]),$
listed in increasing order of \(\max(\mathcal D\cap[m])\).

Then $\max(\mathcal U\cap[k])=\max(\mathcal C\cap[k]),
\max(\mathcal U\cap[l])=\max(\mathcal C'\cap[l]),
\max(\mathcal U\cap[m])=\max(\mathcal C'\cap[m]);$
moreover $V_r^{\uparrow}(\mathcal U)=V_r(\mathcal C)=V_r(\mathcal A_r),V_r^{\downarrow}(\mathcal U)=V_r(\mathcal C')=V_r(\mathcal B_s).$

Let \(h,\mu\) denote the constants given by
Corollary~\ref{cor: constant interval} for the through component
\(\mathcal T\), while \(h',\mu'\) denote the corresponding
constants for the through component \(\mathcal U\).

Take $\vec x\in \bigcup_{\vec r,\vec d}(\theta_p^{\vec r}\cap\Omega_q^{\vec d})
\subset
\bigcup_{\vec r,\vec d}(\theta_{\bar p}^{\vec r}\cap\Omega_q^{\vec d}).$
By Corollary~\ref{cor: constant interval} and Definition~\ref{def: colour},
the colors of \(\mathcal U\) in \(q\cdot\bar p\) and of $\mathcal A_j$, $\mathcal B_j$ and $\mathcal T$ in
\(q\cdot p\) are represented by the same vector \(\vec x\).

For fixed $(\vec r,\vec d)\in \Gamma^k \times \Gamma^m$ satisfying $\delta_{q\cdot p}(\vec r,\vec d)=1$,
we consider the difference $X$ between $\prod_{i=1}^{\max(\mathcal C\cap[k])} r_i$ and $\prod_{i=1}^{\max(\mathcal C'\cap[m])} d_i$, i.e. $\Bigl(\prod_{i=1}^{\max(\mathcal C\cap[k])} r_i\Bigr)X=\Bigl(\prod_{i=1}^{\max(\mathcal C'\cap[m])} d_i\Bigr).$

Since \(\mathcal T\) is of through type and \(\delta_{q\cdot p}(\vec r,\vec d)=1\), its defining relation gives
\[
\Bigl(\prod_{i=1}^{\max(\mathcal T\cap[k])} r_i\Bigr)^{-1}
\Bigl(\prod_{i=1}^{\max(\mathcal T\cap[m])} d_i\Bigr)
=
\mathrm{lab}_{q\cdot p}(\mathcal T)
=
th\mu^{-1}\iota.
\]
Using this relation together with the defining relations of
\(\mathcal A_1,\dots,\mathcal A_r\) and \(\mathcal B_1,\dots,\mathcal B_s\), we get
\begin{align*}
X
&=(\prod_{i=1}^{\max(\mathcal C\cap[k])} r_i)^{-1}\Bigl(\prod_{i=1}^{\max(\mathcal C'\cap[m])} d_i\Bigr)\\
&=\Bigl[ \prod_{j=1}^{r}\Bigl(\prod_{i=\min(\mathcal A_j\cap[k])}^{\max(\mathcal A_j\cap[k])} r_i\Bigr)\Bigr]^{-1}
\Bigl(\prod_{i=1}^{\max(\mathcal T\cap[k])} r_i\Bigr)^{-1}
\Bigl(\prod_{i=1}^{\max(\mathcal T\cap[m])} d_i\Bigr)
\Bigl[ \prod_{j=1}^{s}\Bigl(\prod_{i=\min(\mathcal B_j\cap[m])}^{\max(\mathcal B_j\cap[m])} d_i\Bigr)\Bigr]\\
&=\Bigl[ \prod_{j=1}^{r}\Bigl(\prod_{i=\min(\mathcal A_j\cap[k])}^{\max(\mathcal A_j\cap[k])} r_i\Bigr)\Bigr]^{-1}
\,t_1h\mu^{-1}\iota_1\,
\Bigl[ \prod_{j=1}^{s}\Bigl(\prod_{i=\min(\mathcal B_j\cap[m])}^{\max(\mathcal B_j\cap[m])} d_i\Bigr)\Bigr]\\
&= [\mathrm{lab}_{q\cdot p}(\mathcal A_r)\,
   \mathrm{lab}_{q\cdot p}(\mathcal A_{r-1})\cdots
   \mathrm{lab}_{q\cdot p}(\mathcal A_1)]\,
   t_1h\mu^{-1}\iota_1\,
   [\mathrm{lab}_{q\cdot p}(\mathcal B_1)\,
   \mathrm{lab}_{q\cdot p}(\mathcal B_2)\cdots
   \mathrm{lab}_{q\cdot p}(\mathcal B_s)]\\
&= \bigl(t_rS_r^{-1}g_{r}^{-1}t_{r-1}^{-1}\bigr)
   \bigl(t_{r-1}S_{r-1}^{-1}g_{r-1}^{-1}t_{r-2}^{-1}\bigr)\cdots
   \bigl(t_1S_1^{-1}g_1^{-1}t^{-1}\bigr)\,
   th\mu^{-1}\iota
   (\iota^{-1}b_1\Theta_1\iota_1)\cdots(\iota_s^{-1}b_s\Theta_s\iota_s)\\
&= t_rS_r^{-1}g_{r}^{-1}S_{r-1}^{-1}g_{r-1}^{-1}\cdots
   S_1^{-1}g_1^{-1}h\mu^{-1}b_1\Theta_1\cdots b_s\Theta_s\iota_s\\
&= t_r
   \left(
   \prod_{i=\min\{a,b\}+1}^{\max\{a,b\}} x_i
   \right)^{\varepsilon(a,b)}\iota_s
 = \mathrm{lab}_{q\cdot\bar p}(\mathcal U).
\end{align*}
Here
\[
S_j:=\prod_{i=\min(V_l(\mathcal A_j)\cap[l])}^{\max(V_r(\mathcal A_j)\cap[l])}x_i=f_{\mathcal A_j},
\qquad
g_{j}:=\prod_{i=\max(V_r(\mathcal A_{j-1})\cap[l])+1}^{\min(V_l(\mathcal A_j)\cap[l])-1}x_i=g_{\mathcal A_j},
\]
\[
\Theta_j:=\prod_{i=\min(V_l(\mathcal B_j)\cap[l])}^{\max(V_r(\mathcal B_j)\cap[l])}x_i=f_{\mathcal B_j},
\qquad
b_{j}:=\prod_{i=\max(V_r(\mathcal B_{j-1})\cap[l])+1}^{\min(V_l(\mathcal B_j)\cap[l])-1}x_i=b_{\mathcal B_j}.
\]

\smallskip
\noindent\textbf{(b$'$)} If $\mathcal C$ is of upper-half type and $\mathcal C'$ is of lower-half type, and there is no preceding through-type connected component, then the proof is identical to that of \textbf{(b)}, except that all factors attached to the missing through-type component are formally set equal to $1$, namely
\[
t_1=h=\mu=\iota_1=
\Bigl(\prod_{i=1}^{\max(\mathcal T\cap[m])} d_i\Bigr)
=1.
\]
With this convention, the computation in \textbf{(b)} remains unchanged.

\noindent\textbf{(c)} If \(\mathcal C'\) is of through type and $\max(\mathcal C\cap[k])<\max(\mathcal C'\cap[k]),$
we have:$V_r^{\uparrow}(\mathcal U)=V_r^{\uparrow}(\mathcal C'), V_r^{\downarrow}(\mathcal U)=V_r^{\downarrow}(\mathcal C'),$ and $\max(\mathcal U\cap[k])=\max(\mathcal C'\cap[k]), \max(\mathcal U\cap[m])=\max(\mathcal C'\cap[m]),\max(\mathcal U\cap[l])=\max(\mathcal C'\cap[l]).$

\(h,\mu\) denote the constants given by Corollary~\ref{cor: constant interval}
for the through component \(\mathcal C'\), while \(h',\mu'\) denote the corresponding
constants for the through component \(\mathcal U\). Take $\vec x\in \bigcup_{\vec r,\vec d}(\theta_p^{\vec r}\cap\Omega_q^{\vec d})
\subset
\bigcup_{\vec r,\vec d}(\theta_{\bar p}^{\vec r}\cap\Omega_q^{\vec d}).$
By Corollary~\ref{cor: constant interval} and
Definition~\ref{def: colour}, the color of \(\mathcal U\) in \(q\cdot\bar p\) and $\mathcal C'\) in $ q\cdot p$  can be represented by \(\vec x\), therefore:
\[
\mathrm{lab}_{q\cdot\bar p}(\mathcal U)=t_{\mathcal C'} h'(\mu')^{-1}\iota_{\mathcal C'}
=t_{\mathcal C'}\Bigl(\prod_{i=\max(V_r^{\uparrow}(\mathcal U))+1}^{\max(\mathcal U\cap[l])}x_i\Bigr)\Bigl(\prod_{i=\max(V_r^{\downarrow}(\mathcal U))+1}^{\max(\mathcal U\cap[l])}x_i\Bigr)^{-1}\iota_{\mathcal C'}
=t_{\mathcal C'} h(\mu)^{-1}\iota_{\mathcal C'}=\mathrm{lab}_{q\cdot p}(\mathcal C').
\]
Thus the defining relation for \(\mathcal U\) is identical to that for \(\mathcal C'\), i.e.
\[
\prod_{i=1}^{\max(\mathcal C'\cap[k])}r_i
=\prod_{i=1}^{\max(\mathcal U\cap[k])} r_i
=
\Bigl(\prod_{i=1}^{\max(\mathcal U\cap[m])} d_i\Bigr)\,
\mathrm{lab}_{q\cdot\bar p}(\mathcal U)^{-1}
=
\Bigl(\prod_{i=1}^{\max(\mathcal C'\cap[m])} d_i\Bigr)\,
\mathrm{lab}_{q\cdot p}(\mathcal C')^{-1}.
\]

Therefore $\delta_{q\cdot\bar p}(\vec r,\vec d)=1.$
This proves the \textup{(p-case)}. The \textup{(q-case)} is proved in the same way.
\end{proof}

\begin{proposition}\label{Solution}
Let $p\in \NC_{\Lambda}(k,l)$ and $q\in \NC_{\Lambda}(l,m)$ with
$T_q\circ T_p\neq 0$, and let $(q\cdot p,\vec u)$ be the coloured non-crossing partition
defined in Definition~\ref{def: colour}. Then, for every
$\vec r\in\Gamma^k$ and $\vec d\in\Gamma^m$, we have
\[
\delta_{q\cdot p}(\vec r,\vec d)=1
\quad\Longleftrightarrow\quad
\theta_p^{\vec r}\cap\Omega_q^{\vec d}\neq\emptyset.
\]
\end{proposition}

\begin{proof}

The implication $(\Longleftarrow)$ follows directly from the construction of $q\cdot p$.

Assume that \(\theta_p^{\vec r}\cap\Omega_q^{\vec d}\neq\emptyset\), and choose
\(\vec x\in \theta_p^{\vec r}\cap\Omega_q^{\vec d}\).
For each no-cycle \((p,q)\)-connected component \(\mathcal C\),
Corollary~\ref{colour of components} gives exactly the required product identity,
and by Remark~\ref{rem:label-represented-by-middle-vector} the quantities
appearing there are precisely those used in Definition~\ref{def: colour} to define
\(\mathrm{lab}(\mathcal C)\). Hence the colour condition for every block of
\(q\cdot p\) is satisfied, so \(\delta_{q\cdot p}(\vec r,\vec d)=1\).

$(\Longrightarrow)$ fix $\vec r\in\Gamma^k$ and $\vec d\in\Gamma^m$ such that $\delta_{q\cdot p}(\vec r,\vec d)=1.$ We must show that the system $\mathcal E_{\vec r,\vec d}^{p,q}$ admits a solution. We prove the statement by induction on $|K(q,p)|$.

If $|K(q,p)|=1$, the claim follows from Corollary~\ref{|K(q,p)|=1}.

Now assume that the statement holds whenever $|K(q,p)|=n$, and let $(p,q)$ satisfy
$|K(q,p)|=n+1$.

If the intervals \(H_{\mathcal C}\), \(\mathcal C\in K(q,p)\), are pairwise disjoint,
then the system \(\mathcal E_{\vec r,\vec d}^{p,q}\) splits into independent subsystems
indexed by the connected components \(\mathcal C\in K(q,p)\). More precisely, for each
\(\mathcal C\in K(q,p)\), define $K_{\mathcal C}:=\bigcup_{B\in \mathcal C_p}(B\cap[k]), L_{\mathcal C}:=H_{\mathcal C},
M_{\mathcal C}:=\bigcup_{D\in \mathcal C_q}(D\cap[m]),$
and let $\rho_{\mathcal C}^{k}:K_{\mathcal C}\to[k_{\mathcal C}],
\rho_{\mathcal C}^{l}:L_{\mathcal C}\to[l_{\mathcal C}],
\rho_{\mathcal C}^{m}:M_{\mathcal C}\to[m_{\mathcal C}]$
be the order-preserving bijections. We denote by $p_{\mathcal C}\in \NC_{\Lambda}(k_{\mathcal C},l_{\mathcal C}),
q_{\mathcal C}\in \NC_{\Lambda}(l_{\mathcal C},m_{\mathcal C})$
the coloured partitions obtained from the block families \(\mathcal C_p\) and
\(\mathcal C_q\) by applying these order-preserving relabellings, with colours inherited
from \(p\) and \(q\). We also write $\vec r_{\mathcal C}:=(r_i)_{i\in K_{\mathcal C}},\vec d_{\mathcal C}:=(d_i)_{i\in M_{\mathcal C}},$
where the entries are listed in increasing order.

Then the subsystem of \(\mathcal E_{\vec r,\vec d}^{p,q}\) involving the variables
\((x_i)_{i\in H_{\mathcal C}}\) is, after the relabelling $y_{\rho_{\mathcal C}^{l}(i)}:=x_i(i\in H_{\mathcal C}),$
precisely $\mathcal E_{\vec r_{\mathcal C},\vec d_{\mathcal C}}^{p_{\mathcal C},q_{\mathcal C}}.$
Since the intervals \(H_{\mathcal C}\) are pairwise disjoint, these subsystems have
disjoint sets of variables. Moreover, by construction, $|K(q_{\mathcal C},p_{\mathcal C})|=1.$
The assumption \(\delta_{q\cdot p}(\vec r,\vec d)=1\) restricts to
$\delta_{q_{\mathcal C}\cdot p_{\mathcal C}}
(\vec r_{\mathcal C},\vec d_{\mathcal C})=1(\mathcal C\in K(q,p)).$
Therefore, by Corollary~\ref{|K(q,p)|=1}, each subsystem admits a solution. Since the
sets \(H_{\mathcal C}\) are pairwise disjoint, these solutions can be concatenated to give
a solution of the whole system \(\mathcal E_{\vec r,\vec d}^{p,q}\).

Assume now that the intervals \(H_{\mathcal C}\), \(\mathcal C\in K(q,p)\), are not pairwise disjoint.
By Lemma~\ref{lem:gluing-step} and Corollary~\ref{cor:gluing-consequences}, one of the two cases in Lemma~\ref{lem:gluing-step} occurs. We only treat the \textup{(p-case)}; the \textup{(q-case)} is analogous.

Thus there exist connected components \(\mathcal C,\mathcal C'\in K(q,p)\) and blocks
$B=\pi_p(\min \mathcal C')\in \mathrm{FreeOut}_p^{\downarrow}(\mathcal C'),D\in \mathcal C_p,$
such that $\operatorname{span}(B)\subsetneq \operatorname{span}(D),$
and such that the gluing
\[
\bar p:=(p\setminus\{B,D\})\cup\{D\cup B\}
\]
satisfies $|K(q,\bar p)|=|K(q,p)|-1=n$
and $\theta_{\bar p}^{\vec r}\cap\Omega_q^{\vec d}
\supseteq
\theta_p^{\vec r}\cap\Omega_q^{\vec d}$
for every \(\vec r\in\Gamma^k\) and \(\vec d\in\Gamma^m\).

Moreover, by Proposition~\ref{prop:delta-under-gluing}, $\delta_{q\cdot p}(\vec r,\vec d)=1
\quad\Longrightarrow\quad
\delta_{q\cdot\bar p}(\vec r,\vec d)=1.$

Since \(T_q\circ T_p\neq 0\), there exist $\vec r^{\,\prime}\in\Gamma^k,\vec d^{\,\prime}\in\Gamma^m$
such that$\theta_p^{\vec r^{\,\prime}}\cap\Omega_q^{\vec d^{\,\prime}}\neq\emptyset.$
Hence $\theta_{\bar p}^{\vec r^{\,\prime}}\cap\Omega_q^{\vec d^{\,\prime}}\neq\emptyset,$
so that \(T_q\circ T_{\bar p}\neq 0\). Therefore the inductive hypothesis applies to the pair \((\bar p,q)\). Since $|K(q,\bar p)|=n\text{and}
\delta_{q\cdot\bar p}(\vec r,\vec d)=1,$
we obtain $\theta_{\bar p}^{\vec r}\cap\Omega_q^{\vec d}\neq\emptyset.$
Choose $\vec y\in \theta_{\bar p}^{\vec r}\cap\Omega_q^{\vec d}.$

It remains to prove that \(\vec y\in \theta_p^{\vec r}\cap\Omega_q^{\vec d}\), i.e. that $\prod_{i=\min B}^{\max B} y_i=\operatorname{col}(B).$

For this, fix
\(\vec z\in \theta_p^{\vec r^{\,\prime}}\cap\Omega_q^{\vec d^{\,\prime}}\),
and compare \(\vec y\) and \(\vec z\) on the component \(\mathcal C'\).
By noncrossingness and the choice of \(B\), we have
$\max(\operatorname{span}(B))<\min L$   for all  $L\in Th_p(\mathcal C)
\quad
\bigl(\text{resp. } L\in Th_q(\mathcal C')\bigr),
$
if \(B=\pi_p(\min\mathcal C')\)
\(\bigl(\text{resp. } B=\pi_q(\min\mathcal C')\bigr)\).

Take any $\vec z\in \theta_p^{\vec r'}\cap \Omega_q^{\vec d'}$.
Fix an entrance $E$ of $\mathcal C'$.
By the definition of $\bar p$ and Proposition \ref{constant entrance }, we have
$\vec y,\vec z\in \mathcal S_E(q,p)$, hence $\prod \vec y_{\restriction E}=\prod \vec z_{\restriction E}$.

Moreover, for any $W\in\mathrm{FreeOut}_p^{\downarrow}(\mathcal C')
\ \sqcup\
\mathrm{FreeOut}_q^{\uparrow}(\mathcal C'),
$
we have
$\prod\vec{y}_{\restriction \mathrm{span}(W)}
=
\prod\vec{z}_{\restriction \mathrm{span}(W)}
=
\mathrm{col}(W).$

\medskip
\noindent\textbf{Case 1.} Assume $\mathcal C'$ is of upper-half type (the lower-half case is analogous).
Since $\delta_{q\cdot p}(\vec r,\vec d)=1=\delta_{q\cdot p}(\vec r',\vec d')$, we obtain
\[
\prod_{\mathcal C'\cap[k]}\vec r
=\mathrm{lab}(\mathcal C')^{-1}
=
\prod_{\mathcal C'\cap[k]}\vec r'.
\]
As $\vec y\in\theta_{\bar p}^{\vec r}\cap\Omega_q^{\vec d}$ and
$\vec z\in\theta_p^{\vec r'}\cap\Omega_q^{\vec d'}$, we have
$\delta_{\bar p}(\vec r,\vec y)=1=\delta_p(\vec r',\vec z)$, hence
\[
\prod\vec y_{\restriction F_{\mathcal C'}}
=
g^{-1}(t')^{-1}\Bigl(\prod_{\mathcal C'\cap[k]}\vec r\Bigr)t
=
g^{-1}(t')^{-1}\Bigl(\prod_{\mathcal C'\cap[k]}\vec r'\Bigr)t
=
\prod\vec z_{\restriction F_{\mathcal C'}}.
\]

Let $H_{\mathcal C'}=[\min(\mathcal C'\cap[\ell]),\max(\mathcal C'\cap[\ell])]$.
Now consider $\mathcal C'_q$ together with all lower entrances of $\mathcal C'$. Because
\[
H_{\mathcal C'}
=
\Bigl(\bigsqcup_{E\in \mathrm{Ent}^{\downarrow}(\mathcal C')}\mathrm{span}(E)\Bigr)
\ \bigsqcup\
\Bigl(\bigsqcup_{W\in{\mathrm{FreeOut}_q^{\uparrow}}(\mathcal C')} \mathrm{span}(W)\Bigr),
\]
on every component of this decomposition we already proved
$\prod \vec y_{\restriction(\cdot)}=\prod \vec z_{\restriction(\cdot)}$.
This yields
\[
\prod_{i=\min(\mathcal C'\cap[\ell])}^{\max(\mathcal C'\cap[\ell])}y_i
=
\prod_{i=\min(\mathcal C'\cap[\ell])}^{\max(\mathcal C'\cap[\ell])}z_i.
\]

Consider $\mathcal C'_p\setminus\{B\}$ together with all upper entrances of $\mathcal C'$. Because
\[
H_{\mathcal C'}
=
\Bigl(\bigsqcup_{E\in \mathrm{Ent}^{\uparrow}(\mathcal C')}\mathrm{span}(E)\Bigr)
\ \bigsqcup\
\Bigl(\bigsqcup_{W\in \mathrm{FreeOut}_p^{\downarrow}(\mathcal O') }\mathrm{span}(W)\Bigr)
\ \bigsqcup\
F_{\mathcal C'}
\ \bigsqcup\
\mathrm{span}(B),
\]
on every component of this decomposition except $\mathrm{span}(B)$
we already proved
$\prod \vec y_{\restriction(\cdot)}=\prod \vec z_{\restriction(\cdot)}$.
This yields
\[
\prod_{i=\min(\mathcal C'\cap[\ell])}^{\min(B)-1}y_i
=
\prod_{i=\min(\mathcal C'\cap[\ell])}^{\min(B)-1}z_i,
\qquad
\prod_{i=\max(B)+1}^{\max(\mathcal C'\cap[\ell])}y_i
=
\prod_{i=\max(B)+1}^{\max(\mathcal C'\cap[\ell])}z_i.
\]

Therefore, taking the product over the decomposition of $H_{\mathcal C'}$ above and cancelling
the already-matched factors, we conclude
\[
\prod\vec y_{\restriction \mathrm{span}(B)}
=
\mathrm{col}(B)
=
\prod\vec z_{\restriction \mathrm{span}(B)}.
\]

\medskip
\noindent\textbf{Case 2.} Assume $\mathcal C'$ is of through type.

 Write the preceding through-type component left to $\mathcal C'$ as $\mathcal C_0$. WLOG, we assume there are no upper(lower)-trivial components between $\mathcal C_0$ and $\mathcal C'$. Write $\{\mathcal C_1,\dots,\mathcal C_{u_n}\}$ (resp. $\{\mathcal O_1,\dots,\mathcal O_{d_n}\}$)is precisely the set of all outer upper-half-type(resp. lower) components between $\mathcal C_0$ and $\mathcal C'$.

Since 
\(\delta_{q\cdot p}(\vec r,\vec d)=1=\delta_{q\cdot p}(\vec r',\vec d')\),
then for each  component \(\mathcal C_j\) \((1\le j\le u_n)\) and each component
\(\mathcal O_j\) \((1\le j\le d_n)\) we have
$\prod_{ \mathcal C_j\cap[k]} \vec r=\prod_{ \mathcal C_j\cap[k]} \vec r',
\prod_{ \mathcal O_j\cap[m]} \vec d=\prod_{ \mathcal O_j\cap[m]} \vec d'.$
Moreover, the component \(\mathcal C_0\) yields $\Bigl(\prod_{i=1}^{\max(\mathcal C_0\cap [k])} r_i\Bigr)\,t_0h_0\mu_0^{-1}\iota_0
=
\prod_{i=1}^{\max(\mathcal C_0\cap [m])} d_i$. For the  component \(\mathcal C'\), combining the defining relations gives:
\[
(\prod_{[1,\max{C_0\cap [k]}]}\vec r)\Bigl(\prod_{j=1}^{u_n}\prod_{ \mathcal C_j\cap[k]} \vec r\Bigr)
\,t_{u_n}g_{u_n}\Bigl(g_{u_n}^{-1}t_{u_n}^{-1}(\prod_{ \mathcal C'\cap[k]}\vec r)\,t\Bigr)\,
h\mu^{-1}\iota
=
(\prod_{[1,\max{C_0\cap [m]]}}\vec d)\Bigl(\prod_{j=1}^{d_n}\prod_{ O_j\cap[m]} \vec d\Bigr)
(\prod_{ \mathcal C'\cap[k]} \vec d).
\]
Using the identity coming from \(\mathcal C_0\), we can rewrite the right-hand side as
$$
(\prod_{[1,\max{C_0\cap [k]}]}\vec r)\,t_0h_0\mu_0^{-1}\iota_0
\Bigl(\prod_{j=1}^{d_n}\prod_{ \mathcal O_j\cap[m]} \vec d\Bigr)
(\prod_{ \mathcal C'\cap[m]} \vec d)$$
$$=
(\prod_{[1,\max{C_0\cap [k]}]}\vec r)\,t_0h_0\mu_0^{-1}
\Bigl(\iota_0(\prod_{j=1}^{d_n}\prod_{ \mathcal O_j\cap[m]} \vec d)\,
\iota_{d_n}^{-1}b_{d_n}\Bigr)
\Bigl(b_{d_n}^{-1}\iota_{d_n}(\prod_{\mathcal C'\cap[m]} \vec d)\Bigr).
$$

Then:
\[
\Bigl(\prod_{j=1}^{u_n}\prod_{ \mathcal C_j\cap[k]} \vec r\Bigr)
\,t_{u_n}g_{u_n}\Bigl(g_{u_n}^{-1}t_{u_n}^{-1}(\prod_{ \mathcal C'\cap[k]}\vec r)\,t\Bigr)\,
h\mu^{-1}\iota
=
t_0h_0\mu_0^{-1}
\Bigl(\iota_0(\prod_{j=1}^{d_n}\prod_{ \mathcal O_j\cap[m]} \vec d)\,
\iota_{d_n}^{-1}b_{d_n}\Bigr)
\Bigl(b_{d_n}^{-1}\iota_{d_n}(\prod_{\mathcal C'\cap[m]} \vec d)\Bigr).
\]

Next, from 
\(\delta_{\bar p}(\vec r,\vec y)=1=\delta_{\bar p}(\vec r',\vec z)\) and 
\(\delta_q(\vec y,\vec d)=1=\delta_q(\vec z,\vec d')\),
we obtain the identifications:
$\prod_{i=\min V_l^{\uparrow}(\mathcal C')}^{\max V_r^{\uparrow}(\mathcal C')} y_i
=
g_{u_n}^{-1}t_{u_n}^{-1}(\prod_{ \mathcal C'\cap[k]}\vec r)\,t,$
$\prod_{i=\min V_l^{\uparrow}(\mathcal C')}^{\max V_r^{\uparrow}(\mathcal C')} z_i
=g_{u_n}^{-1}t_{u_n}^{-1}(\prod_{ \mathcal C'\cap[k]}\vec r')\,t$
and
$\prod_{i=\min V_l^{\downarrow}(\mathcal C')}^{\max V_r^{\downarrow}(\mathcal C')} y_i
=b_{d_n}^{-1}\iota_{d_n}(\prod_{\mathcal C'\cap[m]} \vec d)
$, $\prod_{i=\min V_l^{\downarrow}(\mathcal C')}^{\max V_r^{\downarrow}(\mathcal C')} z_i
=b_{d_n}^{-1}\iota_{d_n}(\prod_{\mathcal C'\cap[m]} \vec d')
.$

Therefore,
$\left(\prod_{i=\min V_l^{\uparrow}(\mathcal C')}^{\max V_r^{\uparrow}(\mathcal C')} y_i\right)
h\mu^{-1}
\left(\prod_{i=\min V_l^{\downarrow}(\mathcal C')}^{\max V_r^{\downarrow}(\mathcal C')} y_i\right)^{-1}
=
C_+^{-1}h_0\mu_0^{-1}C_-
=
(C_+')^{-1}h_0\mu_0^{-1}C_-'
=
\left(\prod_{i=\min V_l^{\uparrow}(\mathcal C')}^{\max V_r^{\uparrow}(\mathcal C')} z_i\right)
h\mu^{-1}
\left(\prod_{i=\min V_l^{\downarrow}(\mathcal C')}^{\max V_r^{\downarrow}(\mathcal C')} z_i\right)^{-1}$, where we set
\[
C_+:=t_0^{-1}\Bigl(\prod_{j=1}^{u_n}\prod_{ \mathcal C_j\cap[k]} \vec r\Bigr)t_{u_n}g_{u_n},
\qquad
C_-:=\iota_0\Bigl(\prod_{j=1}^{d_n}\prod_{ \mathcal O_j\cap[m]} \vec d\Bigr)
b_{d_n}^{-1}\iota_{d_n},
\]
and similarly \(C_+'\) (resp. \(C_-'\)) is defined by replacing \((\vec r,\vec d)\) with \((\vec r',\vec d')\).
By $\prod_{ \mathcal C_j\cap[k]} \vec r=\prod_{ \mathcal C_j\cap[k]} \vec r',
\prod_{ \mathcal O_j\cap[m]} \vec d=\prod_{ \mathcal O_j\cap[m]} \vec d'.$ , we have \(C_+=C_+'\) and \(C_-=C_-'\).  

Since \(H_{\mathcal C'}\) decomposes into the disjoint union of
the  upper(resp. lower) entrances of \(\mathcal C'\), the spans of outer blocks in \(\mathcal C'_p\) (resp. \(\mathcal C'_q\)),
and the interval \([\min V_l^{\uparrow}(\mathcal C'),\,\max V_r^{\uparrow}(\mathcal C')]\)
(resp. \([\min V_l^{\downarrow}(\mathcal C'),\,\max V_r^{\downarrow}(\mathcal C')]\)),
we can compare the corresponding products along these pieces.

In particular, we have
\[
(\prod \vec y_{\restriction \mathrm{span}(B)})
\Bigl(\prod_{i=\max B+1}^{\min V_l^{\uparrow}(\mathcal C')-1} y_i\Bigr)
\Bigl(\prod_{i=\min V_l^{\uparrow}(\mathcal C')}^{\max V_r^{\uparrow}(\mathcal C')} y_i\Bigr)\,h
=
\Bigl(\prod_{i=\min B}^{\min V_l^{\downarrow}(\mathcal C')-1} y_i\Bigr)
\Bigl(\prod_{i=\min V_l^{\downarrow}(\mathcal C')}^{\max V_r^{\downarrow}(\mathcal C')} y_i\Bigr)\,\mu .
\]
Using$
\Bigl(\prod_{i=\min V_l^{\uparrow}(\mathcal C')}^{\max V_r^{\uparrow}(\mathcal C')} y_i\Bigr)\,h\mu^{-1}\,
\Bigl(\prod_{i=\min V_l^{\downarrow}(\mathcal C')}^{\max V_r^{\downarrow}(\mathcal C')} y_i\Bigr)^{-1}
= C_+^{-1}h_0\mu_0^{-1}C_-,$
we obtain
$$\prod \vec y_{\restriction \mathrm{span}(B)}
=
\Bigl(\prod_{i=\min B}^{\min V_l^{\downarrow}(\mathcal C')-1} y_i\Bigr)\,
\bigl(C_+^{-1}h_0\mu_0^{-1}C_-\bigr)^{-1}.$$
Now \(\vec y\) and \(\vec z\) have identical restriction products on every entrance and on each block contained in
\([\min B,\,\min V_l^{\downarrow}(\mathcal C')-1]\); hence$\prod_{i=\min B}^{\min V_l^{\downarrow}(\mathcal C')-1} y_i
=
\prod_{i=\min B}^{\min V_l^{\downarrow}(\mathcal C')-1} z_i .$
Consequently, $
\prod \vec y_{\restriction \mathrm{span}(B)}
=
\prod \vec z_{\restriction \mathrm{span}(B)}
=
\mathrm{col}_p(B).$

\end{proof}

\begin{proposition}[Vertical composition]
Let
$p\in \NC_{\Lambda}\bigl((k,l),(g_1,\dots,g_k;\,h_1,\dots,h_l)\bigr), q\in \NC_{\Lambda}\bigl((l,m),(h_1,\dots,h_l;\,d_1,\dots,d_m)\bigr).$ Then either$T_q\circ T_p=0,$
or $T_q\circ T_p=|\Lambda|^{\,c(\mathcal G)-1}T_{q\cdot p}.$
Consequently, the linear span of the operators induced by colored partitions is closed under composition.
\end{proposition}
\begin{proof}
By Corollary \ref{Number of solutions} and Proposition \ref{Solution}, we obtain: when $T_q\circ T_p\neq 0$

\[
T_p\circ T_q(e_{\vec r})
=\sum_{\vec s\in\theta_p^{\vec r}}\ \sum_{\vec d\in\Lambda^{m}:\ \vec s\in\Omega_q^{\vec d}} e_{\vec d}
=|\Lambda|^{\,c-1}\sum_{\vec d\in\Lambda^{m}:\theta_p^{\vec{r}} \cap \Omega_q^{\vec{d}}\neq \emptyset} e_{\vec d}=|\Lambda|^{\,c-1}\sum_{\vec d\in\Lambda^{m}:\delta_{q\cdot p}(\vec{r},\vec{d})=1} e_{\vec d}=|\Lambda|^{\,c-1}T_{q\cdot p}(e_{\vec r})
\]
where $c=c(\mathcal G)$ be the number of connected components of the gain graph $\mathcal G$ associated to\((p^{-},q^{+})\)
\end{proof}

\subsection{Construction of the category $\mathcal C_{\Gamma,\Lambda}$}
\begin{proposition}
Let $p\in \NC_{\Lambda}\bigl((k,l),(g_1,\dots,g_k;\,h_1,\dots,h_l)\bigr), q\in \NC_{\Lambda}\bigl((m,n),(g'_1,\dots,g'_m;\,h'_1,\dots,h'_n)\bigr).$ Then:
\begin{enumerate}

\item \textbf{Horizontal concatenation (tensor product).}
We have $T_p\otimes T_q = T_{p\otimes q},$
where $p\otimes q$ is obtained from $p$ and $q$ by horizontal concatenation, keeping the
$\Lambda$-color of each block unchanged, and $p\otimes q \in
\NC_{\Lambda}\bigl((k+m,l+n),
(g_1,\dots,g_k,g'_1,\dots,g'_m;\,h_1,\dots,h_l,h'_1,\dots,h'_n)\bigr).$

\item \textbf{Reflection (adjoint).}
For \(p\in \NC_{\Lambda}\bigl((k,l),(g_1,\dots,g_k;\,h_1,\dots,h_l)\bigr)\), define
\(p^{*}\) to be the partition obtained by reflecting \(p\) with respect to a horizontal
line between the two rows of points, and replacing the color \(col(V)\) of each block \(V\)
by its inverse \(col(V)^{-1}\).
Then $p^{*}\in \NC_{\Lambda}\bigl((l,k),(h_1,\dots,h_l;\,g_1,\dots,g_k)\bigr),$
and $T_p^{*}=T_{p^{*}}.$

\end{enumerate}

\end{proposition} 

\begin{proof}
\begin{enumerate}
\item Identify $e_{\vec r}\otimes e_{\vec r'}$ with $e_{(\vec r,\vec r')}$.
It is enough to prove
$\delta_{p\otimes q}\bigl((\vec r,\vec r'),(\vec s,\vec s')\bigr)
=
\delta_p(\vec r,\vec s)\,\delta_q(\vec r',\vec s').$

We first note that if $\delta_p(\vec r,\vec s)=1$, then$\prod_{i=1}^k r_i=\prod_{j=1}^l s_j.$
 Indeed, multiplying over $\partial p=\{V_1\prec\cdots\prec V_r\}$ in the boundary order,
one gets $\prod_{V\in\partial p}^{\prec} t_V
=
\left(\prod_{i=1}^k r_i\right)^{-1}\left(\prod_{j=1}^l s_j\right).$
Since $(p,\vec t)\in\NC_\Lambda(k,l)$, the left-hand side is equal to $1$, hence
$\prod_{i=1}^k r_i=\prod_{j=1}^l s_j.$

Now assume that $\delta_p(\vec r,\vec s)=1$ and $\delta_q(\vec r',\vec s')=1$.
For blocks coming from $p$, the defining relations in $p\otimes q$ are exactly those of $p$.
For single-layer blocks coming from $q$, the shifted relations are exactly those of $q$.
If $W$ is a through-block of $q$, then in $p\otimes q$ its relation becomes
$\left(\prod_{i=1}^{l} s_i\right)\left(\prod_{i=1}^{\max W_-} s'_i\right)
=
\left(\prod_{j=1}^{k} r_j\right)\left(\prod_{j=1}^{\max W_+} r'_j\right)t_W,$
which, using $\prod_{i=1}^k r_i=\prod_{j=1}^l s_j$, is equivalent to the original relation
$\prod_{i=1}^{\max W_-} s'_i
=
\left(\prod_{j=1}^{\max W_+} r'_j\right)t_W.$
Hence
$\delta_{p\otimes q}\bigl((\vec r,\vec r'),(\vec s,\vec s')\bigr)=1.$

Conversely, if $\delta_{p\otimes q}\bigl((\vec r,\vec r'),(\vec s,\vec s')\bigr)=1,$
then restricting to the blocks of $p$ yields $\delta_p(\vec r,\vec s)=1$, hence $\prod_{i=1}^k r_i=\prod_{j=1}^l s_j.$
Using this identity, the defining relations for the shifted blocks coming from $q$ reduce exactly to
those of $q$, so $\delta_q(\vec r',\vec s')=1$.
Therefore $\delta_{p\otimes q}\bigl((\vec r,\vec r'),(\vec s,\vec s')\bigr)
=
\delta_p(\vec r,\vec s)\,\delta_q(\vec r',\vec s').$
This proves $T_p\otimes T_q=T_{p\otimes q}$.

\item  $\langle T_p e_{\vec r},e_{\vec s}\rangle=\delta_p(\vec r,\vec s)$, so it suffices to prove that $\delta_p(\vec r,\vec s)=\delta_{p^{*}}(\vec s,\vec r),\forall\,\vec r\in\Lambda^k,\ \vec s\in\Lambda^l.$
Fix a block $V$ of $p$, and let $V^{*}$ be the corresponding block of $p^{*}$. We compare the defining relations.

If $V_{-}\neq\emptyset$ and $V_{+}\neq\emptyset$, then for $p$ we have$\prod_{i=1}^{\max V_-} s_i=\left(\prod_{j=1}^{\max V_+} r_j\right)t_V$
For $p^{*}$, since upper and lower rows are exchanged and $t_{V^{*}}=t_V^{-1}$, the condition becomes$\prod_{j=1}^{\max V_+} r_j=\left(\prod_{i=1}^{\max V_-} s_i\right)t_{V^{*}}=\left(\prod_{i=1}^{\max V_-} s_i\right)t_V^{-1},$
which is equivalent to the previous one.

If $V_{-}\neq\emptyset$ and $V_{+}=\emptyset$, then for $p$ we have
$\prod_{i=\min V_-}^{\max V_-} s_i=t_V.$
After reflection, $V^{*}_{-}=\emptyset$ and $V^{*}_{+}\neq\emptyset$, so the condition for $p^{*}$ is
$\left(\prod_{i=\min V_-}^{\max V_-} s_i\right)t_{V^{*}}=1
\quad\Longleftrightarrow\quad
\left(\prod_{i=\min V_-}^{\max V_-} s_i\right)t_V^{-1}=1,$
hence again $\prod_{i=\min V_-}^{\max V_-} s_i=t_V$.

If $V_{-}=\emptyset$ and $V_{+}\neq\emptyset$, then for $p$ we have
$\left(\prod_{j=\min V_+}^{\max V_+} r_j\right)t_V=1.$
After reflection, $V^{*}_{-}\neq\emptyset$ and $V^{*}_{+}=\emptyset$, so the condition for $p^{*}$ is$\prod_{j=\min V_+}^{\max V_+} r_j=t_{V^{*}}=t_V^{-1},$
which is equivalent to $\left(\prod_{j=\min V_+}^{\max V_+} r_j\right)t_V=1$.

Thus each block relation for $p$ is equivalent to the corresponding block relation for $p^{*}$, so$
\delta_p(\vec r,\vec s)=\delta_{p^{*}}(\vec s,\vec r).$
Therefore
$\langle T_p e_{\vec r},e_{\vec s}\rangle
=\delta_p(\vec r,\vec s)
=\delta_{p^{*}}(\vec s,\vec r)
=\langle e_{\vec r},T_{p^{*}}e_{\vec s}\rangle,$
and hence $T_p^{*}=T_{p^{*}}$.
\end{enumerate}
\end{proof}
\begin{definition}
Let $p\in NC(k,l)$. We define $Lrot^k(p)\in NC(0,k+l)$ as follows.

Label the points of $Lrot^k(p)$ by $\{1,\dots,k+l\}$. Define a map
\[
\theta:\{1,\dots,k+l\}\to [k]\sqcup [l]
\]
by
\[
\theta(i):=
\begin{cases}
k+1-i \in [k], & 1\le i\le k,\\
i-k \in [l], & k+1\le i\le k+l.
\end{cases}
\]
Then $a,b\in \{1,\dots,k+l\}$ lie in the same block of $Lrot^k(p)$ if and only if
$\theta(a)$ and $\theta(b)$ lie in the same block of $p$.

If $(p,\vec t)$ is a $\Lambda$-colored partition, we define $
Lrot^k(p,\vec t):=(Lrot^k(p),\vec t),$
that is, the color of each block is unchanged.
\end{definition}

\begin{corollary}
For $(p,\vec t)\in NC_{\Lambda}\bigl((k,l),(g_1,\dots,g_k;\,h_1,\dots,h_l)\bigr)$, the map $T_{(p,\vec t)}$ belongs to $\Mor\bigl(u(g_1)\otimes u(g_2)\otimes\cdots\otimes u(g_k),\,
u(h_1)\otimes u(h_2)\otimes\cdots\otimes u(h_l)\bigr).$
\end{corollary}

\begin{proof}
Let $(p',\vec t):=(Lrot^k(p),\vec t)\in
\NC_{\Lambda}\bigl((0,k+l),(g_k^{-1},\dots,g_1^{-1},h_1,\dots,h_l)\bigr).$

Let $\cup_k:=\bigl\{\{1,2k\},\{2,2k-1\},\dots,\{k,k+1\}\bigr\}
\in \NC\bigl((0,2k),(g_1,\dots,g_k,g_k^{-1},\dots,g_1^{-1})\bigr),$
and let $(\cup_k,\vec e)$ be the corresponding element of $NC_{\Lambda}\bigl((0,2k),(g_1,\dots,g_k,g_k^{-1},\dots,g_1^{-1})\bigr)$
such that every block is decorated by the unit element. Set $(\cap_k,\vec e):=(\cup_k,\vec e)^*.$
Then $(\cap_k,\vec e)\in
NC_{\Lambda}\bigl((2k,0),(g_1,\dots,g_k,g_k^{-1},\dots,g_1^{-1})\bigr),$
and, by the adjoint property proved above, $T_{(\cap_k,\vec e)}=T_{(\cup_k,\vec e)}^*.$

We claim that $T_{(p,\vec t)}
=\bigl(T_{(\cap_k,\vec e)}\otimes id^{\otimes l}\bigr)
\circ
\bigl(id^{\otimes k}\otimes T_{(p',\vec t)}\bigr).$ Indeed, fix $\vec r=(r_1,\dots,r_k)\in\Lambda^k$. By definition,
\[
T_{(p',\vec t)}(1)
=
\sum_{\vec a\in\Lambda^k,\ \vec s\in\Lambda^l}
\delta_{p'}(\vec a,\vec s)\,
e_{a_1}\otimes\cdots\otimes e_{a_k}\otimes e_{s_1}\otimes\cdots\otimes e_{s_l}.
\]
Hence $\bigl(id^{\otimes k}\otimes T_{(p',\vec t)}\bigr)(e_{r_1}\otimes\cdots\otimes e_{r_k})$
is equal to
\[
\sum_{\vec a,\vec s}
\delta_{p'}(\vec a,\vec s)\,
e_{r_1}\otimes\cdots\otimes e_{r_k}\otimes
e_{a_1}\otimes\cdots\otimes e_{a_k}\otimes
e_{s_1}\otimes\cdots\otimes e_{s_l}.
\]
Applying $T_{(\cap_k,\vec e)}\otimes id^{\otimes l}$, we obtain
\[
\sum_{\vec a,\vec s}
\delta_{p'}(\vec a,\vec s)\,
T_{(\cap_k,\vec e)}
(e_{r_1}\otimes\cdots\otimes e_{r_k}\otimes e_{a_1}\otimes\cdots\otimes e_{a_k})
\,
e_{s_1}\otimes\cdots\otimes e_{s_l}.
\]
Now, by the definition of $(\cap_k,\vec e)$, $T_{(\cap_k,\vec e)}
(e_{r_1}\otimes\cdots\otimes e_{r_k}\otimes e_{a_1}\otimes\cdots\otimes e_{a_k})
=
1$
if and only if $a_1=r_k^{-1},\quad a_2=r_{k-1}^{-1},\quad \dots,\quad a_k=r_1^{-1},$
and it is equal to $0$ otherwise. Therefore the above sum reduces to
\[
\sum_{\vec s\in\Lambda^l}
\delta_{p'}\bigl((r_k^{-1},\dots,r_1^{-1}),\vec s\bigr)\,
e_{s_1}\otimes\cdots\otimes e_{s_l}.
\]
Since $p'=Lrot^k(p)$, the defining relations for $p'$ are exactly those for $p$, with the
upper labeling $(r_1,\dots,r_k)$ replaced by the lower labeling
$(r_k^{-1},\dots,r_1^{-1})$ on the first $k$ points. Hence
\[
\delta_{p'}\bigl((r_k^{-1},\dots,r_1^{-1}),\vec s\bigr)=\delta_p(\vec r,\vec s),
\qquad \forall\,\vec s\in\Lambda^l.
\]
It follows that
\[
\bigl(T_{(\cap_k,\vec e)}\otimes id^{\otimes l}\bigr)
\circ
\bigl(id^{\otimes k}\otimes T_{(p',\vec t)}\bigr)
(e_{r_1}\otimes\cdots\otimes e_{r_k})
=
\sum_{\vec s\in\Lambda^l}\delta_p(\vec r,\vec s)\,
e_{s_1}\otimes\cdots\otimes e_{s_l},
\]
which is exactly $T_{(p,\vec t)}(e_{r_1}\otimes\cdots\otimes e_{r_k})$. This proves the claim. Now $T_{(p',\vec t)}$ is an intertwiner by the case of colored partitions with no upper points, since $(p',\vec t)\in NC_\Lambda\bigl((0,k+l),(g_k^{-1},\dots,g_1^{-1},h_1,\dots,h_l)\bigr).$ Moreover, $(\cup_k,\vec e)\in NC_\Lambda\bigl((0,2k),(g_1,\dots,g_k,g_k^{-1},\dots,g_1^{-1})\bigr),$
so again by the case of colored partitions with no upper points,
\[
T_{(\cup_k,\vec e)}\in
Mor\bigl(\epsilon,\,
u(g_1)\otimes\cdots\otimes u(g_k)\otimes
u(g_k^{-1})\otimes\cdots\otimes u(g_1^{-1})\bigr).
\]
Since $(\cap_k,\vec e)=(\cup_k,\vec e)^*$ and $T_{q^*}=T_q^*$, it follows that$
T_{(\cap_k,\vec e)}=T_{(\cup_k,\vec e)}^*$
is an intertwiner from$
u(g_1)\otimes\cdots\otimes u(g_k)\otimes
u(g_k^{-1})\otimes\cdots\otimes u(g_1^{-1})$ to $\epsilon$.

Therefore $id^{\otimes k}\otimes T_{(p',\vec t)}$ and $
T_{(\cap_k,\vec e)}\otimes id^{\otimes l}$
are both intertwiners, and since intertwiner spaces are stable under composition, the identity $
T_{(p,\vec t)}
=\bigl(T_{(\cap_k,\vec e)}\otimes id^{\otimes l}\bigr)
\circ
\bigl(id^{\otimes k}\otimes T_{(p',\vec t)}\bigr)$
shows that
\[
T_{(p,\vec t)}\in
Mor\bigl(u(g_1)\otimes\cdots\otimes u(g_k),\,u(h_1)\otimes\cdots\otimes u(h_l)\bigr).
\]

\end{proof}

\begin{theorem}\label{thm:C-Gamma-Lambda-category}
Let $\mathcal C_{\Gamma,\Lambda}$ be the concrete linear category defined as follows:
\begin{enumerate}
\item The objects are finite tuples of elements of $\Gamma$.
\item For $\vec g=(g_1,\dots,g_k)$ and $\vec h=(h_1,\dots,h_l)$,
\[
\Mor_{\mathcal C_{\Gamma,\Lambda}}(\vec g,\vec h)
:=
\mathrm{Span}\Bigl\{
T_{(p,\vec t)} \in \mathcal L\bigl(\ell^2(\Lambda)^{\otimes k},\,\ell^2(\Lambda)^{\otimes l}\bigr)
\ \Big|\
(p,\vec t)\in \NC_{\Lambda}\!\bigl((k,l),(\vec g,\vec h)\bigr)
\Bigr\}.
\]
\item The composition, tensor product, and involution of morphisms are the usual ones in the corresponding operator spaces.
\end{enumerate}
Then $\mathcal C_{\Gamma,\Lambda}$ is a rigid concrete $C^*$-tensor category.
\end{theorem}

\section{Reconstruction of $\GG=\widehat{\Gamma}\wr_{\ast,\beta}\Lambda$}
\begin{corollary}
The category $\mathcal C_{\Gamma,\Lambda}$ is a rigid concrete $C^*$-tensor category. Hence, by the Woronowicz--Tannaka--Krein reconstruction theorem, there exist a compact quantum group $\mathbb H$ and a unitary monoidal equivalence $\Rep(\mathbb H)\simeq \mathcal C_{\Gamma,\Lambda}.$
Moreover, Proposition~\ref{prop:WTK-concrete-universal} yields a surjective Hopf $*$-algebra morphism $\pi:\Pol(\mathbb H)\to \Pol(\GG)$
such that
\[
(\id\otimes\pi)(w(g))=u(g),\qquad g\in\Gamma.
\]
Consequently, $\pi$ extends to a surjective $*$-homomorphism
\[
\varphi:C(\mathbb H)\to C(\GG).
\]
\end{corollary}
\begin{proof}
By construction, every morphism space of $\mathcal C_{\Gamma,\Lambda}$ is a
linear subspace of an intertwiner space between tensor powers of the Hilbert
space $\ell^2(\Lambda)$. More precisely,
\[
\Mor\!\bigl((g_1,\dots,g_k),(h_1,\dots,h_l)\bigr)
\subset
\mathcal L\bigl(\ell^2(\Lambda)^{\otimes k},\ell^2(\Lambda)^{\otimes l}\bigr).
\]

The previous proposition shows that these morphism spaces are stable under
tensor product, composition and adjoint. Hence $\mathcal C_{\Gamma,\Lambda}$
is a $*$-tensor subcategory of the representation category of the free wreath product $\GG$.

Since the representation category $\Rep(\GG)$ is a rigid $C^*$-tensor category,
it follows that $\mathcal C_{\Gamma,\Lambda}$ is itself a rigid $C^*$-tensor category.

The functor $\omega$ is the obvious concrete unitary fiber functor, so the
Woronowicz--Tannaka--Krein reconstruction theorem yields a compact quantum group
$\mathbb H$ such that
\[
\Rep(\mathbb H)\simeq \mathcal C_{\Gamma,\Lambda}.
\]
For each $g\in\Gamma$, the object $g$ corresponds to a finite-dimensional unitary
representation $w(g)$ of $\mathbb H$ on $\ell^2(\Lambda)$, and therefore
\[
w(g)=\sum_{r,s\in\Lambda} e_{rs}\otimes w_{r,s}(g).
\]

Finally, the inclusion of all morphism spaces of $\mathcal C_{\Gamma,\Lambda}$
into the corresponding intertwiner spaces for $\GG$ allows us to apply
Proposition~\ref{prop:WTK-concrete-universal}. This gives the surjective Hopf
$*$-algebra morphism
\[
\pi:\Pol(\mathbb H)\to \Pol(\GG)
\]
such that
\[
(\id\otimes \pi)(w(g))=u(g),\qquad g\in\Gamma,
\]
and hence the induced surjective $*$-homomorphism
\[
\varphi:C(\mathbb H)\to C(\GG)
\]
with the same property.
\end{proof}

\begin{remark}
For each $g\in\Gamma$, let $
w(g)\in B(\ell^2(\Lambda))\otimes C(\mathbb H)$
be the representation corresponding to the generating object $g$ of $\mathcal C_{\Gamma,\Lambda}$.
Fixing the canonical orthonormal basis $(e_r)_{r\in\Lambda}$ of $\ell^2(\Lambda)$ and writing
$(e_{rs})_{r,s\in\Lambda}$ for the matrix units of $B(\ell^2(\Lambda))$, Write $w(g)=\sum_{r,s\in\Lambda} e_{rs}\otimes w_{r,s}(g).$
Then the surjective Hopf $*$-algebra morphism $\pi:\Pol(\mathbb H)\to \Pol(\GG)$ constructed above satisfies $\pi\bigl(w_{r,s}(g)\bigr)=u_{r,s}(g), g\in\Gamma,\ r,s\in\Lambda.$

\end{remark}

\begin{lemma}[Relations for $w(1)$]
\label{lem:relations-w1}
Write $w(1)=\sum_{r,s\in\Lambda} e_{rs}\otimes w_{r,s}(1)
\in B(\ell^2(\Lambda))\otimes C(\mathbb H).$
Assume that the following basic partition maps constructed above are intertwiners:
\[
T_{p_1}\in \Mor(\epsilon,w(1)),\qquad
T_{p(t)}\in \Mor(w(1)\otimes w(1),w(1))\quad(t\in\Lambda),
\qquad
T_P\in \Mor(w(1)\otimes w(1),w(1)),
\]
where $T_{p_1}(1)=e_1,$ $T_{p(t)}(e_r\otimes e_s)=\delta_{s,t^{-1}}\,e_{rt}, r,s,t\in\Lambda,$
and $T_P(e_r\otimes e_s)=e_{rs},r,s\in\Lambda.$
Then, together with the unitarity of $w(1)$, these intertwining relations force the following
relations in $C(\mathbb H)$:
\begin{enumerate}
\item[\textup{(1)}] For all $\gamma\in\Lambda$,
\[
w_{\gamma,1}(1)=\delta_{\gamma,1}\,1.
\]

\item[\textup{(2)}] For all $c,j\in\Lambda$,
\[
w_{c,j}(1)=\delta_{c,j}\,\bar j,
\qquad
\bar j:=w_{j,j}(1).
\]
In particular, each $\bar j$ is unitary.

\item[\textup{(3)}] For all $r,s\in\Lambda$,
\[
\overline{rs}=\bar r\,\bar s.
\]
Equivalently,
\[
\Lambda\to \mathcal U(C(\mathbb H)),
\qquad
r\mapsto \bar r,
\]
is a group homomorphism, with $\bar 1=1$.
\end{enumerate}
\end{lemma}

\begin{proof}
\emph{(1)} The intertwining relation $T_{p_1}\in\Mor(\epsilon,w(1))$ means
$w(1)(T_{p_1}\otimes 1)=T_{p_1}\otimes 1$. Since $T_{p_1}(1)=e_1$, we have
$w(1)(e_1\otimes 1)=e_1\otimes 1$, i.e.
\[
\sum_{\gamma\in\Lambda} e_\gamma\otimes w_{\gamma,1}(1)=e_1\otimes 1,
\]
which yields $w_{\gamma,1}(1)=\delta_{\gamma,1}\,1$.

\medskip
\emph{(2)} Fix $t\in\Lambda$. The condition $T_{p(t)}\in\Mor(w(1)\otimes w(1),w(1))$
is equivalent to
\[
(T_{p(t)}\otimes 1)(w(1)\otimes w(1))=w(1)(T_{p(t)}\otimes 1).
\]
Applying both sides to $e_r\otimes e_s\otimes 1$ and comparing coefficients gives, for all
$\gamma,r,s\in\Lambda$,
\[
w_{\,\gamma t^{-1},\,r}(1)\;w_{t^{-1},\,s}(1)=\delta_{s,t^{-1}}\;w_{\gamma,\,rt}(1).
\]
Since $w(1)$ is unitary. If $s\neq t^{-1}$, the right-hand side vanishes, hence
$w_{\gamma t^{-1},r}(1)\,w_{t^{-1},s}(1)=0$ for all $\gamma,r$. Left-multiplying by
$w_{\gamma t^{-1},r}(1)^*$ and summing over $\gamma$, the column unitarity of $w(1)$ yields
$w_{t^{-1},s}(1)=0$ for all $s\neq t^{-1}$. Since $t$ ranges over $\Lambda$, this shows that
each row has only its diagonal entry possibly nonzero, i.e. $w_{c,j}(1)=\delta_{c,j}\bar{j}$.
Row (or column) unitarity then implies $\bar{j}\in\mathcal U(C(\mathbb H))$.

\medskip
\emph{(3)} Finally, from $T_P\in\Mor(w(1)\otimes w(1),w(1))$ we get, by the same coefficient
comparison as in the standard $\NC(2,1)$ one-block case, that for all $\gamma,r,s\in\Lambda$,
\[
w_{\gamma,rs}(1)=\sum_{\substack{a,b\in\Lambda\\ab=\gamma}} w_{a,r}(1)\,w_{b,s}(1).
\]
Using diagonality $w_{a,r}(1)=\delta_{a,r}\bar{r}$ and $w_{b,s}(1)=\delta_{b,s}\bar{s}$, the sum reduces to
\[
w_{\gamma,rs}(1)=\delta_{\gamma,rs}\,\bar{r}\bar{s}.
\]
On the other hand, diagonality also gives $w_{\gamma,rs}(1)=\delta_{\gamma,rs}\,\bar{rs}$, hence
$\bar{rs}=\bar{r} \bar{s}$ for all $r,s\in\Lambda$. Taking $r=s=1$ and using (1) gives $\bar{1}=w_{1,1}(1)=1$.
\end{proof}

\begin{lemma}[Relations induced by $q(t)$ and $p_2$]
\label{lem:bar-nu-flip-no-shorthand}
Let $g\in\Gamma$ and let $t\in\Lambda$. We use the following two families of
partition intertwiners constructed above.

First, for each $t\in\Lambda$, let
\[
T_{q(t)}\in \Mor\bigl(w(1)\otimes w(g),\,w(g)\otimes w(1)\bigr)
\]
be the intertwiner associated with the partition $q(t)\in\NC(2,2)$ whose blocks are
$\{1,2,3\}$ and $\{4\}$, with $\Lambda$-colors $t$ and $t^{-1}$ and
$\Gamma$-colors $(1,g,g,1)$.
Explicitly,
\[
T_{q(t)}(e_r\otimes e_s)
=
e_{rst}\otimes e_{t^{-1}},
\qquad r,s\in\Lambda.
\]
Second, let
\[
T_{p_2}\in \Mor\bigl(w(1)\otimes w(g),\,w(g)\otimes w(1)\bigr)
\]
be the intertwiner associated with the one-block partition $p_2\in\NC(2,2)$ with
$\Lambda$-color $1$ and $\Gamma$-colors $(1,g,g,1)$. Its action is
\[
T_{p_2}(e_j\otimes e_d)
=
\sum_{\substack{i,k\in\Lambda\\ ik=jd}} e_i\otimes e_k.
\]

Then these intertwining relations imply the following identities in $C(\mathbb H)$.

\begin{enumerate}
\item[\textup{(1)}] For all $r,s,t\in\Lambda$,
\begin{equation}\label{eq:shift-no-shorthand}
w_{r,s}(g)\,\bar t = w_{rt,\,st}(g).
\end{equation}
In particular, taking $t=s^{-1}$ gives
\begin{equation}\label{eq:reduce-first-col}
w_{r,s}(g)\,\bar s^{-1} = w_{rs^{-1},\,1}(g).
\end{equation}

\item[\textup{(2)}] For all $r,s\in\Lambda$,
\begin{equation}\label{eq:flip-twist-no-shorthand}
\bar s^{-1}\,w_{r,s}(g)=w_{s^{-1},\,r^{-1}}(g)\,\bar r.
\end{equation}

\item[\textup{(3)}] Define, for each $x\in\Lambda$,
\begin{equation}\label{eq:def-bar-nu-no-shorthand}
\bar\nu_x(g):=w_{x,1}(g).
\end{equation}
Then, by \eqref{eq:reduce-first-col},
\[
\bar\nu_{rs^{-1}}(g)=w_{r,s}(g)\,\bar s^{-1},
\qquad r,s\in\Lambda.
\]
Combining this identity with \eqref{eq:flip-twist-no-shorthand} gives the
$\bar\nu$--flip relation
\begin{equation}\label{eq:bar-nu-flip-no-shorthand}
\bar s^{-1}\,\bar\nu_{rs^{-1}}(g)=\bar\nu_{s^{-1}r}(g)\,\bar s^{-1},
\qquad r,s\in\Lambda.
\end{equation}
\end{enumerate}
\end{lemma}

\begin{proof}
We first derive \eqref{eq:shift-no-shorthand}. Since
\[
T_{q(t)}\in \Mor\bigl(w(1)\otimes w(g),\,w(g)\otimes w(1)\bigr),
\]
we have
\begin{equation}\label{eq:q-intertw}
(T_{q(t)}\otimes 1)\,(w(1)\otimes w(g))
=
(w(g)\otimes w(1))\,(T_{q(t)}\otimes 1).
\end{equation}
Fix $r,s\in\Lambda$ and apply \eqref{eq:q-intertw} to
$e_r\otimes e_s\otimes 1$. Using the diagonality of $w(1)$,
\[
w(1)(e_r\otimes 1)=e_r\otimes \bar r,
\qquad
w(g)(e_s\otimes 1)=\sum_{b\in\Lambda} e_b\otimes w_{b,s}(g),
\]
we obtain
\[
(w(1)\otimes w(g))(e_r\otimes e_s\otimes 1)
=
\sum_{b\in\Lambda} e_r\otimes e_b\otimes \bar r\,w_{b,s}(g).
\]
Hence the left-hand side of \eqref{eq:q-intertw} is
\begin{equation}\label{eq:q-LHS}
\sum_{b\in\Lambda} T_{q(t)}(e_r\otimes e_b)\otimes \bar r\,w_{b,s}(g).
\end{equation}
By the construction of $T_{q(t)}$,
\[
T_{q(t)}(e_r\otimes e_b)=e_{rbt}\otimes e_{t^{-1}}.
\]
Thus
\[
(T_{q(t)}\otimes 1)(w(1)\otimes w(g))(e_r\otimes e_s\otimes 1)
=
\sum_{b\in\Lambda} e_{rbt}\otimes e_{t^{-1}}
\otimes \bar r\,w_{b,s}(g).
\]

On the other hand,
\[
(T_{q(t)}\otimes 1)(e_r\otimes e_s\otimes 1)
=
e_{rst}\otimes e_{t^{-1}}\otimes 1.
\]
Therefore the right-hand side of \eqref{eq:q-intertw} is
\[
(w(g)\otimes w(1))(e_{rst}\otimes e_{t^{-1}}\otimes 1)
=
\sum_{a\in\Lambda}
e_a\otimes e_{t^{-1}}\otimes w_{a,rst}(g)\,\bar t^{-1}.
\]
Comparing the coefficients of $e_a\otimes e_{t^{-1}}$ gives
\[
\bar r\,w_{r^{-1}at^{-1},s}(g)
=
w_{a,rst}(g)\,\bar t^{-1},
\qquad a,r,s,t\in\Lambda.
\]
Taking $r=1$ gives
\[
w_{at^{-1},s}(g)=w_{a,st}(g)\,\bar t^{-1}.
\]
Right-multiplying by $\bar t$ and replacing $a$ by $at$ yields
\[
w_{a,s}(g)\,\bar t=w_{at,st}(g),
\qquad a,s,t\in\Lambda,
\]
which is exactly \eqref{eq:shift-no-shorthand}. In particular, taking
$t=s^{-1}$ gives \eqref{eq:reduce-first-col}.

We next derive \eqref{eq:flip-twist-no-shorthand}. Since
\[
T_{p_2}\in \Mor\bigl(w(1)\otimes w(g),\,w(g)\otimes w(1)\bigr),
\]
we have
\begin{equation}\label{eq:p2-intertw}
(T_{p_2}\otimes 1)\,(w(1)\otimes w(g))
=
(w(g)\otimes w(1))\,(T_{p_2}\otimes 1).
\end{equation}
Fix $j,d\in\Lambda$ and apply \eqref{eq:p2-intertw} to
$e_j\otimes e_d\otimes 1$. Using
\[
w(1)(e_j\otimes 1)=e_j\otimes \bar j,
\qquad
w(g)(e_d\otimes 1)=\sum_{x\in\Lambda}e_x\otimes w_{x,d}(g),
\]
we get
\[
(w(1)\otimes w(g))(e_j\otimes e_d\otimes 1)
=
\sum_{x\in\Lambda} e_j\otimes e_x\otimes \bar j\,w_{x,d}(g).
\]
Since
\[
T_{p_2}(e_j\otimes e_x)
=
\sum_{\substack{i,k\in\Lambda\\ ik=jx}}e_i\otimes e_k,
\]
the coefficient of $e_i\otimes e_k$ on the left-hand side of
\eqref{eq:p2-intertw} is
\[
\bar j\,w_{j^{-1}ik,d}(g).
\]

On the other hand,
\[
(T_{p_2}\otimes 1)(e_j\otimes e_d\otimes 1)
=
\sum_{\substack{a,b\in\Lambda\\ ab=jd}} e_a\otimes e_b\otimes 1.
\]
Applying $w(g)\otimes w(1)$, the coefficient of $e_i\otimes e_k$ on the
right-hand side is
\[
w_{i,jdk^{-1}}(g)\,\bar k.
\]
Comparing coefficients gives
\[
\bar j\,w_{j^{-1}ik,d}(g)
=
w_{i,jdk^{-1}}(g)\,\bar k,
\qquad i,k,j,d\in\Lambda.
\]
Now set
\[
j=s^{-1},\qquad d=s,\qquad i=s^{-1},\qquad k=r.
\]
Then
\[
\bar s^{-1}\,w_{r,s}(g)
=
w_{s^{-1},r^{-1}}(g)\,\bar r,
\]
which is \eqref{eq:flip-twist-no-shorthand}.

Finally, define
\[
\bar\nu_x(g):=w_{x,1}(g),
\qquad x\in\Lambda,
\]
as in \eqref{eq:def-bar-nu-no-shorthand}. By \eqref{eq:reduce-first-col},
\[
w_{r,s}(g)\,\bar s^{-1}
=
\bar\nu_{rs^{-1}}(g).
\]
Starting from \eqref{eq:flip-twist-no-shorthand} and right-multiplying by
$\bar s^{-1}$ gives
\[
\bar s^{-1}\,w_{r,s}(g)\,\bar s^{-1}
=
w_{s^{-1},r^{-1}}(g)\,\bar r\,\bar s^{-1}.
\]
The left-hand side is
\[
\bar s^{-1}\bigl(w_{r,s}(g)\,\bar s^{-1}\bigr)
=
\bar s^{-1}\,\bar\nu_{rs^{-1}}(g).
\]
For the right-hand side, applying \eqref{eq:shift-no-shorthand} with
$(r,s,t)=(s^{-1},r^{-1},r)$ gives
\[
w_{s^{-1},r^{-1}}(g)\,\bar r
=
w_{s^{-1}r,1}(g)
=
\bar\nu_{s^{-1}r}(g).
\]
Therefore
\[
\bar s^{-1}\,\bar\nu_{rs^{-1}}(g)
=
\bar\nu_{s^{-1}r}(g)\,\bar s^{-1},
\]
which is exactly \eqref{eq:bar-nu-flip-no-shorthand}.
\end{proof}

\begin{lemma}[Multiplicativity relation]
Let $p_3\in\NC(2,1)$ be the two-level partition with upper points labelled $1,2$
(from left to right) and the unique lower point labelled $3$, consisting of a single block
$W=\{1,2,3\}$. Endow $W$ with $\Gamma$-colors
$\mathrm{col}_\Gamma(1)=g,\mathrm{col}_\Gamma(2)=h,\mathrm{col}_\Gamma(3)=gh,$
and with $\Lambda$-color $1$. Let $T_{p_3}:\ell^2(\Lambda)^{\otimes 2}\to \ell^2(\Lambda)$ be the
associated intertwiner. Then
\[
T_{p_3}\in \Mor\bigl(w(g)\otimes w(h),\, w(gh)\bigr)
\quad\Longrightarrow\quad
w_{\gamma,1}(gh)=\sum_{\substack{r,s\in\Lambda\\ rs=\gamma}}w_{r,1}(g)\,w_{s,1}(h),
\qquad \forall\,g,h\in\Gamma,\ \forall\,\gamma\in\Lambda.
\]
In particular, with the notation
\[
\bar{\nu}_\gamma(x):=w_{\gamma,1}(x)\qquad (x\in\Gamma,\ \gamma\in\Lambda),
\]
this reads
\[
\bar{\nu}_{\gamma}(gh)=\sum_{\substack{r,s\in\Lambda\\ rs=\gamma}}\bar{\nu}_{r}(g)\,\bar{\nu}_{s}(h),
\qquad \forall\,g,h\in\Gamma,\ \forall\,\gamma\in\Lambda.
\]
\end{lemma}
\begin{proof}
Recall that the map induced by the one-block partition $p_3\in\NC(2,1)$ is given on the
canonical basis by
\[
T_{p_3}(e_r\otimes e_s)=e_{rs},\qquad r,s\in\Lambda.
\]
The assumption $T_{p_3}\in\Mor(w(g)\otimes w(h),w(gh))$ means that
\begin{equation}\label{eq:Tp3-intertw}
(T_{p_3}\otimes 1)\,(w(g)\otimes w(h)) \;=\; w(gh)\,(T_{p_3}\otimes 1).
\end{equation}
Fix $r,s\in\Lambda$ and apply \eqref{eq:Tp3-intertw} to $e_r\otimes e_s\otimes 1$.

Writing
\[
w(x)=\sum_{a,b\in\Lambda} e_{ab}\otimes w_{ab}(x),\qquad x\in\Gamma,
\]
we have
\[
w(g)(e_r\otimes 1)=\sum_{a\in\Lambda} e_a\otimes w_{a r}(g),
\qquad
w(h)(e_s\otimes 1)=\sum_{b\in\Lambda} e_b\otimes w_{b s}(h).
\]
Hence
\[
(w(g)\otimes w(h))(e_r\otimes e_s\otimes 1)
=\sum_{a,b\in\Lambda} e_a\otimes e_b\otimes w_{a r}(g)\,w_{b s}(h),
\]
and therefore, using $T_{p_3}(e_a\otimes e_b)=e_{ab}$,
\[
(T_{p_3}\otimes 1)(w(g)\otimes w(h))(e_r\otimes e_s\otimes 1)
=\sum_{a,b\in\Lambda} e_{ab}\otimes w_{a r}(g)\,w_{b s}(h).
\]
On the other hand,
\[
w(gh)(T_{p_3}\otimes 1)(e_r\otimes e_s\otimes 1)
= w(gh)(e_{rs}\otimes 1)
=\sum_{\gamma\in\Lambda} e_\gamma\otimes w_{\gamma,\,rs}(gh).
\]
Comparing coefficients of $e_\gamma$ on both sides yields, for all $\gamma,r,s\in\Lambda$,
\[
w_{\gamma,\,rs}(gh)=\sum_{\substack{a,b\in\Lambda\\ab=\gamma}} w_{a r}(g)\,w_{b s}(h).
\]
Taking $r=s=1$ gives
\[
w_{\gamma,1}(gh)=\sum_{\substack{a,b\in\Lambda\\ab=\gamma}} w_{a,1}(g)\,w_{b,1}(h),
\]
which is exactly the claimed identity (with the relabeling $a\mapsto r$, $b\mapsto s$).
\end{proof}

 \begin{lemma}[Cup relation and unitarity]\label{lem:cup-unitary-consequence}
 
Let $\xi$ be the two-point one-block partition with $\Lambda$-color $1$. The associated
intertwiner $T_\xi\in \Mor\bigl(\epsilon,\ w(g)\otimes w(g^{-1})\bigr)$ is given by
\[
T_\xi(1)=\xi=\sum_{r\in\Lambda} e_r\otimes e_{r^{-1}}
\ \in\ \ell^2(\Lambda)\otimes\ell^2(\Lambda).
\]
Then, for all $c,s\in\Lambda$,
\begin{equation}\label{eq:main}
w_{c,s}(g^{-1}) \;=\; w_{c^{-1},\,s^{-1}}(g)^{*}.
\end{equation}
\end{lemma}

\begin{proof}
Since $T_\xi\in \Mor\bigl(\epsilon,\ w(g)\otimes w(g^{-1})\bigr)$, we have the intertwining relation
\begin{equation}\label{eq:intertwine-xi}
(w(g)\otimes w(g^{-1}))(\xi\otimes 1)=\xi\otimes 1.
\end{equation}

\medskip
\noindent\textbf{Step 1: Coefficient form of \eqref{eq:intertwine-xi}.}
Expanding the left-hand side of \eqref{eq:intertwine-xi} we get
\[
(w(g)\otimes w(g^{-1}))(\xi\otimes 1)
=\sum_{r\in\Lambda}(w(g)\otimes w(g^{-1}))(e_r\otimes e_{r^{-1}}\otimes 1).
\]
Using
$w(g)(e_r\otimes 1)=\sum_{a\in\Lambda} e_a\otimes w_{a r}(g),
w(g^{-1})(e_{r^{-1}}\otimes 1)=\sum_{c\in\Lambda} e_c\otimes w_{c,\,r^{-1}}(g^{-1}),$
we obtain
\[
(w(g)\otimes w(g^{-1}))(\xi\otimes 1)
=\sum_{r\in\Lambda}\ \sum_{a,c\in\Lambda}
e_a\otimes e_c\otimes w_{a r}(g)\,w_{c,\,r^{-1}}(g^{-1}).
\]
On the other hand,
\[
\xi\otimes 1=\sum_{r\in\Lambda} e_r\otimes e_{r^{-1}}\otimes 1
=\sum_{a,c\in\Lambda} e_a\otimes e_c\otimes \delta_{c,\,a^{-1}}\,1,
\].
Comparing coefficients of the basis vectors $e_a\otimes e_c$ in
$\ell^2(\Lambda)\otimes\ell^2(\Lambda)$ yields
\begin{equation}\label{eq:cup-relation}
\sum_{r\in\Lambda} w_{a r}(g)\,w_{c,\,r^{-1}}(g^{-1})=\delta_{c,\,a^{-1}}\,1,
\qquad \forall\,a,c\in\Lambda.
\end{equation}

\medskip
\noindent\textbf{Step 2: Unitarity and the $*$-formula.}
Since $w(g)$ is unitary, its coefficients satisfy, for all $r,s\in\Lambda$,
\begin{equation}\label{eq:unitary-column}
\sum_{a\in\Lambda} w_{a r}(g)^*\,w_{a s}(g)=\delta_{r,s}\,1.
\end{equation}
Fix $c\in\Lambda$ and $r\in\Lambda$. Starting from \eqref{eq:cup-relation}, multiply on the left
by $w_{a r}(g)^*$ and sum over $a\in\Lambda$:
\[
\sum_{a\in\Lambda} w_{a r}(g)^*
\Bigl(\sum_{\rho\in\Lambda} w_{a \rho}(g)\,w_{c,\rho^{-1}}(g^{-1})\Bigr)
=
\sum_{a\in\Lambda} w_{a r}(g)^*\,\delta_{c,a^{-1}}\,1.
\]
On the left-hand side, interchange the sums and use \eqref{eq:unitary-column}:
\[
\sum_{\rho\in\Lambda}
\Bigl(\sum_{a\in\Lambda} w_{a r}(g)^*\,w_{a \rho}(g)\Bigr)\,w_{c,\rho^{-1}}(g^{-1})
=
\sum_{\rho\in\Lambda}\delta_{r,\rho}\,w_{c,\rho^{-1}}(g^{-1})
=
w_{c,r^{-1}}(g^{-1}).
\]
On the right-hand side, the Kronecker delta forces $a=c^{-1}$, hence
\[
\sum_{a\in\Lambda} w_{a r}(g)^*\,\delta_{c,a^{-1}}\,1
=
w_{c^{-1},r}(g)^*.
\]
Therefore, for all $c,r\in\Lambda$,
\[
w_{c,r^{-1}}(g^{-1})=w_{c^{-1},r}(g)^*.
\]
Replacing $r^{-1}$ by $s$ yields \eqref{eq:main}.
\end{proof}

 \begin{theorem}\label{Reconstruction}
The compact quantum groups $\GG$ and $\mathbb H$ are isomorphic. More precisely, there exists an isomorphism of compact quantum groups
\[
\psi:C(\GG)\xrightarrow{\ \cong\ } C(\mathbb H)
\]
such that
\[
\psi(s)=\bar{s},\qquad \forall\, s\in\Lambda,
\]
and
\[
\psi\bigl(\nu_\gamma(g)\bigr)=\bar{\nu}_\gamma(g),
\qquad \forall\, \gamma\in\Lambda,\ g\in\Gamma.
\]
Equivalently,
\[
\GG\simeq \mathbb H.
\]
\end{theorem}
\begin{proof}
By construction, the elements $\bar{s}$ ($s\in\Lambda$) and
$\bar{\nu}_\gamma(g)$ ($\gamma\in\Lambda$, $g\in\Gamma$) generate
$C(\mathbb H)$, and they satisfy the same defining relations as the
generators $s$ and $\nu_\gamma(g)$ of $C(\GG)$.
Hence the universal property of $C(\GG)$ yields a surjective
$*$-homomorphism
\[
\psi:C(\GG)\to C(\mathbb H)
\]
satisfying the above formulas.

Recall that we already have a surjective $*$-homomorphism
\[
\varphi:C(\mathbb H)\to C(\GG)
\]
such that
\[
(\id\otimes\varphi)(w(g))=u(g),\qquad g\in\Gamma .
\]
By construction, the compositions $\varphi\circ\psi$ and
$\psi\circ\varphi$ act as the identity on the generators of
$C(\GG)$ and $C(\mathbb H)$ respectively. Therefore
\[
\varphi\circ\psi=\id_{C(\GG)},\qquad
\psi\circ\varphi=\id_{C(\mathbb H)}.
\]
It follows that $\varphi$ and $\psi$ are mutually inverse
$*$-isomorphisms. Consequently,
\[
C(\mathbb H)\cong C(\GG),
\]
and hence the compact quantum groups $\mathbb H$ and $\GG$
are isomorphic.
\end{proof}

 \begin{remark}
An interesting special case of the reconstruction theorem occurs when
\[
\Gamma=\{e\}.
\]
In this situation, the category built from $\Lambda$-coloured noncrossing partitions satisfying the boundary condition reconstructs the compact quantum group $\widehat{\Lambda}$. Since a finite group is completely determined by its dual compact quantum group, it follows that every finite group $\Lambda$ can be recovered from this enhanced noncrossing-partition formalism.

We emphasize that this should not be interpreted in the usual ``easy quantum group'' sense. Indeed, the group structure of $\Lambda$ is not encoded by the bare noncrossing partition itself, but by the additional block labelling together with the boundary product constraint. In other words, the planar noncrossing skeleton is preserved, while the full multiplication law of $\Lambda$ is carried by the decoration data.
\end{remark}

\nocite{*}
\bibliographystyle{alpha}
\bibliography{BIBLIO}

\end{document}